\documentclass[11pt,reqno]{amsart}
\usepackage[width=150mm,height=220mm,
centering,
]{geometry}
\usepackage{mathrsfs}
\usepackage{amsxtra}
\usepackage{amsmath,amsfonts,amssymb}
\usepackage{upgreek}
\usepackage{xcolor}
\usepackage[numbers,sort&compress]{natbib}
\usepackage{csquotes}
\usepackage{url}
\numberwithin{equation}{section}
\allowdisplaybreaks
\ifcsname proof\endcsname
  \else

\fi

\newtheorem{theorem}{Theorem}[section]
\newtheorem{proposition}[theorem]{Proposition}
\newtheorem{lemma}[theorem]{Lemma}

\newtheorem{definition}[theorem]{Definition}
\newtheorem{remark}[theorem]{Remark}
\newtheorem{corollary}[theorem]{Corollary}

\def\rmd{\mathrm d}
\def\supp{\operatorname{supp}}
\begin{document}

\parindent 2em
\title[Local Smoothing Estimates of a Class of Schr\"{o}dinger Operators]
{Local Smoothing Estimates of a Class of Schr\"{o}dinger Operators on Waveguides Using Refined Decoupling}

\author[W. Sun]{Wenchang Sun}
\address{
School of Mathematical Sciences and LPMC\\
Nankai University\\
Tianjin\\
China}

\email{sunwch@nankai.edu.cn}

\author[J. Tan]{Jiheng Tan}
\address{
School of Mathematical Sciences and LPMC\\
Nankai University\\
Tianjin\\
China}

\email{tanjh@mail.nankai.edu.cn}

\thanks{Corresponding author: Wenchang Sun.}

\subjclass[2020]{35B65, 35Q41, 42B37.}
\keywords{Schr\"{o}dinger operator, Local Smoothing Estimates, Waveguide, Decoupling}

\begin{abstract}
We prove local smoothing estimates for a class of Schr\"odinger operators on waveguides.
Our results extend the previously known range of exponents in every dimension and are sharp in certain cases.
The proof develops a transference argument that preserves the discrete periodic frequency structure,
together with a refined decoupling estimate adapted to waveguides.
We also derive consequences for Euclidean local smoothing and almost everywhere convergence.
\end{abstract}

\thanks{This work was supported by the
National Natural Science Foundation of China (No. 12571104 and 12271267).}

\maketitle
\date{\today}

\section{Introduction}\label{sec:introduction}
Waveguide spaces $\mathbb T^m\times\mathbb R^n$ combine periodic and Euclidean directions and arise naturally in the study of 
dispersive equations, particularly Schr\"odinger equations. Well-posedness and scattering for nonlinear Schr\"odinger equations 
on these spaces have been developed in \cite{IP,HP,TV,CGYZ,CGZ,YYZ21}. 
The interaction between the periodic and Euclidean directions is also reflected in modified scattering and long-time dynamics \cite{HPTV,ZZ}. 
On the linear side, global-in-time Strichartz estimates \cite{B,BCP}, 
and bilinear Strichartz estimates \cite{DFYYZ} have been established.

Let $m,n\geq1$ and $\mathbb T^m=\mathbb R^m/\mathbb Z^m$. Denote by $\Delta_{\mathrm{LB}}$ the Laplace--Beltrami operator
on $\mathbb T^m\times\mathbb R^n$, and write $\Delta_{\mathbb T^m\times\mathbb R^n}=\Delta_{\mathrm{LB}}/(2\pi)$.
We consider the Schr\"odinger equation
\begin{equation*}
\begin{cases}
2\pi i\partial_tu(x,y,t)-\Delta_{\mathrm{LB}}u(x,y,t)=0,\\
u(x,y,0)=f(x,y).
\end{cases}
\end{equation*}
Its solution is given by
\begin{equation*}
\begin{aligned}
u(x,y,t)
&=e^{-it\Delta_{\mathbb T^m\times\mathbb R^n}}f(x,y)\\
&=\sum_{k\in\mathbb Z^m}\int_{\mathbb R^n}
\widehat f(k,\xi)
e\left(k\cdot x+\xi\cdot y+t(|k|^2+|\xi|^2)\right)
\rmd\xi.
\end{aligned}
\end{equation*}
We are interested in the local smoothing estimate
\begin{equation}\label{eq:conj}
\|e^{-it\Delta_{\mathbb T^m\times\mathbb R^n}}f\|_{L^p(\mathbb T^m\times\mathbb R^n\times[0,1])}
\lesssim
\|f\|_{L_s^p(\mathbb T^m\times\mathbb R^n)},
\end{equation}
where $L_s^p$ denotes the Bessel potential space.

Chen, Guo, Shen and Yan \cite{CGSY} proved that \eqref{eq:conj} holds for
\begin{equation*}
p\geq2+\frac4{m+n},
\qquad
s>(m+2n)\left(\frac12-\frac1p\right)-\frac2p.
\end{equation*}
They also showed that \eqref{eq:conj} fails for
\begin{equation*}
0\leq s<(m+2n)\left(\frac12-\frac1p\right)-\frac2p.
\end{equation*}
The lower bound on $s$ is positive precisely when
$p>2+4/(m+2n)$.
We improve the range of $p$ while retaining the sharp Sobolev regularity. 

Our main result applies to the following class of phases.

\begin{theorem}\label{thm:main}
Let $m,n\geq 1$ and $a_0>1$. Suppose that $\Phi$ is real-valued, belongs to $C^{1}(\mathbb{R}^{m+n})\cap C^{\infty}(\mathbb{R}^{m+n}\setminus \{0\})$,
and satisfies $\nabla^{2}\Phi(\eta,\xi)\sim |(\eta,\xi)|^{a_0-2}I$ and $|D^{\beta}\Phi(\eta,\xi)|\lesssim |(\eta,\xi)|^{a_0-|\beta|}$ for $|\beta|\geq 3$ and $(\eta,\xi)\neq 0$.
If $1<a_0\leq 2$, let
\begin{equation*}
p_0=
\begin{cases}
\displaystyle 2+\frac4{m+2},&n=1,\\[8pt]
\displaystyle 2+\frac8{2m+7},&n=2,\\[8pt]
\displaystyle 2+\frac{4(7n+36)}{(7n+36)m+11n^2+54n},&n\geq3.
\end{cases}
\end{equation*}
If $a_0>2$, let
\begin{equation*}
p_0=\begin{cases}
\displaystyle 2+\frac{4a_0}{2a_0+\max\{0,4-a_0\}\min\{m,a_0/(2(a_0-2))\}}, &n=1,\\[12pt]
\displaystyle 2+\frac{8a_0}{7a_0+\max\{0,6-a_0\}\min\{m,a_0/(a_0-2)\}}, &n=2,\\[12pt]
\displaystyle 2+\frac{4a_0(7n+36)}{a_0n(11n+54)+\max\{0,16n+72-a_0n\}\min\{m,a_0n/(2(a_0-2))\}}, &n\geq3.
\end{cases}
\end{equation*} 
For every $p_0\leq p<\infty$ and
\begin{equation*}
s>(m+a_0 n)\left(\frac12-\frac1p\right)-\frac{a_0}{p},
\end{equation*}
we have
\begin{equation}\label{eq:thm:main}
\|U_\Phi(t)f\|_{L^{p}(\mathbb{T}^{m}\times \mathbb{R}^{n}\times [0,1])}
\lesssim \|f\|_{L_s^{p}(\mathbb{T}^{m}\times \mathbb{R}^{n})},
\end{equation}
where
\begin{equation*}
U_\Phi(t)f(x,y)
=\sum_{k\in \mathbb{Z}^{m}}\int_{\mathbb{R}^{n}}\widehat{f}(k,\xi)e\left(k\cdot x+\xi\cdot y+t\Phi(k,\xi)\right)\rmd \xi.
\end{equation*}
\end{theorem}

\begin{theorem}\label{thm:sharpness}
Let $m,n\geq 1$, $a_0>1$ and $2\leq p<\infty$. For every phase $\Phi$ satisfying the assumptions of Theorem \ref{thm:main}, 
\eqref{eq:thm:main} can hold only if $s\geq 0$, and fails whenever
\begin{equation}\label{eq:regularity1}
s<(m+a_0 n)\left(\frac{1}{2}-\frac{1}{p}\right)-\frac{a_0}{p}, 
\end{equation}
If $1<a_0\leq 2$, it also fails whenever
\begin{equation}\label{eq:regularity2}
s<(2-a_0)\left((m+n)\left(\frac{1}{2}-\frac{1}{p}\right)-\frac{1}{p}\right).
\end{equation}
Consequently, for any fixed such phase, if \eqref{eq:thm:main} holds for every $s$ strictly above the index in \eqref{eq:regularity1}, then necessarily
\begin{equation}\label{eq:necessary-p}
p\geq
\begin{cases}
\displaystyle 2+\frac{4}{m+2n},&1<a_0\leq2,\\[8pt]
\displaystyle 2+\frac{2a_0}{m+a_0n},&a_0>2.
\end{cases}
\end{equation}
\end{theorem}

In particular, when $n=1$ and $1<a_0\leq 2$, the range of $p$ in Theorem \ref{thm:main} is sharp. For the quadratic phase $\Phi(\eta,\xi)=|(\eta,\xi)|^{2}$, we have
\begin{equation*}
p_0<2+\frac4{m+n}
\end{equation*}
for all $m,n\geq1$, Theorem \ref{thm:main} improves the result of  \cite{CGSY} in every dimension.

The Euclidean Schr\"odinger local smoothing problem is closely related to Fourier restriction estimates. 
Rogers \cite{R} showed that the restriction conjecture for the paraboloid implies the Schr\"odinger local smoothing conjecture through a pseudo-conformal transformation. 
This argument does not directly extend to $\mathbb T^m\times\mathbb R^n$, since the periodic Fourier frequencies are 
discrete and the pseudo-conformal change of variables does not preserve periodicity. 
Nevertheless, the common quadratic phase suggests that restriction methods may be adapted to the waveguide problem, 
provided that the discrete frequency structure is taken into account.

The proof of Theorem \ref{thm:main} also applies when $m=0$. 
In this case, unlike on waveguide, the two ranges $1<a_0\leq 2$ and $a_0>2$ lead to the same lower bound on $p$,
independent of $a_0$. We obtain the following Euclidean estimate.

\begin{corollary}\label{cor:Euclidean}
Let $n\geq 1$ and $a_0>1$. Suppose that $\Phi$ is real-valued, belongs to $C^{1}(\mathbb{R}^{n})\cap C^{\infty}(\mathbb{R}^{n}\setminus \{0\})$,
and satisfies $\nabla^{2}\Phi(\xi)\sim |\xi|^{a_0-2}I$ and $|D^{\beta}\Phi(\xi)|\lesssim |\xi|^{a_0-|\beta|}$ for $|\beta|\geq 3$ and $\xi\neq 0$.
Let 
\begin{equation*}
p_E=
\begin{cases}
\displaystyle 4,&n=1,\\[4pt]
\displaystyle\frac{22}{7},&n=2,\\[4pt]
\displaystyle 2+\frac{4(7n+36)}{11n^2+54n},&n\geq3.
\end{cases}
\end{equation*}
For every $p_E\leq p<\infty$ and
\begin{equation*}
s>a_0 n\left(\frac12-\frac1p\right)-\frac{a_0}{p},
\end{equation*}
we have
\begin{equation}\label{eq:Euclidean}
\|U_{\Phi}(t)g\|_{L^{p}(\mathbb{R}^{n}\times [0,1])}
\lesssim \|g\|_{L_s^{p}(\mathbb{R}^{n})},
\end{equation}
where $U_{\Phi}(t)g(y)=\int_{\mathbb{R}^{n}}\widehat{g}(\xi)e(y\cdot \xi+t\Phi(\xi))\rmd \xi.$
\end{corollary}

Taking $\Phi(\xi)=|\xi|^{a_0}$ gives the fractional Schr\"odinger estimate, with the same range $p\geq p_E$. For $a_0=2$ and $\Phi(\xi)=|\xi|^{2}$, 
the Sobolev regularity in \eqref{eq:Euclidean} agrees with that in the Euclidean Schr\"odinger local smoothing conjecture. 
Rogers and Seeger \cite{RS} established Euclidean Schr\"odinger local smoothing estimates for $p>2+4/(n+1)$. 
Further space-time bounds were obtained by Lee, Rogers and Seeger \cite{LRS}, 
while Guo, Roos and Yung \cite{GRY} improved the range of $p$ for local smoothing in higher dimensions. 

For $n=2$, Corollary \ref{cor:Euclidean} gives $p\geq 22/7$. The same exponent appears in the three-dimensional Fourier restriction estimate of 
Wang and Wu \cite{WW} for hypersurfaces with positive definite second fundamental form. 
For $a_0>1$, Corollary \ref{cor:Euclidean} applies to a broader class of phases than the fractional Schr\"odinger results in \cite{GMZ,GOW}
while improving the range of $p$ in certain dimensions.

Schr\"odinger local smoothing estimates are closely related to maximal estimates and pointwise convergence, as studied in \cite{R,RS}. 
The pointwise convergence problem was initiated by Carleson \cite{C}, and the sharp regularity for $H^s(\mathbb R^n)$ initial data 
in dimensions $n\geq2$ was established, up to the endpoints, in \cite{DGL,DZ}. 
Related results for $L_s^p(\mathbb R^n)$ initial data can be found in \cite{LRS,PST}. 
The argument in \cite[Section 4]{RS} applied to \eqref{eq:Euclidean} gives the following fractional Schr\"odinger maximal estimate.

\begin{corollary}\label{cor:pointwise}
Let $n\geq1$, $a_0>1$ and let $p_E$ be as in Corollary \ref{cor:Euclidean}. Suppose that
\begin{equation*}
p_E\leq p<\infty,
\qquad s>a_0n\left(\frac12-\frac1p\right).
\end{equation*}
Then, for $\Phi(\xi)=|\xi|^{a_0}$, 
\begin{equation*}
\left\|\sup_{0<t<1}\left|U_{\Phi}(t)g\right|\right\|_{L^p(\mathbb R^n)}
\lesssim \|g\|_{L_s^p(\mathbb R^n)}.
\end{equation*}
Consequently, for every $g\in L_s^p(\mathbb R^n)$,
\begin{equation*}
\lim_{t\to 0^+}U_{\Phi}(t)g(y)=g(y)
\end{equation*}
for almost every $y\in\mathbb R^n$.
\end{corollary}

Our proof builds on the method of Wang and Wu \cite{WW},
which combines refined decoupling with two-ends Furstenberg inequalities.
To extend this method to waveguides, we develop a transference argument that retains the discrete periodic frequency structure throughout
the proof. We also establish a refined decoupling estimate adapted to waveguides,
based on the theorem of Guth, Iosevich, Ou and Wang \cite{GIOW}.
Combining these estimates with the geometric inequalities of \cite{WW}, we prove Theorem \ref{thm:main}.

The paper is organized as follows. In Section \ref{sec:preliminaries}, we collect preliminary estimates.
Section \ref{sec:reductions} reduces the local smoothing problem to a Euclidean estimate through Littlewood--Paley decomposition,
interpolation and a continuous extension of the periodic frequency data. In Section \ref{sec:wave-packets},
we introduce the wave packet decomposition and establish the corresponding $L^2$ estimates.
Section \ref{sec:refined-decoupling} is devoted to the refined decoupling estimate. In Section \ref{sec:local-estimate},
we combine refined decoupling with geometric estimates and induction on scales to complete the proof of Theorem \ref{thm:main}.
Section \ref{sec:sharpness} proves the necessary conditions in Theorem~\ref{thm:sharpness}.

\textbf{Notation.}
Throughout this paper, for nonnegative quantities $A$ and $B$, we write $A\lesssim B$ if $A\leq CB$ for some positive constant $C$
independent of the relevant variables.
We write $A\gtrsim B$ if $B\lesssim A$, and $A\sim B$ if both $A\lesssim B$ and $A\gtrsim B$ hold.
We write $A\ll B$ or $B\gg A$ if $A\leq cB$ for a sufficiently small absolute constant $c>0$.

We use $B_r(z)$ and $B(z,r)$ interchangeably to denote the ball of radius $r$ centered at $z$.
When the dimension is specified, we write $B_r^j(z)$ or $B^j(z,r)$ for balls in $\mathbb R^j$.
An $L^p$ norm without an indicated domain is taken over the full domain of the function.

We set $e(z)=e^{2\pi iz}$.
For a function $h$ on $\mathbb R^n$, its Fourier transform $\widehat h$ and inverse Fourier transform $h^\vee$ are defined by
\begin{equation*}
\widehat h(\xi)
=\int_{\mathbb R^n}h(y)e(-y\cdot\xi)\rmd y,
\qquad
h^\vee(y)
=\int_{\mathbb R^n}h(\xi)e(y\cdot\xi)\rmd\xi.
\end{equation*}
For a function $h$ on $\mathbb T_L^m\times\mathbb R^n$, where $\mathbb T_L^m=\mathbb R^m/(L\mathbb Z^m)$, we define
\begin{equation*}
\begin{aligned}
\widehat h(k,\xi)
&=\frac1{L^m}\int_{\mathbb T_L^m}\int_{\mathbb R^n}
h(x,y)e(-k\cdot x-\xi\cdot y)\rmd y\rmd x,\\
h(x,y)
&=\sum_{k\in L^{-1}\mathbb Z^m}
\int_{\mathbb R^n}
\widehat h(k,\xi)e(k\cdot x+\xi\cdot y)\rmd\xi.
\end{aligned}
\end{equation*}
We use $\widehat h$ and $h^\vee$ to denote the Fourier transform and inverse Fourier transform,
respectively, regardless of the underlying space.

\section{Preliminaries}\label{sec:preliminaries}

We recall the oscillatory integral, interpolation and incidence estimates used below.

\subsection{Oscillatory Integral Estimates}
We use two classical results from Stein \cite{St}, namely nonstationary phase and Fourier decay for  hypersurfaces with nonvanishing Gaussian curvature.

Consider an oscillatory integral of the form
\begin{equation*}
\int_{\mathbb{R}^{d}}e^{i\lambda \phi(x)}\psi(x)\rmd x,
\end{equation*}
we say that a phase function $\phi$ defined in a neighborhood of a point $x_0$ has $x_0$ as a critical point if
\begin{equation*}
(\nabla \phi)(x_0)=\left. \left(\frac{\partial \phi}{\partial x_1}, \cdots, \frac{\partial \phi}{\partial x_d}\right)\right|_{x=x_0}=0.
\end{equation*}

\begin{proposition}[{\cite[Proposition 4, Page 341]{St}}]\label{prop:noncritical}
Suppose $\psi$ is smooth, has compact support, and that $\phi$ is a smooth real-valued function that has no critical points in the support of $\psi$. Then
\begin{equation*}
I(\lambda)=\int_{\mathbb{R}^{d}}e^{i\lambda \phi(x)}\psi(x)\rmd x=O(\frac{1}{\lambda^{N}}),
\end{equation*}
as $\lambda\to \infty$, for every $N\geq 0$.
\end{proposition}

Assume that a smooth hypersurface $S\subset \mathbb{R}^{d}$ is locally given by the graph $x_d=\phi(x_1, \cdots, x_{d-1})$.
Its Gaussian curvature at $(x', \phi(x'))$ is given by
\begin{equation*}
G_S(x')=\frac{\det \nabla^{2}\phi(x')}{(1+|\nabla \phi(x')|^{2})^{(d+1)/2}}.
\end{equation*}
Let $\rmd \sigma$ denote the surface measure on $S$, and let $\rmd \mu=\phi\rmd \sigma$, where $\phi\in C_c^{\infty}(S)$.

\begin{proposition}[{\cite[Theorem 1, Page 348]{St}}]\label{prop:surface}
Suppose $S$ is a smooth hypersurface in $\mathbb{R}^{d}$, whose Gaussian curvature is nonzero everywhere, and let $\rmd \mu=\phi\rmd \sigma$ be as above.
Then
\begin{equation*}
\left|\widehat{\rmd \mu}(\xi)\right|\lesssim |\xi|^{(1-d)/2}.
\end{equation*}
\end{proposition}

Combining Propositions \ref{prop:noncritical} and \ref{prop:surface}, we obtain the following kernel bounds.
\begin{corollary}\label{cor:kernel}
Let $d\geq 1$, $\chi\in C_c^{\infty}(\mathbb{R}^{d})$ and real-valued $\Psi\in C^{\infty}(\mathbb{R}^{d})$. For $r\geq 1$ and $\omega_0$ in a 
fixed compact set, suppose that
\begin{equation*}
\left|\det \nabla^{2}\Psi\left(\omega_0+\frac{v}{r^{1/2}}\right)\right|\geq c>0, \qquad v\in \supp \chi.
\end{equation*}
Set
\begin{equation*}
K_{t,r,\omega_0}(z)=\int_{\mathbb{R}^{d}}\chi\left(r^{1/2}(\omega-\omega_0)\right)e(z\cdot \omega+t\Psi(\omega))\rmd \omega.
\end{equation*}
For every integer $N\geq 0$ and $t\in \mathbb{R}$, we have
\begin{equation}\label{eq:kernel-cap}
|K_{t,r,\omega_0}(z)|
\lesssim \frac{1}{(r+|t|)^{d/2}\left(1+r^{1/2}|z+t\nabla\Psi(\omega_0)|/(r+|t|)\right)^N},
\end{equation}
and 
\begin{equation}\label{eq:kernel-time}
\|K_{t,r,\omega_0}\|_{L^{1}}\lesssim \left(1+\frac{|t|}{r}\right)^{d/2}.
\end{equation}
Moreover, for sufficiently large $C$,
\begin{equation}\label{eq:kernel-noncritical}
|K_{t,r,\omega_0}(z)|\lesssim \frac{1}{r^{d/2}(1+|z+t\nabla\Psi(\omega_0)|/r^{1/2})^N}, \qquad |z+t\nabla\Psi(\omega_0)|\geq C(r^{1/2}+|t|/r^{1/2}).
\end{equation}
\end{corollary}

\begin{proof}
Set $v = r^{1/2}(\omega - \omega_0)$, $w=(z+t\nabla\Psi(\omega_0))/r^{1/2}$, $s=t/r$, and
\begin{equation*}
q_r(v) = r\left(\Psi\left(\omega_0 + \frac{v}{r^{1/2}}\right) - \Psi(\omega_0) - \frac{1}{r^{1/2}}\nabla\Psi(\omega_0)\cdot v\right),
\end{equation*}
we obtain
\begin{equation}\label{eq:K_J_rel}
K_{t,r,\omega_0}(z) = \frac{e(z\cdot\omega_0 + t\Psi(\omega_0))}{r^{d/2}} J(w,s),
\end{equation}
where
\begin{equation*}
J(w,s) = \int_{\mathbb{R}^d} \chi(v) e\left(w\cdot v + s q_r(v)\right) dv.
\end{equation*}
Evidently, $q_r(0) = 0$, $\nabla q_r(0) = 0$, and $\nabla^2 q_r(v) = \nabla^2 \Psi(\omega_0 + v/r^{1/2})$. 
Since $\chi \in C_c^\infty(\mathbb{R}^d)$ and $\omega_0$ lies in a fixed compact set, we may fix $A, C_0 \ge 1$ 
such that $\operatorname{supp}\chi \subset B(0,A)$, $|\nabla q_r(v)| \le C_0$ for all $v \in \operatorname{supp}\chi$.

We first prove \eqref{eq:kernel-noncritical}. Put $\lambda = |w|$ and
\begin{equation*}
\phi(v) = \frac{w}{|w|}\cdot v + \frac{s}{\lambda} q_r(v).
\end{equation*}
If $\lambda \ge C(1 + |s|)$ with $C \ge 2C_0$, note that
\begin{equation*}
|\nabla\phi(v)|=\left|\frac{w}{|w|}+\frac{s}{\lambda}\nabla q_r(v) \right| \ge 1 - \frac{|s||\nabla q_r(v)|}{\lambda} \ge \frac{1}{2}, \qquad v \in \operatorname{supp}\chi.
\end{equation*}
Applying Proposition \ref{prop:noncritical} yields
\begin{equation*}
|J(w,s)| = \left|\int_{\mathbb{R}^d} \chi(v) e(\lambda\phi(v))\,dv\right| \lesssim \frac{1}{\lambda^N}.
\end{equation*}
Combining this with the trivial bound $|J(w,s)| \le \|\chi\|_{L^1}$ for $\lambda$ bounded, we obtain
\begin{equation}\label{eq:J_nonstat}
|J(w,s)| \lesssim \frac{1}{(1 + |w|)^N}, \qquad |w| \ge C(1 + |s|),
\end{equation}
for every $N \ge 0$. In view of \eqref{eq:K_J_rel}, this proves \eqref{eq:kernel-noncritical} whenever $|z + t\nabla\Psi(\omega_0)| \ge C(r^{1/2} + |t|/r^{1/2})$.

For \eqref{eq:kernel-cap}, we use Proposition \ref{prop:surface}. Consider the smooth hypersurface in $\mathbb{R}^{d+1}$ given by $S = \left\{(v, q_r(v)) : v \in U\right\}$,
where $U \subset \mathbb{R}^d$ is a bounded open neighborhood of $\operatorname{supp}\chi$. Let $d\sigma$ denote the surface measure on $S$, so that
\begin{equation*}
d\sigma(v, q_r(v)) = \left(1 + |\nabla q_r(v)|^2\right)^{1/2}\,dv.
\end{equation*}
We can therefore rewrite $J(w,s)$ as
\begin{equation*}
J(w,s) 
= \int_S e\left((w,s)\cdot y\right)\psi(y)\,d\sigma(y)
=\widehat{\psi\, \rmd \sigma}(-w,-s),
\end{equation*}
where $\psi(v,q_r(v))=\chi(v)/(1+|\nabla q_r(v)|^{2})^{1/2}$.
Moreover, by hypothesis, the Gaussian curvature $K(v)$ of $S$ at $(v, q_r(v))$ satisfies
\begin{equation*}
|K(v)| = \frac{|\det \nabla^2 q_r(v)|}{\left(1 + |\nabla q_r(v)|^2\right)^{(d+2)/2}} 
= \frac{|\det \nabla^2\Psi(\omega_0 + v/r^{1/2})|}{\left(1 + |\nabla q_r(v)|^2\right)^{(d+2)/2}} 
\geq \frac{c}{(1 + C_0^2)^{(d+2)/2}} > 0
\end{equation*}
for all $v \in \operatorname{supp}\chi$. Proposition \ref{prop:surface} then gives
\begin{equation*}
|J(w,s)| \lesssim \frac{1}{|(w,s)|^{d/2}}, \qquad |(w,s)| \geq 1.
\end{equation*}
Together with this trivial bound $|J(w,s)| \le \|\chi\|_{L^1}$ for $|(w,s)| < 1$, we obtain
\begin{equation}\label{eq:J_disp}
|J(w,s)| \lesssim \frac{1}{(1 + |w| + |s|)^{d/2}} 
\leq \frac{1}{(1 + |s|)^{d/2}}.
\end{equation}

If $|w| \le C(1 + |s|)$, then $1 + |w|/(1 + |s|) \sim 1$, and hence \eqref{eq:K_J_rel} and \eqref{eq:J_disp} give
\begin{equation*}
|K_{t,r,\omega_0}(z)| 
\lesssim \frac{1}{r^{d/2}(1 + |s|)^{d/2}} 
\sim \frac{1}{(r + |t|)^{d/2}\left(1 + r^{1/2}|z + t\nabla\Psi(\omega_0)|/(r + |t|)\right)^N}.
\end{equation*}
If $|w| > C(1 + |s|)$, taking $M > N + d/2$ in \eqref{eq:J_nonstat} and using $1 + |w| > 1 + |s|$, we obtain from \eqref{eq:K_J_rel} that
\begin{align*}
|K_{t,r,\omega_0}(z)| &\lesssim \frac{1}{r^{d/2}(1 + |w|)^{d/2}(1 + |w|)^{M - d/2}} \\
&\le \frac{1}{r^{d/2}(1 + |s|)^{d/2}\left(1 + \frac{|w|}{1 + |s|}\right)^N} \\
&= \frac{1}{(r + |t|)^{d/2}\left(1 + r^{1/2}|z + t\nabla\Psi(\omega_0)|/(r + |t|)\right)^N}.
\end{align*}
Combining both cases gives \eqref{eq:kernel-cap}. Finally, integrating \eqref{eq:kernel-cap} 
over $z \in \mathbb{R}^d$ with $N > d$ yields $\|K_{t,r,\omega_0}\|_{L^1} \lesssim (1 + |t|/r)^{d/2}$, which proves \eqref{eq:kernel-time}.
\end{proof}

\subsection{Restricted Interpolation}
We use restricted weak-type estimates to obtain strong-type bounds in the reduction of the local smoothing problem.
We recall the relevant definitions and an interpolation theorem from Grafakos \cite{G249}.

\begin{definition}
Let $T$ be an operator defined on the space of finitely simple functions $S(X)$ on a measure space $(X,\mu)$ and let $0<p, q\leq \infty$.
We say that $T$ is of restricted weak type (p,q) if
\begin{equation*}
\|T(1_A)\|_{L^{q,\infty}}\leq C\mu(A)^{1/p}
\end{equation*}
for all measurable subsets $A$ of $X$ with finite measure. Estimates of the form above are called restricted weak type estimates.
\end{definition}

\begin{proposition}[{\cite[Corollary 1.4.22]{G249}}]\label{prop:interpolation}
Let $T$ be a quasi-linear operator defined on the space of simple functions on $X$ and taking values in the set of measurable functions on $Y$.
Let $0<p_0\neq p_1\leq \infty$ and $0<q_0\neq q_1\leq \infty$. If $T$ is restricted weak type $(p_0, q_0)$ and $(p_1, q_1)$ with constants $M_0$ and $M_1$,
respectively, and for some $0<\theta<1$ we have
\begin{equation*}
\frac{1}{p}=\frac{1-\theta}{p_0}+\frac{\theta}{p_1}, \qquad \frac{1}{q}=\frac{1-\theta}{q_0}+\frac{\theta}{q_1},
\end{equation*}
and $p\leq q$, then $T$ satisfies the strong type estimate
\begin{equation}\label{eq:interpolation}
\|Tf\|_{L^{q}}\leq CM_0^{1-\theta}M_1^{\theta}\|f\|_{L^{p}}
\end{equation}
for all $f$ in $S_0(X)$. Moreover, if $T$ is linear (or sublinear with nonnegative values),
then it has a unique bounded extension from $L^{p}(X,\mu)$ to $L^{q}(Y,\nu)$ that satisfies estimate \eqref{eq:interpolation}
for all $f\in L^{p}(X)$ with the constant $C$ replaced by $C2^{2/p}\max\{1, 2^{1/p-1}\}^{2}$.
\end{proposition}

\subsection{Incidence Geometry Estimates}
The geometric part of our argument relies on estimates for unions of shaded tubes.
Wang and Wu \cite{WW} established two-ends Furstenberg inequalities in the plane and obtained related estimates in higher dimensions,
drawing in part on the work of Katz and Tao \cite{KT}. We recall the relevant definitions and estimates below.

\begin{definition}
Let $L$ be a family of lines in $\mathbb{R}^{d}$ and let $\delta\in (0,1)$. A shading $Y$: $L\to B^{d}(0,1)$ is an assignment such that $Y(\ell)\subset N_\delta(\ell)\cap B^{d}(0,1)$
is a union of $\delta$-balls in $\mathbb{R}^{d}$ for all $\ell\in L$. We write $(L,Y)_\delta$ to emphasize the dependence on $\delta$.
\end{definition}

\begin{definition}
Let $\delta\in (0,1)$ and let $(L,Y)_\delta$ be a set of lines and shading. We say $Y$ is $\lambda$-dense,
if $|Y(\ell)|\geq \lambda |N_\delta(\ell)\cap B^{d}(0,1)|$.
\end{definition}

\begin{definition}
Let $\delta\in (0,1)$ and let $(L, Y)_\delta$ be a set of lines and shading. Let $0<\varepsilon_2<\varepsilon_1<1$. We say $Y$ is $(\varepsilon_1, \varepsilon_2, C)$-two-ends
if for all $\ell \in L$ and all $\delta\times \delta^{\varepsilon_1}$-tubes
$J\subset N_\delta(\ell)$,
\begin{equation*}
|Y(\ell)\cap J|
\leq C\delta^{\varepsilon_2}|Y(\ell)|.
\end{equation*}
When the constant $C$ is not important in the context, we say $Y$ is $(\varepsilon_1, \varepsilon_2)$-two-ends, or simply two-ends.
\end{definition}

For each line $\ell$ in $\mathbb R^d$, let $V(\ell)$ denote its direction, identified with a point on $S^{d-1}$.
For a family of lines $L$, let $V(L)=\{V(\ell):\ell\in L\}$, counted with multiplicity.

\begin{definition}[{\cite[Definition 1.25]{WW}}]
A family of $\delta$-separated lines $L$ is called $q$-parallel if every $\delta$-ball on $S^{d-1}$ contains at most $q$ points of $V(L)$.
\end{definition}

\begin{proposition}[{\cite[Theorem 2.1]{WW}}]\label{prop:twoends-plane}
Let $\delta\in (0,1)$. Let $(L, Y)_\delta$ be a set of $1$-parallel, $\delta$-separated lines in $\mathbb{R}^{2}$ with an $(\varepsilon_1, \varepsilon_2)$-two-ends, $\lambda$-dense shading. Then for any $\varepsilon>0$,
\begin{equation*}
|\cup_{\ell\in L}Y(\ell)|\geq c_\varepsilon \delta^{\varepsilon}\delta^{\varepsilon_1/2}\lambda^{1/2}\sum_{\ell\in L}|Y(\ell)|.
\end{equation*}
\end{proposition}

\begin{proposition}[{\cite[Lemma 3.1]{WW}}]\label{prop:twoends-three-d}
Let $\delta\in (0,1)$. Let $(L, Y)_\delta$ be a set of $1$-parallel, $\delta$-separated lines in $\mathbb{R}^{3}$ with an $(\varepsilon_1, \varepsilon_2)$-two-ends, $\lambda$-dense shading. Then for any $\varepsilon>0$,
\begin{equation*}
|\cup_{\ell\in L}Y(\ell)|\geq c_\varepsilon \delta^{\varepsilon}\delta^{3\varepsilon_1/4}\lambda^{3/4}\delta^{1/2}\sum_{\ell\in L}|Y(\ell)|.
\end{equation*}
\end{proposition}

\begin{proposition}[{\cite[Lemma 3.3]{WW}}]\label{prop:twoends-general}
Let $\delta\in (0,1)$. Let $(L, Y)_\delta$ be a set of $\delta$-separated lines in $\mathbb{R}^{d}$ with an $(\varepsilon_1, \varepsilon_2)$-two-ends, $\lambda$-dense shading. Then for any $\varepsilon>0$,
\begin{equation*}
|\cup_{\ell\in L}Y(\ell)|\geq c_\varepsilon \delta^{\varepsilon}\delta^{\varepsilon_1/2}\lambda \delta^{(d-1)/2}(\delta^{d-1}\#L)^{1/2}.
\end{equation*}
\end{proposition}

The preceding estimate is useful for small values of the shading density. For larger densities,
we use the following inequality of Katz and Tao in the form stated by Wang and Wu \cite[Theorem 3.5]{WW}.

\begin{proposition}[{\cite[Theorem 3.5]{WW}}]\label{prop:twoends-dense}
Let $\delta\in (0,1)$. Let $(L, Y)_\delta$ be a set of $1$-parallel, $\delta$-separated lines in $\mathbb{R}^{d}$ with an $(\varepsilon_1, \varepsilon_2)$-two-ends, $\lambda$-dense shading. Suppose $\lambda \geq \delta^{1/4}$. Then for any $\varepsilon>0$,
\begin{equation*}
|\cup_{\ell\in L}Y(\ell)|\geq c_\varepsilon \delta^{\varepsilon}\delta^{\varepsilon_1}\lambda^{(2d+10)/7}\delta^{(3d-3)/7}\left(\sum_{\ell\in L}|Y(\ell)|\right)^{4/7}.
\end{equation*}
\end{proposition}

\section{Reductions}\label{sec:reductions}
We first assume that $f$ is smooth in the periodic variables and Schwartz in the Euclidean variables. The general case follows by density.
We begin with some reductions of the local smoothing problem. For $L\geq 1$, write
\begin{equation*}
U_{\Phi, L}(t)f(x,y)=\sum_{k\in L^{-1}\mathbb{Z}^{m}}\int_{\mathbb{R}^{n}}\widehat{f}(k,\xi)e(k\cdot x+\xi\cdot y+t\Phi(k,\xi))\rmd \xi.
\end{equation*}
\subsection{Littlewood--Paley Decomposition}
Without loss of generality, we may assume that $\Phi(0)=0$ and $\nabla\Phi(0)=0$.
Otherwise, we can replace $\Phi$ by $\Phi-\Phi(0)-\nabla \Phi(0)\cdot \omega$, since
\begin{equation*}
U_{\Phi-\Phi(0)-\nabla\Phi(0)\cdot \omega}(t)f(x,y)=e(-t\Phi(0))U_{\Phi}(t)f((x,y)-t\nabla\Phi(0))
\end{equation*}
and $\|U_{\Phi}(t)f\|_{L^{p}(\mathbb{T}^{m}\times \mathbb{R}^{n}\times [0,1])}=\|U_{\Phi-\Phi(0)-\nabla\Phi(0)\cdot \omega}(t)f\|_{L^{p}(\mathbb{T}^{m}\times \mathbb{R}^{n}\times [0,1])}$.
Let $L=R^{1/a_0}$ and $\Phi_R(\omega)=\Phi(L\omega)/R$. 
By Littlewood--Paley decomposition and scaling, \eqref{eq:thm:main} follows from the estimate below for every $R>1$ and $\varepsilon>0$:
\begin{equation}\label{eq:LP-reduction}
\left\|U_{\Phi_R, L}(t)f \right\|_{L^{p}(\mathbb{T}_{L}^{m}\times \mathbb{R}^{n}\times [0,R])}
\lesssim R^{\varepsilon}(L^{m}R^{n})^{1/2-1/p}\|f\|_{L^{p}},
\end{equation}
where $f$ satisfies $\supp\widehat f\subset \{(k,\xi)\in L^{-1}\mathbb{Z}^{m}\times \mathbb{R}^{n}: |(k,\xi)|\sim 1\}$. We write $\mathbb{T}_{L}^{m}=\mathbb{R}^{m}/(L\mathbb{Z}^{m})$.
\begin{lemma}\label{lem:LP-reduction}
Let
\begin{equation*}
\sigma=(m+a_0n)\left(\frac{1}{2}-\frac{1}{p}\right)-\frac{a_0}{p}.
\end{equation*}
If \eqref{eq:LP-reduction} holds for all $L> 1$, then for any $g\in L^{p}(\mathbb{T}^{m}\times \mathbb{R}^{n})$, 
obeying $\supp \widehat{g}\subset \{(k,\xi)\in \mathbb{Z}^{m}\times \mathbb{R}^{n}:|(k,\xi)|\sim L\}$, we obtain
\begin{equation}\label{eq:LP-rescaled}
\left\|U_{\Phi}(t)g\right\|_{L^{p}(\mathbb{T}^{m}\times \mathbb{R}^{n}\times [0,1])}
\lesssim L^{\sigma+a_0\varepsilon}\|g\|_{L^{p}}.
\end{equation}
\end{lemma}
\begin{proof}
For the function $g$ above, there exists $f\in L^{p}(\mathbb{T}^{m}_{L}\times \mathbb{R}^{n})$ obeying $\supp \widehat{f}\subset \{(k,\xi)\in L^{-1}\mathbb{Z}^{m}\times \mathbb{R}^{n}: |(k,\xi)|\sim 1\}$ such that
\begin{equation*}
\widehat{g}(k,\xi)=\widehat{f}(k/L^1, \xi/L^1), \qquad g(x,y)=L^{n}f(Lx, Ly).
\end{equation*}
Hence, we obtain
\begin{equation*}
U_{\Phi}(t)g(x,y)
=L^{n}U_{\Phi_R,L}(Rt)f(Lx,Ly).
\end{equation*}
By scaling, we have
\begin{equation}\label{eq:LP-scale1}
\|U_\Phi(t)g\|_{L^{p}(\mathbb{T}^{m}\times \mathbb{R}^{n}\times [0,1])}
=L^{n-(m+n+a_0)/p}\|U_{\Phi_R,L}(t)f\|_{L^{p}(\mathbb{T}_{L}^{m}\times \mathbb{R}^{n}\times [0,R])},
\end{equation}
and
\begin{equation}\label{eq:LP-scale2}
\|g\|_{L^{p}}=L^{n-(m+n)/p}\|f\|_{L^{p}}.
\end{equation}
Substituting \eqref{eq:LP-scale1} and \eqref{eq:LP-scale2} into \eqref{eq:LP-reduction}, 
we obtain \eqref{eq:LP-rescaled}.
\end{proof}

For $g\in L^{p}(\mathbb{T}^{m}\times \mathbb{R}^{n})$, take a Littlewood--Paley decomposition
\begin{equation*}
g=g_0+\sum_{j=1}^{\infty}g_j,
\end{equation*}
where $g_0$ satisfies $\supp \widehat{g_0}\subset \{(k,\xi)\in \mathbb{Z}^{m}\times \mathbb{R}^{n}:|(k,\xi)|\lesssim 1\}$ and each $g_j$ satisfies $\supp \widehat{g_j}\subset \{(k,\xi)\in \mathbb{Z}^{m}\times \mathbb{R}^{n}: |(k,\xi)|\sim 2^{j}\}$.
For the low frequency term $g_0$, direct calculation implies
\begin{equation*}
\|U_{\Phi}(t)g_0\|_{L^{p}(\mathbb{T}^{m}\times \mathbb{R}^{n}\times [0,1])}
\lesssim \|g_0\|_{L^{p}}
\sim \|g_0\|_{L^{p}_s}.
\end{equation*}
For high frequency part $g_j$, by \eqref{eq:LP-rescaled}, we obtain
\begin{align*}
\|U_{\Phi}(t)g_j\|_{L^{p}(\mathbb{T}^{m}\times \mathbb{R}^{n}\times [0,1])}
&\lesssim 2^{j(\sigma+a_0\varepsilon)}\|g_j\|_{L^{p}}\\
&=2^{j(\sigma+a_0\varepsilon-s)}2^{js}\|g_j\|_{L^{p}}\\
&\lesssim 2^{j(\sigma+a_0\varepsilon-s)}\|g_j\|_{L_s^{p}}.
\end{align*}
Choose $\varepsilon>0$ such that $a_0\varepsilon<s-\sigma$. Summing in $j$ yields
\begin{align*}
\|U_\Phi(t)g\|_{L^{p}(\mathbb{T}^{m}\times \mathbb{R}^{n}\times [0,1])}
&\lesssim \sum_{j=0}^{\infty}\|U_{\Phi}(t)g_j\|_{L^{p}(\mathbb{T}^{m}\times \mathbb{R}^{n}\times [0,1])}\\
&\lesssim \|g_0\|_{L_s^{p}}+\sum_{j=1}^{\infty} 2^{j(\sigma+a_0\varepsilon-s)}\|g_j\|_{L_s^{p}}\\
&\lesssim \|g\|_{L_s^{p}}.
\end{align*}
Hence, it remains to show \eqref{eq:LP-reduction}. 

After a fixed rescaling, we may assume that $\supp \widehat{f}\subset \{(k,\xi)\in L^{-1}\mathbb{Z}^{m}\times \mathbb{R}^{n}: |k|\lesssim 1, |\xi|\lesssim 1\}$.
We also allow fixed enlargements of the Fourier support in the estimates below. 
Since $\Phi_R$ may fail to be smooth near the origin, we take a smooth extension of it.

Choose $\chi\in C_c^\infty(\mathbb{R}^{m+n})$ such that $0\leq\chi\leq1$, $\chi=1$ near the origin, 
and $\supp\chi$ is contained in a sufficiently small ball disjoint from $\{\omega\in \mathbb{R}^{m+n}: |\omega|\sim1\}$.
Let $\kappa\in C_c^\infty(\mathbb{R}^{m+n})$ be nonnegative with $\int\kappa=1$, and set
$\kappa_\delta(\omega)=1/\delta^{m+n}\kappa(\omega/\delta)$.
For fixed sufficiently small $\delta>0$, define
\begin{equation*}
\widetilde{\Phi}_R(\omega)
=(1-\chi(\omega))\Phi_R(\omega)+\chi(\omega)\kappa_\delta*\Phi_R(\omega).
\end{equation*}
Then $\widetilde{\Phi}_R\in C^\infty(\mathbb{R}^{m+n})$ and
$\widetilde{\Phi}_R=\Phi_R$ on $\{\omega\in \mathbb{R}^{m+n}: |\omega|\sim1\}$.

Moreover, for sufficiently small fixed $\delta>0$ and $|\omega|\lesssim 1$,
\begin{equation*}
\nabla^2\widetilde\Phi_R(\omega)\sim I,
\qquad
|D^\beta\widetilde\Phi_R(\omega)|
\lesssim 1,\quad |\beta|\geq3.
\end{equation*}
Indeed, the phase assumptions and $R=L^{a_0}$ give
\begin{equation*}
\nabla^2\Phi_R(\omega)
=\frac{L^2}{R}\nabla^2\Phi(L\omega)
\sim|\omega|^{a_0-2}I,
\qquad
|D^\beta\Phi_R(\omega)|
\lesssim |\omega|^{a_0-|\beta|},
\quad |\beta|\geq3,
\end{equation*}
for $\omega\neq0$.
Since $a_0>1$, the Hessian of $\Phi_R$ is locally integrable. Thus
\begin{align*}
\nabla^2(\kappa_\delta*\Phi_R)(\omega)
&=\int_{\mathbb R^{m+n}}\kappa_\delta(z)\nabla^2\Phi_R(\omega-z)\rmd z\\
&\sim\int_{\mathbb R^{m+n}}\kappa_\delta(z)|\omega-z|^{a_0-2}\rmd z\,I
\sim I,
\end{align*}
for $|\omega|\lesssim1$. On the support of the derivatives of $\chi$, the phase $\Phi_R$ is smooth with uniform derivative bounds, so
\begin{equation*}
\|\kappa_\delta*\Phi_R-\Phi_R\|_{C^2}
\lesssim\delta.
\end{equation*}
Consequently, in the transition region,
\begin{equation*}
\nabla^2\widetilde\Phi_R
=\nabla^2\Phi_R+O(\delta)\sim I
\end{equation*}
for sufficiently small $\delta$.
Outside $\supp\chi$, the extension agrees with $\Phi_R$. The higher derivative bounds follow from the phase assumptions and the smoothness
of the mollification. Henceforth, we replace $\Phi_R$ by $\widetilde{\Phi}_R$ and retain the notation $\Phi$.

For fixed $\eta\in\mathbb R^m$ and a function $g$ on
$\mathbb R^n$, we denote
\begin{equation*}
\mathcal E_n(t;\eta)g(y)
=\int_{\mathbb R^n}\widehat g(\xi)
e(y\cdot\xi+t\Phi(\eta,\xi))\rmd\xi.
\end{equation*}
We also write $\mathcal E_{m,n}(t)=U_{\Phi,L}(t)$. If $t$ is not important, we omit it.

\subsection{Restricted Interpolation}
The following interpolation lemma reduces \eqref{eq:LP-reduction} to a mixed $L^2$--$L^\infty$ estimate at $p_0$.
In the estimates below, $L=R^{1/{a_0}}$, and $\Phi$ is the smooth function constructed above.
\begin{lemma}\label{lem:interpolation-reduction}
Let $p_0$ be as in Theorem \ref{thm:main}. Suppose that, for every $\varepsilon>0$,
\begin{equation}\label{eq:waveguide-mixed}
\|\mathcal{E}_{m,n}(t)f\|_{L^{p_0}(\mathbb{T}_L^m\times\mathbb{R}^n\times[0,R])}
\lesssim R^{\varepsilon}(L^{m}R^{n})^{1/2-1/p_0}\|f\|_{L^2}^{2/p_0}
\|f\|_{L^\infty}^{1-2/p_0}
\end{equation}
holds for all $R> 1$ with $L=R^{1/a_0}$, and all $f$ satisfying $\supp \widehat{f}\subset \{(k,\xi)\in L^{-1}\mathbb{Z}^{m}\times \mathbb{R}^{n}:|k|\leq 1,|\xi|\leq 1\}$. 
Then \eqref{eq:LP-reduction} holds for every $p_0\leq p<\infty$ and $\varepsilon>0$.
\end{lemma}

\begin{proof}
By rescaling, \eqref{eq:waveguide-mixed} also holds when $|k|\lesssim 1$ and $|\xi|\lesssim 1$.
Choose $\chi_1\in C_c^\infty(\mathbb{R}^m)$ and $\chi_2\in C_c^\infty(\mathbb{R}^n)$ such that
$\chi_1=1$ on $B^{m}(0,1)$ with $\supp\chi_1\subset B^{m}(0,2)$,
and $\chi_2=1$ on $B^{n}(0,1)$ with $\supp\chi_2\subset B^{n}(0,2)$.

For a bounded function $g$ on $\mathbb{T}_L^m\times\mathbb{R}^n$, extend $g$ periodically in the $x$-variables and define
\begin{equation*}
h_g(x,y)
=\int_{\mathbb{R}^m}\int_{\mathbb{R}^n}\chi_1^\vee(\eta)\chi_2^\vee(\xi) g(x-\eta,y-\xi)\rmd \eta\rmd \xi.
\end{equation*}
Then
\begin{equation*}
\widehat{h_g}(k,\xi)
=\chi_1(k)\chi_2(\xi)\widehat g(k,\xi),
\end{equation*}
and Young's inequality gives
\begin{equation*}
\|h_g\|_{L^2}\lesssim\|g\|_{L^2},
\qquad \|h_g\|_{L^\infty}\lesssim\|g\|_{L^\infty}.
\end{equation*}
Define the linear operator
\begin{equation*}
Tg(x,y,t)=\mathcal{E}_{m,n}(t)h_g(x,y).
\end{equation*}

We first verify the $L^\infty$ bound. Define
\begin{equation*}
K_{2,t,k}(y)
=\int_{\mathbb{R}^{n}}\chi_2(\xi)e(y\cdot \xi+t\Phi(k,\xi))\rmd \xi,
\end{equation*}
\begin{equation*}
K_t(x,y)=\frac{1}{L^{m}}\sum_{k\in L^{-1}\mathbb{Z}^{m}}\chi_1(k)e(k\cdot x)K_{2,t,k}(y).
\end{equation*}
For $|k|\lesssim 1$, the phase assumption gives $\nabla_\xi^{2}\Phi(k,\xi)\sim I$ and $|\nabla_\xi \Phi(k,0)|\lesssim 1$ on the support of $\chi_2$.
Corollary \ref{cor:kernel} therefore yields, for every $N>n$,
\begin{equation*}
\sup_{|k|\lesssim 1}|K_{2,t,k}(y)|\lesssim \frac{1}{(1+|t|)^{n/2}(1+|y|/(1+|t|))^{N}}.
\end{equation*}
By H\"older's inequality and orthogonality,
\begin{align*}
\|K_t\|_{L^{1}(\mathbb{T}_L^m\times\mathbb{R}^n)}
&\leq L^{m/2}\int_{\mathbb{R}^{n}}\left(\int_{\mathbb{T}_L^{m}} |K_t(x,y)|^{2}\rmd x\right)^{1/2}\rmd y\\
&\leq L^{m/2}\int_{\mathbb{R}^{n}}\left(\frac{1}{L^{m}}\sum_k |\chi_1(k)K_{2,t,k}(y)|^{2}\right)^{1/2}\rmd y\\
&\lesssim L^{m/2}\int_{\mathbb{R}^{n}}\sup_{|k|\lesssim 1}|K_{2,t,k}(y)|\rmd y\\
&\lesssim L^{m/2}(1+|t|)^{n/2}.
\end{align*}
Since $Tg=K_t\ast g$, Young's inequality gives
\begin{equation*}
\|Tg\|_{L^\infty(\mathbb{T}_L^m\times\mathbb{R}^n\times [0,R])}
\lesssim L^{m/2}R^{n/2}\|g\|_{L^{\infty}}.
\end{equation*}
In particular, $T$ is of restricted weak type $(\infty,\infty)$.

Now let $A\subset \mathbb{T}_L^m\times\mathbb{R}^n$ be measurable with $|A|<\infty$. To use the assumed estimate, choose 
$\vartheta\in C_c^\infty(\mathbb{R}^n)$ with $0\leq\vartheta\leq 1$ and $\vartheta=1$ near the origin.
The functions $h_{\vartheta(y/j)1_A}$ are smooth in the periodic variables and Schwartz in the nonperiodic variables, and converge
to $h_{1_A}$ in $L^2$. Their $L^2$ and $L^\infty$ norms are bounded by $C|A|^{1/2}$ and $C$, respectively.
Applying \eqref{eq:waveguide-mixed} and Fatou's lemma along this subsequence, we obtain
\begin{align*}
\left\|T1_A\right\|_{L^{p_0,\infty}(\mathbb{T}_L^m\times\mathbb{R}^n\times [0,R])}
&\leq \left\|\lim_{j\to \infty}T(\vartheta(y/j)1_A)\right\|_{L^{p_0}(\mathbb{T}_L^m\times\mathbb{R}^n\times [0,R])}\\
&\leq \liminf_{j\to \infty} \left\|T(\vartheta(y/j)1_A)\right\|_{L^{p_0}(\mathbb{T}_L^m\times\mathbb{R}^n\times [0,R])}\\
&\lesssim R^{\varepsilon}(L^{m}R^{n})^{1/2-1/p_0}\left\|h_{\vartheta(y/j)1_A}\right\|_{L^{2}}^{2/p_0}
\left\|h_{\vartheta(y/j)1_A}\right\|_{L^{\infty}}^{1-2/p_0}\\
&\lesssim R^{\varepsilon}(L^{m}R^{n})^{1/2-1/p_0}|A|^{1/p_0}.
\end{align*}
Therefore $T$ is of restricted weak type $(p_0,p_0)$ with constants
\begin{equation*}
M_0 \lesssim R^{\varepsilon}(L^{m}R^{n})^{1/2-1/p_0},
\qquad M_1\lesssim L^{m/2}R^{n/2}.
\end{equation*}

For $p_0<p<\infty$, take $0<\theta<1$ such that
\begin{equation*}
\frac1p=\frac{1-\theta}{p_0} +\frac{\theta}{\infty}.
\end{equation*}
Since $T$ is linear, Proposition \ref{prop:interpolation} gives
\begin{align*}
\|Tg\|_{L^p(\mathbb{T}_L^m\times\mathbb{R}^n\times [0,R])}
&\lesssim M_0^{1-\theta}M_1^\theta \|g\|_{L^p}\\
&\lesssim R^{p_0\varepsilon/p}(L^{m}R^{n})^{1/2-1/p} \|g\|_{L^p}.
\end{align*}
For $p=p_0$, fix $q>p_0$ and choose $0<\gamma<1$ such that
\begin{equation*}
\frac{1}{p_0}=\frac{1-\gamma}{2}+\frac{\gamma}{q}.
\end{equation*}
Since the phase is real-valued, Plancherel's identity gives
\begin{align*}
\|Tg\|_{L^2(\mathbb{T}_L^m\times\mathbb{R}^n\times [0,R])}
&=\left(\int_0^{R}\left\|\mathcal{E}_{m,n}(t)h_g\right\|_{L^{2}(\mathbb{T}_L^{m}\times \mathbb{R}^{n})}^{2}\rmd t\right)^{1/2}\\
&=R^{1/2}\left(L^{m}\sum_{k\in L^{-1}\mathbb{Z}^{m}}\int_{\mathbb{R}^{n}}\left|\chi_1(k)\chi_2(\xi)\widehat{g}(k,\xi)\right|^{2}\rmd \xi\right)^{1/2}\\
&\lesssim R^{1/2}\|g\|_{L^{2}}.
\end{align*}
Interpolating this estimate with the preceding $L^{q}$ bound yields
\begin{equation*}
\|Tg\|_{L^{p_0}(\mathbb{T}_L^m\times\mathbb{R}^n\times [0,R])}
\lesssim R^{(1-\gamma)/2+\varepsilon\gamma}(L^{m}R^{n})^{1/2-1/p_0}\|g\|_{L^{p_0}}.
\end{equation*}
Taking $q$ sufficiently close to $p_0$ and adjusting $\varepsilon$ proves the endpoint estimate.
Since $|k|\leq 1$ and $|\xi|\leq 1$ on $\supp \widehat{f}$, we have $h_f=f$. By density and arbitrariness of $\varepsilon>0$, 
this proves \eqref{eq:LP-reduction}.
\end{proof}

\subsection{Reduction to a Uniform Euclidean Estimate}
We next pass from \eqref{eq:waveguide-mixed} to a Euclidean extension estimate by replacing the discrete periodic frequencies with disjoint smooth bumps.

We identify $f$ with its $L$-periodic extension in the $x$-variables, while its $L^p$ norms are still taken over $\mathbb T_L^m\times\mathbb R^n$.
Let $\rho\in C_c^{\infty}(\mathbb{R}^{m})$ such that
$\supp \rho\subset B^{m}(0,1)$ and $\int_{\mathbb{R}^{m}}\rho(x)\rmd x=1$. Set $\tau\ll 1$ obeying $\tau \ll 1/R$. Define
\begin{equation*}
\rho_\tau=\frac{1}{\tau^{m}}\rho\left(\frac{\cdot}{\tau}\right), \qquad F_\tau(\eta, \xi)=\sum_{k\in L^{-1}\mathbb{Z}^{m}}\widehat{f}_k(\xi)\rho_\tau(\eta-k),
\end{equation*}
where
\begin{equation*}
f_k(y)=\frac{1}{L^{m}}\int_{\mathbb{T}_L^{m}}f(x,y)e(-k\cdot x)\rmd x.
\end{equation*}
We write $F_\tau[f]$ to emphasize the dependence on $f$. Define the extension operator 
\begin{equation*}
E_{m+n} G(x,y,t)=\int_{\mathbb{R}^{m+n}}G(\eta,\xi)e(x\cdot \eta+y\cdot \xi+t\Phi(\eta,\xi))\rmd \eta \rmd \xi.
\end{equation*}
The next two lemmas record the basic properties of $F_\tau$ and  hold for all $1\leq L\leq R$.
\begin{lemma}\label{lem:Ftau-basic}
Let $0<\tau\leq c/R$, where $c>0$ is sufficiently small. Then the sets $k+\tau\supp\rho$, $k\in L^{-1}\mathbb{Z}^m$, are pairwise disjoint. Moreover,
\begin{equation}\label{eq:Ftau-inverse}
F_\tau^\vee(x,y)=\rho^\vee(\tau x)f(x,y),
\end{equation}
and
\begin{equation}\label{eq:Ftau-L2}
\tau^m\|F_\tau\|_{L^2}^2
=\|\rho\|_{L^2}^2\sum_{k\in L^{-1}\mathbb{Z}^m}\|f_k\|_{L^2}^2,
\end{equation}
\begin{equation}\label{eq:Ftau-L2-periodic}
(\tau L)^m\|F_\tau\|_{L^2}^2
=\|\rho\|_{L^2}^2
\|f\|_{L^2}^2,
\end{equation}
\begin{equation}\label{eq:Ftau-Linfty}
\|F_\tau^\vee\|_{L^\infty}
\sim \|f\|_{L^\infty}.
\end{equation}
Furthermore, we have
\begin{equation}\label{eq:Ftau-extension}
E_{m+n}F_\tau(x,y,t)
=\sum_{k\in L^{-1}\mathbb{Z}^m}
\int_{\mathbb{R}^{m}}\rho_\tau(\eta-k)e(x\cdot \eta)\mathcal{E}_n(t;\eta)f_k(y)\rmd \eta
\end{equation}
and for each $y$, $t$,
\begin{equation}\label{eq:Ftau-orthogonality}
\int_{\mathbb{R}^{m}}\left|E_{m+n}F_\tau(x,y,t)\right|^{2}\rmd x
=\sum_{k\in L^{-1}\mathbb{Z}^{m}}\int_{\mathbb{R}^{m}}\left|\rho_\tau(\eta-k)\right|^{2}|\mathcal{E}_n(t;\eta)f_k(y)|^{2}\rmd \eta.
\end{equation}
\end{lemma}

\begin{proof}
Since $L\leq R$ and $\tau\leq c/R$, the supports of $\rho_\tau(\cdot-k)$ are pairwise disjoint when $c$ is sufficiently small. 
By Fourier inversion,
\begin{equation*}
F_\tau^\vee(x,y)
=\rho^\vee(\tau x)\sum_{k\in L^{-1}\mathbb{Z}^m}e(k\cdot x)f_k(y)
=\rho^\vee(\tau x)f(x,y),
\end{equation*}
which gives \eqref{eq:Ftau-inverse}. Plancherel's identity and the disjointness of the supports yield \eqref{eq:Ftau-L2}.
By Parseval's identity on $\mathbb{T}_L^m$:
\begin{equation*}
\|F_\tau\|_{L^2}^2
=\frac{\|\rho\|_2^2}{\tau^{m}}\sum_k\|f_k\|_{L^2}^2
=\frac{\|\rho\|_2^2}{(\tau L)^{m}}\|f\|_{L^2}^2,
\end{equation*}
proving \eqref{eq:Ftau-L2-periodic}.
Since $\rho^\vee(0)=1$ and $\tau L\leq c$, we have $|\rho^\vee(\tau x)|\sim1$ on $[0,L)^m$ for sufficiently small $c$.
Together with \eqref{eq:Ftau-inverse} and the boundedness of $\rho^\vee$, this gives \eqref{eq:Ftau-Linfty}.
The definition of $F_\tau$ gives \eqref{eq:Ftau-extension}; 
Plancherel's identity in $x$ and the disjoint $\eta$-supports then give \eqref{eq:Ftau-orthogonality}.
\end{proof}

The next lemma identifies the limits of the Euclidean norms of $F_\tau^\vee$ and $E_{m+n}F_\tau$ with their periodic counterparts.

\begin{lemma}\label{lem:Ftau-limit}
Let $1\leq p<\infty$ and let $f$ be smooth in the periodic variables and Schwartz in the Euclidean variables, with compact Fourier support. Then
\begin{equation}\label{eq:Ftau-Lp-limit}
\lim_{\tau\to 0}(\tau L)^{m/p}\|F_\tau^\vee\|_{L^{p}}
=\|\rho^\vee\|_{L^p} \|f\|_{L^{p}}
\end{equation}
and
\begin{equation}\label{eq:extension-Lp-limit}
\lim_{\tau\to 0}(\tau L)^{m/p}\|E_{m+n}F_\tau\|_{L^{p}(\mathbb{R}^{m+n}\times [0,R])}
=\|\rho^\vee\|_{L^p}\|\mathcal{E}_{m,n}f\|_{L^{p}(\mathbb{T}_L^{m}\times \mathbb{R}^{n}\times [0,R])}.
\end{equation}
\end{lemma}
\begin{proof}
Since $f$ is $L$-periodic on $\mathbb{R}^{m}$, by \eqref{eq:Ftau-inverse}, we have
\begin{equation}\label{eq:Ftau-limit1}
\begin{aligned}
\|F_\tau^{\vee}\|_{L^{p}}^{p}
&=\int_{\mathbb{R}^{m}}\int_{\mathbb{R}^{n}}|f(x,y)|^{p}|\rho^\vee(\tau x)|^{p}\rmd y\rmd x\\
&=\sum_{l\in \mathbb{Z}^{m}}\int_{[0,L)^{m}+Ll}\int_{\mathbb{R}^{n}}|f(x,y)|^{p}|\rho^\vee(\tau x)|^{p}\rmd y\rmd x\\
&=\int_{[0,L)^{m}}\int_{\mathbb{R}^{n}}|f(x,y)|^{p}\sum_{l\in \mathbb{Z}^{m}}|\rho^\vee(\tau x+\tau Ll)|^{p}\rmd y\rmd x.
\end{aligned}
\end{equation}
Noting that $\rho^{\vee}$ is Schwartz, due to Riemann's sum, it yields that
\begin{equation}\label{eq:Ftau-limit2}
(\tau L)^{m}\sum_{l\in \mathbb{Z}^{m}}|\rho^{\vee}(\tau x+\tau Ll)|^{p}
\to \|\rho^{\vee}\|_{L^{p}}^{p}, \qquad \tau \to 0,
\end{equation}
uniformly for $x\in [0,L)^{m}$. Combining \eqref{eq:Ftau-limit1} and \eqref{eq:Ftau-limit2}, we obtain \eqref{eq:Ftau-Lp-limit}.

To prove \eqref{eq:extension-Lp-limit}, set $D_\tau=E_{m+n}F_\tau-\rho^{\vee}(\tau x)\mathcal{E}_{m,n}f$. The change of variables $\eta=k+\tau u$ gives
\begin{equation*}
D_\tau(x,y,t)
=\sum_k e(k\cdot x)\int_{\mathbb{R}^{m}}\int_{\mathbb{R}^{n}}\rho(u)\widehat{f_k}(\xi)e(\tau x\cdot u+y\cdot\xi)
\big(e(t\Phi(k+\tau u,\xi))-e(t\Phi(k,\xi))\big)\rmd\xi\rmd u.
\end{equation*}
By the fundamental theorem of calculus, we have
\begin{equation*}
e(t\Phi(k+\tau u,\xi))-e(t\Phi(k,\xi))
=2\pi i t\tau\int_{0}^{1}u\cdot \nabla_\eta\Phi(k+s\tau u,\xi)e(t\Phi(k+s\tau u,\xi))\rmd s.
\end{equation*}
Define
\begin{equation*}
a_{k,\tau,t}(u,\xi)=2\pi i t\rho(u)\widehat{f_k}(\xi)\int_{0}^{1}u\cdot \nabla_\eta\Phi(k+s\tau u,\xi)e(t\Phi(k+s\tau u,\xi))\rmd s.
\end{equation*}
Thus, we can rewrite $D_\tau$ as
\begin{equation*}
D_\tau(x,y,t)=\tau\sum_k e(k\cdot x)\int_{\mathbb{R}^{m+n}}a_{k,\tau,t}(u,\xi)e(\tau x\cdot u+y\cdot \xi)\rmd u\rmd \xi.
\end{equation*}
For fixed $f$, $L$, $R$, the functions $a_{k,\tau,t}(u,\xi)$ form a bounded family in $C_c^{\infty}(\mathbb{R}^{m+n})$, uniformly for $0<\tau\leq c/R$, 
$0\leq t\leq R$ and the finitely many relevant $k$. Thus, for every $N$,
\begin{equation*}
|D_\tau(x,y,t)|
\lesssim \frac{\tau L^{m}}{(1+|(\tau x,y)|)^{2N}}
\lesssim \frac{\tau L^{m}}{(1+\tau |x|)^{N}(1+|y|)^{N}}.
\end{equation*}
Taking $Np>\max\{m,n\}$, we obtain
\begin{equation}\label{eq:extension-limit1}
\begin{aligned}
(\tau L)^{m/p}\|D_\tau\|_{L^{p}(\mathbb{R}^{m+n}\times [0,R])}
&\lesssim (\tau L)^{m/p}\tau L^{m}\left(\int_{0}^{R}\rmd t\int_{\mathbb{R}^{m}}\frac{\rmd x}{(1+\tau |x|)^{Np}}
\int_{\mathbb{R}^{n}}\frac{\rmd y}{(1+|y|)^{Np}}\right)^{1/p}\\
&\lesssim \tau L^{m+m/p}R^{1/p}\to 0, \qquad \tau\to 0.
\end{aligned}
\end{equation}
Since $\mathcal{E}_{m,n}(t)f$ is Schwartz in $y$ uniformly for $0\leq t\leq R$, \eqref{eq:Ftau-limit2} also gives
\begin{equation}\label{eq:extension-limit2}
\lim_{\tau\to0}(\tau L)^{m/p}\|\rho^\vee(\tau x)\mathcal{E}_{m,n}f\|_{L^p(\mathbb{R}^{m+n}\times[0,R])}
=\|\rho^\vee\|_{L^p}\|\mathcal{E}_{m,n}f\|_{L^p(\mathbb{T}_L^m\times\mathbb{R}^n\times[0,R])}.
\end{equation}
The reverse triangle inequality, \eqref{eq:extension-limit1} and \eqref{eq:extension-limit2} prove \eqref{eq:extension-Lp-limit}.
\end{proof}

\begin{remark}
Restricting the $(y,t)$-integrals in \eqref{eq:Ftau-limit1} and \eqref{eq:extension-limit2} 
shows that \eqref{eq:extension-Lp-limit} remains valid with $\mathbb R^n\times[0,R]$ replaced by any fixed measurable subset of it.
\end{remark}

By Lemmas \ref{lem:interpolation-reduction} and \ref{lem:Ftau-limit}, and \eqref{eq:Ftau-L2-periodic}, \eqref{eq:Ftau-Linfty}, 
to prove \eqref{eq:LP-reduction}, it suffices to show that
\begin{equation}\label{eq:Euclidean-mixed}
\begin{aligned}
\|E_{m+n}F_\tau\|_{L^{p_0}(\mathbb{R}^{m+n}\times[0,R])}
&\lesssim R^\varepsilon\left(L^mR^n\right)^{1/2-1/p_0}\|F_\tau\|_{L^2}^{2/p_0}\|F_\tau^\vee\|_{L^\infty}^{1-2/p_0}\\
&\lesssim R^\varepsilon\left(L^mR^n\right)^{1/2-1/p_0}\|F_\tau\|_{L^2}^{2/p_0}\|f\|_{L^\infty}^{1-2/p_0}
\end{aligned}
\end{equation}
holds for every $R>1$, with $L=R^{1/a_0}$ and $0<\tau\leq c/R$. When $m=0$, we omit the periodic variables and sums, 
put $F_\tau=\widehat f$, and interpret the periodic factors and the norms of $\rho$ as $1$. 

\subsection{Space-Time Localization}
We finally reduce \eqref{eq:Euclidean-mixed} to a local estimate on bounded balls in space-time, 
as is standard in Schr\"odinger local smoothing estimates (see, e.g., \cite[Lemma 2.2]{LRS}).

Decompose $[0,R]$ into
\begin{equation*}
[0,R]=[0,1]\cup
\bigcup_{j\geq 0, 2^j<R} [2^j,\min\{2^{j+1},R\}].
\end{equation*}
Writing $T=2^j$, we consider balls $B_r(y_0,t_0)\subset\mathbb{R}^{n+1}$ with $r\sim T$ and $t_0\sim T$,
chosen so that $t\sim T$ on their fixed enlargements.

\begin{lemma}\label{lem:time-reduction}
For every $\varepsilon>0$, suppose that
\begin{equation}\label{eq:local-reduction}
\int_{\mathbb{R}^{m}\times B_r(y_0,t_0)}|E_{m+n}F_\tau(x,y,t)|^{p_0}\rmd x\rmd y\rmd t
\lesssim R^\varepsilon \left(L^{m}T^{n}\right)^{p_0/2-1} \|F_\tau\|_{L^2}^2 \|f\|_{L^\infty}^{p_0-2}
\end{equation}
holds for every dyadic $1\leq T\leq R$, with $R=L^{a_0}$, $0<\tau\leq c/R$,
$B_r(y_0,t_0)$ above, whenever $|k|\leq1$ and $|\xi|\leq2$ on $\supp\widehat f$.
Then, for $f$ satisfying $|k|\leq1$ and $|\xi|\leq1$ on $\supp\widehat f$,
\begin{equation}\label{eq:time-interval}
\int_T^{2T}\int_{\mathbb{R}^{m+n}}|E_{m+n}F_\tau(x,y,t)|^{p_0}\rmd x\rmd y\rmd t
\lesssim R^\varepsilon \left(L^{m}T^{n}\right)^{p_0/2-1} \|F_\tau\|_{L^2}^2\|f\|_{L^\infty}^{p_0-2}.
\end{equation}
Consequently, \eqref{eq:Euclidean-mixed} holds.
\end{lemma}

\begin{proof}
Choose $\chi\in C_c^\infty(\mathbb{R}^n)$ such that $\chi=1$ on $B^n(0,1)$ with $\supp\chi\subset B^n(0,2)$, and set
\begin{equation*}
K_{t,\eta}(y)=\int_{\mathbb R^n}\chi(\xi)e(y\cdot\xi+t\Phi(\eta,\xi))\rmd\xi.
\end{equation*}
For $t\in[0,1]$, Corollary \ref{cor:kernel} gives, uniformly for $|\eta|\leq2$,
\begin{equation*}
|K_{t,\eta}(y)|\lesssim \frac{1}{(1+|y|)^M},
\end{equation*}
where $M>n$.
Since $\chi=1$ on $\supp \widehat{f_k}$, \eqref{eq:Ftau-extension} gives
\begin{equation*}
E_{m+n}F_\tau(x,y,t)=\int_{\mathbb{R}^{m}}\sum_{k}\rho_\tau(\eta-k)f_k\ast K_{t,\eta}(y) e(\eta\cdot x)\rmd \eta.
\end{equation*}
By $\#\{k:|k|\leq 1\}\lesssim L^{m}$, Cauchy--Schwarz inequality and Parseval's identity, we have
\begin{align*}
|E_{m+n}F_\tau(x,y,t)|
&\leq \sum_k \int_{\mathbb{R}^{m}}|\rho_\tau(\eta-k)|\rmd \eta \cdot \sup_{|\eta|\leq 2}|f_k\ast K_{t,\eta}(y)|\\
&\lesssim L^{m/2}\int_{\mathbb{R}^n}\frac{1}{(1+|y-z|)^M}\left(\sum_k|f_k(z)|^2\right)^{1/2}\rmd z\\
&\lesssim L^{m/2}\|f\|_{L^\infty}.
\end{align*}
Together with spatial $L^2$ unitarity, we obtain
\begin{equation}\label{eq:bounded-time}
\begin{aligned}
&\int_0^1\int_{\mathbb{R}^{m+n}}|E_{m+n}F_\tau(x,y,t)|^{p_0} \rmd x\rmd y\rmd t\\
&\qquad=\int_0^1\int_{\mathbb{R}^{m+n}}|E_{m+n}F_\tau(x,y,t)|^{2}|E_{m+n}F_\tau(x,y,t)|^{p_0-2} \rmd x\rmd y\rmd t\\
&\qquad\lesssim L^{m(p_0/2-1)}\|F_\tau\|_{L^2}^2 \|f\|_{L^\infty}^{p_0-2}.
\end{aligned}
\end{equation}

Now let $T\geq1$ be dyadic. Partition $\mathbb{R}^n$ into disjoint cubes $Q_j$ of side length $T$.
Choose $q_j\in C_c^\infty(\mathbb{R}^n)$,
$0\leq q_j\leq1$, satisfying $q_j(y)=1$ if $\operatorname{d}(y,Q_j)\leq CT$,
and $\supp q_j\subset \{y:\operatorname{dist}(y,Q_j)\leq2CT\}$, where $C$ is sufficiently large.
These supports have bounded overlap, so
$\sum_jq_j(y)^2\lesssim 1$.
Set $f_j=\chi^\vee\ast (q_jf)$.
By \eqref{eq:Ftau-L2}, we have
\begin{equation*}
\begin{aligned}
\sum_j\|F_\tau[f_j]\|_{L^2}^2
&=\frac{\|\rho\|_{L^2}^2}{\tau^m} \sum_{j,k}\|\chi^\vee*(q_jf_k)\|_{L^2}^2\\
&\lesssim\frac{\|\rho\|_{L^2}^2}{\tau^m}\sum_k\int_{\mathbb{R}^n} \sum_jq_j(y)^2|f_k(y)|^2\rmd y\\
&\lesssim\|F_\tau[f]\|_{L^2}^2.
\end{aligned}
\end{equation*}
Moreover, Young's inequality gives
\begin{equation*}
\sup_j\|f_j\|_{L^\infty}
\lesssim\|f\|_{L^\infty}.
\end{equation*}

We first estimate $f_j$ on $Q_j\times[T,2T]$. By construction,
\begin{equation*}
\supp\widehat{f_j}
\subset \{(k,\xi)\in L^{-1}\mathbb{Z}^m\times\mathbb{R}^n:|k|\leq1,\ |\xi|\leq2\}.
\end{equation*}
Each $Q_j\times[T,2T]$ can be covered by a bounded number of balls $B_r(y_0,t_0)$ with $r\sim T$ and $t_0\sim T$, 
chosen so that $t\sim T$ on their fixed enlargements. Applying \eqref{eq:local-reduction}, we obtain
\begin{align*}
&\int_T^{2T}\int_{\mathbb{R}^m\times Q_j}|E_{m+n}F_\tau[f_j](x,y,t)|^{p_0}\rmd x\rmd y\rmd t\\
&\qquad\lesssim R^\varepsilon \left(L^mT^n\right)^{p_0/2-1} \|F_\tau[f_j]\|_{L^2}^2 \|f_j\|_{L^\infty}^{p_0-2}.
\end{align*}
Summing over $j$ gives
\begin{equation}\label{eq:time-near}
\begin{aligned}
&\sum_j\int_T^{2T}\int_{\mathbb{R}^m\times Q_j}|E_{m+n}F_\tau[f_j](x,y,t)|^{p_0}\rmd x\rmd y\rmd t\\
&\qquad \lesssim R^\varepsilon \left(L^mT^n\right)^{p_0/2-1} \|F_\tau[f]\|_{L^2}^2 \|f\|_{L^\infty}^{p_0-2}.
\end{aligned}
\end{equation}

We estimate $f-f_j$ on $Q_j$. Since $\chi=1$ on $\supp \widehat{f_k}$, we have $f_k-(f_j)_k=\chi^{\vee}\ast ((1-q_j)f_k)$.
For $y\in Q_j$, $1-q_j(z)\neq 0$ implies $|y-z|\geq CT$. Thus,
\begin{align*}
\left|\int_{\mathbb{R}^{n}}K_{t,\eta}(y-z)(1-q_j(z))f_k(z)\rmd z\right|
&\leq \int_{\mathbb{R}^{n}} 1_{\{|y-z|\geq CT\}}|K_{t,\eta}(y-z)||f_k(z)|\rmd z\\
&= \int_{|z|\geq CT}|K_{t,\eta}(z)||f_k(y-z)|\rmd z.
\end{align*}
Choosing $C$ sufficiently large so that
\begin{equation*}
\left|z+t\nabla_\xi \Phi(\eta,0)\right|
\geq |z|-t\left|\nabla_\xi\Phi(\eta,0)\right|
\geq |z|-O(T)
\geq \frac{|z|}{2}
\geq C_0(1+t).
\end{equation*}
Corollary \ref{cor:kernel} gives, uniformly for $|\eta|\leq2$ and $t\in [T,2T]$,
\begin{equation*}
|K_{t,\eta}(z)|1_{\{|z|\geq CT\}}
\lesssim \frac{1_{\{|z|\geq CT\}}}{(1+|z|)^M}.
\end{equation*}
By \eqref{eq:Ftau-extension} and Cauchy--Schwarz inequality, we have
\begin{align*}
&|E_{m+n}F_\tau[f-f_j](x,y,t)|\\
&\quad=\left|\int_{\mathbb{R}^{m}}\sum_k \rho_{\tau}(\eta-k)e(\eta\cdot x)\int_{\mathbb{R}^{n}}K_{t,\eta}(y-z)(1-q_j(z))f_k(z)\rmd z \rmd \eta\right|\\
&\quad\lesssim L^{m/2}\int_{\mathbb{R}^{n}}\frac{1_{\{|y-z|\geq CT\}}}{(1+|y-z|)^M}\left(\sum_k|f_k(z)|^2\right)^{1/2}\rmd z.
\end{align*}
On the other hand, \eqref{eq:Ftau-orthogonality} and Minkowski's inequality give
\begin{align*}
&\int_{\mathbb{R}^m}|E_{m+n}F_\tau[f-f_j](x,y,t)|^2\rmd x\\
&\qquad\lesssim\frac{\|\rho\|_2^2}{\tau^m}\left(\int_{\mathbb R^n}\frac{1_{\{|y-z|\geq CT\}}(z)}{(1+|y-z|)^M}\left(\sum_k|f_k(z)|^2\right)^{1/2}\rmd z\right)^2.
\end{align*}
Interpolating these two bounds, summing over the disjoint $Q_j$, and applying Young's inequality, we obtain
\begin{align*}
&\sum_j\int_{\mathbb{R}^m\times Q_j}|E_{m+n}F_\tau[f-f_j](x,y,t)|^{p_0}\rmd x\rmd y\\
&\quad\lesssim\frac{L^{m(p_0/2-1)}\|\rho\|_2^2}{\tau^m}\int_{\mathbb R^n}\left(\int_{\mathbb R^n}\frac{1_{\{|y-z|\geq CT\}}(z)}{(1+|y-z|)^M}
\left(\sum_k|f_k(z)|^2\right)^{1/2}\rmd z\right)^{p_0}\rmd y\\
&\quad\lesssim\frac{L^{m(p_0/2-1)}\|\rho\|_2^2}{\tau^mT^{Np_0}}\int_{\mathbb{R}^n}\left(\sum_k|f_k(y)|^2\right)^{p_0/2}\rmd y.
\end{align*}

By Parseval's identity,
\begin{equation*}
\sum_k|f_k(y)|^2
\leq\|f\|_{L^\infty}^2,
\end{equation*}
and
\begin{equation*}
\frac{\|\rho\|_{L^2}^2}{\tau^m}\sum_k\|f_k\|_{L^2}^2
=\|F_\tau[f]\|_{L^2}^2.
\end{equation*}
Integrating over $[T,2T]$, we obtain
\begin{equation*}
\sum_j\int_T^{2T}\int_{\mathbb{R}^m\times Q_j}|E_{m+n}F_\tau[f-f_j](x,y,t)|^{p_0}\rmd x\rmd y\rmd t
\lesssim
\frac{L^{m(p_0/2-1)}}{T^{Np_0-1}}\|F_\tau[f]\|_{L^2}^2\|f\|_{L^\infty}^{p_0-2}.
\end{equation*}
Taking $N$ sufficiently large, this is bounded by the right-hand
side of \eqref{eq:time-interval}.
Combining this with
\eqref{eq:time-near} proves
\eqref{eq:time-interval}.

Finally, by the dyadic decomposition of $[0,R]$, \eqref{eq:bounded-time} and \eqref{eq:time-interval},
\begin{align*}
&\int_0^R\int_{\mathbb{R}^{m+n}}|E_{m+n}F_\tau(x,y,t)|^{p_0}\rmd x\rmd y\rmd t\\
&\quad\lesssim R^\varepsilon L^{m(p_0/2-1)}\left(1+\sum_{j\geq0,\,2^j<R}2^{jn(p_0/2-1)}\right)\|F_\tau\|_{L^2}^2\|f\|_{L^\infty}^{p_0-2}.
\end{align*}
Since $n(p_0/2-1)>0$, we have
\begin{equation*}
\sum_{j\geq0, 2^j<R}2^{jn(p_0/2-1)}
\lesssim R^{n(p_0/2-1)}.
\end{equation*}
Thus
\begin{equation*}
\int_0^R\int_{\mathbb{R}^{m+n}}|E_{m+n}F_\tau(x,y,t)|^{p_0}\rmd x\rmd y\rmd t
\lesssim R^\varepsilon (L^{m}R^{n})^{p_0/2-1}\|F_\tau\|_{L^2}^2\|f\|_{L^\infty}^{p_0-2}.
\end{equation*}
By \eqref{eq:Ftau-Linfty} and the arbitrariness of $\varepsilon>0$, this proves \eqref{eq:Euclidean-mixed}.
\end{proof}

Hence, it suffices to establish \eqref{eq:local-reduction} to prove \eqref{eq:Euclidean-mixed}.

\section{Wave Packet Decomposition and $L^2$ Estimates}\label{sec:wave-packets}
In this section and Section \ref{sec:refined-decoupling}, $1\leq L\leq R$ may be chosen independently.
When $m=0$, we omit the periodic variables and indices, set $F_\tau[f]=\widehat f$, and replace $\|\rho\|_{L^2}^2/\tau^m$ by $1$.

Assume that $f$ is smooth in the periodic variables and Schwartz in the nonperiodic variables, with
$\supp\widehat f\subset \{(k,\xi)\in L^{-1}\mathbb{Z}^m \times\mathbb{R}^n: |k|\leq1,\ |\xi|\leq3\}$.
Fix a dyadic number $1\leq T\leq R$ and a ball $B_r(y_0,t_0)\subset\mathbb{R}^{n+1}$ with $1\leq r\lesssim T$ and $t_0\sim T$,
such that $t\sim T$ on its fixed enlargements. 

\subsection{Wave Packet Decomposition}

Wave packet decompositions in Fourier restriction theory go back to the work of C\'ordoba and Fefferman \cite{CF} and
were further developed by Bourgain \cite{B91} in the study of higher-dimensional restriction estimates.

We decompose $f$ in the nonperiodic variables at scale $r$, with time $t_0$.
In the nonperiodic frequency space, let $\Theta$ be a finitely overlapping cover of $B^n(0,4)$ by balls $\theta$ of radius $1/r^{1/2}$,
with centers $\xi_\theta$ on a $c/r^{1/2}$-lattice, where $c>0$ is sufficiently small.
Choose a nonnegative smooth partition of unity $\{\varphi_\theta\}_{\theta\in\Theta}$ satisfying
\begin{equation*}
\sum_{\theta\in\Theta}\varphi_\theta=1
\quad\text{on }B^n(0,4),
\qquad \supp\varphi_\theta\subset 2\theta,
\qquad \|D^\alpha\varphi_\theta\|_{L^\infty} \lesssim r^{|\alpha|/2}.
\end{equation*}

In the nonperiodic physical space, consider a finitely overlapping cover of $\mathbb{R}^{n}$ by balls of radius comparable to $r^{1/2}$,
centered at $v\in r^{1/2}\mathbb{Z}^n$.
Let $\{\psi_v\}$ be a nonnegative smooth partition of unity adapted to this cover, satisfying
\begin{equation*}
\sum_{v\in r^{1/2}\mathbb{Z}^n}\psi_v=1,
\qquad 0\leq\psi_v\leq1,
\qquad \supp\widehat{\psi_v} \subset B^n(0,c_1/r^{1/2}),
\end{equation*}
where $c_1>0$ is sufficiently small. Moreover, for every $N\geq0$ and multi-index $\alpha$,
\begin{equation}\label{eq:wp-decay}
|D^\alpha\psi_v(y)|
\lesssim \frac{1}{r^{|\alpha|/2}\left(1+{|y-v|}/{r^{1/2}}\right)^N}.
\end{equation}

Since the ball $B_r(y_0,t_0)$ is centered at time $t_0\sim T$, we decompose $\mathcal{E}_n(t_0;k)f_k$ at time $t_0$.
Note that
\begin{equation*}
\widehat{\mathcal{E}_n(t_0;k)f_k}(\xi)
=e(t_0\Phi(k,\xi))\widehat{f_k}(\xi).
\end{equation*}
For $k\in L^{-1}\mathbb{Z}^m$, $\theta\in\Theta$ and $v\in r^{1/2}\mathbb{Z}^n$, write label $\iota=(k,\theta,v)$.
We use this notation throughout unless the individual components need to be specified. Define
\begin{equation*}
f_\iota^{(r,t_0)}(\xi)
=\left(e(t_0\Phi(k,\cdot)) \widehat{f_k}\varphi_\theta\right) \ast\widehat{\psi_v}(\xi).
\end{equation*}
By the Fourier support conditions,
\begin{equation*}
\supp f_\iota^{(r,t_0)}
\subset
B^n\left(\xi_\theta,
(2+c_1)/r^{1/2}\right).
\end{equation*}
Moreover, the partitions of unity give
\begin{equation*}
\sum_{\theta,v}f_\iota^{(r,t_0)}(\xi)
=e(t_0\Phi(k,\xi))\widehat{f_k}(\xi),
\qquad \iota=(k,\theta,v).
\end{equation*}

For any collection $\mathcal W$ of labels $\iota$, define
\begin{equation*}
f_{\mathcal{W}}(x,y)
=\sum_{\iota\in \mathcal{W}}e(k\cdot x)\int_{\mathbb{R}^{n}}f_\iota^{(r,t_0)}(\xi)e(y\cdot\xi-t_0\Phi(k,\xi))\rmd \xi.
\end{equation*}
Consequently,
\begin{equation}\label{eq:wp-continuous}
F_\tau[f_{\mathcal W}](\eta,\xi)
=\sum_{\iota\in\mathcal W}\rho_\tau(\eta-k)e(-t_0\Phi(k,\xi))f_\iota^{(r,t_0)}(\xi).
\end{equation}
For each label $\iota=(k,\theta,v)$, write
\begin{equation*}
\mathcal{P}_\iota(x,y,t)
=E_{m+n}[\rho_\tau(\eta-k)e(-t_0\Phi(k,\xi))f_\iota^{(r,t_0)}(\xi)](x,y,t).
\end{equation*}
Hence
\begin{equation}\label{eq:wp-evolution}
E_{m+n}F_\tau[f_{\mathcal W}](x,y,t)
=\sum_{\iota\in\mathcal W}\mathcal{P}_\iota(x,y,t).
\end{equation}

Fix a sufficiently small $\varepsilon_0>0$, to be chosen in terms of the loss exponent in subsequent estimates,
and sufficiently large constants $C_1,C_2$. For each $\iota=(k,\theta,v)$, we associate the tube
\begin{equation*}
\Pi_\iota
=\left\{(y,t)\in\mathbb{R}^{n+1}:|t-t_0|\leq C_2r,|y-v+(t-t_0)\nabla_\xi \Phi(k,\xi_\theta)|\leq C_1r^{1/2+\varepsilon_0}\right\}.
\end{equation*}
Different labels are retained even when their corresponding tubes coincide.

\subsection{Localized $L^2$ Estimates}
We first establish weighted $L^2$ estimates needed for the wave packet decomposition.
Fix an integer $N_0>n+2$. For $B\geq1$ and $a_n\in\mathbb{R}^n$, define
\begin{equation*}
w_{B,a_n}(y)
=\frac{1}{\left(1+{|y-a_n|^2}/{B^2}\right)^{N_0}}.
\end{equation*}

\begin{lemma}\label{lem:weighted}
Let $\{b_\theta\}_{\theta\in\Theta}$ be smooth functions whose supports have uniformly bounded overlap. Suppose that
\begin{equation*}
\|D^\alpha b_\theta\|_{L^\infty}
\lesssim B^{|\alpha|},
\qquad |\alpha|\leq N_0.
\end{equation*}
Then, for $w=w_{B,a_n}$ or $w=1/w_{B,a_n}$, we have
\begin{equation}\label{eq:weighted-analysis}
\sum_{\theta\in\Theta}
\left\|(b_\theta\widehat g)^\vee\right\|_{L^2(w)}^2
\lesssim \|g\|_{L^2(w)}^2
\end{equation}
and
\begin{equation}\label{eq:weighted-synthesis}
\left\|\sum_{\theta\in\Theta}(b_\theta\widehat{g_\theta})^\vee\right\|_{L^2(w)}^2
\lesssim\sum_{\theta\in\Theta}\|g_\theta\|_{L^2(w)}^2.
\end{equation}
\end{lemma}

\begin{proof}
Denote $g_{a_n}(y)=g(y+a_n)$, then
\begin{align*}
\left\|(b_\theta\widehat{g})^{\vee}\right\|_{L^{2}(w_{B,a_n}^{\pm 1})}^2
&=\int_{\mathbb{R}^{n}}\left|\left(b_\theta\widehat{g}\right)^{\vee}(y)\right|^{2}\left(1+\frac{|y-a_n|^{2}}{B^{2}}\right)^{\mp N_0}\rmd y\\
&=\int_{\mathbb{R}^{n}}\left|\left(b_\theta\widehat{g}\right)^{\vee}(y+a_n)\right|^{2}\left(1+\frac{|y|^{2}}{B^{2}}\right)^{\mp N_0}\rmd y\\
&=\left\|\left(b_\theta \widehat{g_{a_n}}\right)^{\vee}\right\|_{L^{2}(w_{B,0}^{\pm 1})}^2.
\end{align*}
Thus, it suffices to consider the case $a_n=0$. Assume $w=1/w_{B, 0}$. By the binomial expansion,
\begin{equation*}
w(y)\sim \sum_{|\alpha|\leq N_0}\frac{y^{2\alpha}}{B^{2|\alpha|}}.
\end{equation*}
Hence, by Plancherel's identity and Leibniz's rule, we obtain
\begin{align*}
\sum_{\theta \in \Theta}\left\|\left(b_\theta \widehat{g}\right)^{\vee}\right\|_{L^{2}(w)}^2
&\sim \sum_{\theta \in \Theta}\int_{\mathbb{R}^{n}}\sum_{|\alpha|\leq N_0} \frac{1}{B^{2|\alpha|}}\left|y^{\alpha}\left(b_\theta\widehat{g}
\right)^{\vee}(y)\right|^{2}\rmd y\\
&=\sum_{|\alpha|\leq N_0}\frac{1}{B^{2|\alpha|}}\sum_{\theta \in \Theta}\int_{\mathbb{R}^{n}}\left| \left(D^{\alpha}\left(b_\theta
\widehat{g}\right)\right)^{\vee}(y)\right|^{2}\rmd y\\
&\sim \sum_{|\alpha|\leq N_0}\frac{1}{B^{2|\alpha|}}\sum_{\theta \in \Theta}\int_{\mathbb{R}^{n}}\left|\sum_{\beta\leq \alpha}
D^{\beta}b_\theta(\xi)\cdot \left(y^{\alpha-\beta}g\right)^{\wedge}(\xi)\right|^{2}\rmd \xi\\
&\lesssim \sum_{|\alpha|\leq N_0}\frac{1}{B^{2|\alpha|}}\sum_{\beta \leq \alpha}\sum_{\theta \in \Theta}\int_{\mathbb{R}^{n}}
\left|D^{\beta}b_\theta(\xi)\left(y^{\alpha-\beta}g\right)^{\wedge}(\xi)\right|^{2}\rmd \xi\\
&\lesssim \sum_{|\alpha|\leq N_0}\sum_{\beta \leq \alpha}
\int_{\mathbb{R}^{n}}\frac{y^{2(\alpha-\beta)}}{B^{2|\alpha-\beta|}}|g(y)|^{2}\rmd y\\
&=\sum_{|\gamma|\leq N_0}\sum_{|\beta|\leq N_0-|\gamma|}
\int_{\mathbb{R}^{n}}\frac{y^{2\gamma}}{B^{2|\gamma|}}|g(y)|^{2}\rmd y
\sim \|g\|_{L^{2}(w)}^2.
\end{align*}
This proves \eqref{eq:weighted-analysis}. The proof of \eqref{eq:weighted-synthesis} is similar, just noting that
\begin{equation*}
\sum_{\theta\in \Theta}\left|D^{\beta}b_\theta(\xi)\left(y^{\alpha-\beta}g_\theta\right)^{\wedge}(\xi)\right|^{2}
\leq \sum_{\theta\in \Theta}\left|D^{\beta}b_\theta(\xi)\right|^{2}
\sum_{\theta\in \Theta}\left|\left(y^{\alpha-\beta}g_\theta\right)^{\wedge}(\xi)\right|^{2}.
\end{equation*}
For $w=w_{B,0}$, duality gives
\begin{align*}
\sum_{\theta\in\Theta}\left\|(b_\theta\widehat g)^\vee\right\|_{L^2(w)}^2
&=\sup_{\sum_\theta\|h_\theta\|_{L^2}^2=1}
\left|\sum_\theta\left\langle(b_\theta\widehat g)^\vee,w^{1/2}h_\theta\right\rangle\right|^2\\
&=\sup_{\sum_\theta\|h_\theta\|_{L^2}^2=1}\left|\left\langle w^{1/2}g, \frac{1}{w^{1/2}}\sum_\theta
\left(\overline{b_\theta}\widehat{w^{1/2}h_\theta}\right)^\vee \right\rangle\right|^2\\
&\leq \|g\|_{L^2(w)}^2 \sup_{\sum_\theta\|h_\theta\|_{L^2}^2=1}
\left\| \frac{1}{w^{1/2}}\sum_\theta \left(\overline{b_\theta}\widehat{w^{1/2}h_\theta}\right)^\vee\right\|_{L^2}^2\\
&\lesssim\|g\|_{L^2(w)}^2 \sup_{\sum_\theta\|h_\theta\|_{L^2}^2=1} \sum_\theta \left\|\frac{w^{1/2}}{w^{1/2}}h_\theta\right\|_{L^2}^2\\
&=\|g\|_{L^2(w)}^2.
\end{align*}
Similarly, for $w=w_{B,0}$, \eqref{eq:weighted-synthesis} also holds.
\end{proof}

In particular, the functions $\varphi_\theta$ of the frequency partition satisfy the hypotheses of
Lemma \ref{lem:weighted} whenever $B\geq r^{1/2}$.
Moreover, if
\begin{equation*}
\supp\widehat{g_\theta} \subset B^n(\xi_\theta,C/r^{1/2})
\end{equation*}
for a fixed $C>0$, choosing smooth functions equal to one on these supports and supported
in fixed enlargements of the corresponding balls gives
\begin{equation}\label{eq:weighted-cap-synthesis}
\left\|\sum_{\theta\in\Theta}g_\theta\right\|_{L^2(w)}^2
\lesssim \sum_{\theta\in\Theta}\|g_\theta\|_{L^2(w)}^2,
\end{equation}
where $w=w_{B,a_n}$ or $1/w_{B,a_n}$.

For each $\theta\in\Theta$ and fixed $k$, define
\begin{equation*}
U_\theta(t)g
=\left(e\big(t(\Phi(k,\cdot)-\Phi(k,\xi_\theta)
-\nabla_\xi\Phi(k,\xi_\theta)\cdot(\cdot-\xi_\theta))\big)\widehat g\right)^\vee.
\end{equation*}
If $|t_0|\lesssim T$, $1\leq r\leq T$, $B\geq T/r^{1/2}$, and $\supp\widehat g$ lies in a fixed enlargement of $\theta$, then
\begin{equation}\label{eq:centered-propagation}
\|U_\theta(\pm t_0)g\|_{L^2(w_{B,a_n})}
\lesssim \|g\|_{L^2(w_{B,a_n})}.
\end{equation}
Indeed, choose a smooth function $\chi_\theta$ equal to one on $\supp\widehat g$ and supported in a fixed enlargement of $\theta$,
satisfying $\|D^\alpha\chi_\theta\|_{L^\infty} \lesssim r^{|\alpha|/2}$. Set
\begin{equation*}
b_\theta(\xi)=
\chi_\theta(\xi)e\big(\pm t_0(\Phi(k,\xi)-\Phi(k,\xi_\theta)-\nabla_\xi\Phi(k,\xi_\theta)\cdot(\xi-\xi_\theta))\big).
\end{equation*}
Changing variables $\zeta=r^{1/2}(\xi-\xi_\theta)$, and using the chain rule, we obtain
\begin{equation*}
\|D^\alpha b_\theta\|_{L^\infty}
\lesssim
r^{|\alpha|/2}
\left(1+\frac{|t_0|}{r}\right)^{|\alpha|}.
\end{equation*}
Since $1\leq r\leq T$, $|t_0|\lesssim T$ and $B\geq T/r^{1/2}$, we have
\begin{equation*}
r^{1/2}\left(1+\frac{|t_0|}{r}\right)
\lesssim \frac{T}{r^{1/2}}\leq B.
\end{equation*}
Consequently,
\begin{equation*}
\|D^\alpha b_\theta\|_{L^\infty}
\lesssim B^{|\alpha|},
\qquad |\alpha|\leq N_0.
\end{equation*}
Thus $b_\theta$ satisfies the hypotheses of Lemma \ref{lem:weighted}. Since
$U_\theta(\pm t_0)g=(b_\theta\widehat g)^\vee$,
applying \eqref{eq:weighted-analysis} to $b_\theta$ gives \eqref{eq:centered-propagation}.

To control the wave packets, we introduce the following localized average $L^{2}$:
\begin{equation*}
M_{T,r}(f)^2
=\left(\frac{r^{1/2}}{T}\right)^n \sup_{a_n\in\mathbb{R}^n} \sum_k\int_{\mathbb{R}^n} w_{T/r^{1/2},a_n}(y)|f_k(y)|^2\rmd y.
\end{equation*}
By Parseval's identity and \eqref{eq:Ftau-Linfty}, we have
\begin{equation*}
M_{T,r}(f)
\lesssim \|f\|_{L^\infty}
\lesssim \|F_\tau[f]^\vee\|_{L^\infty}
\end{equation*}
and since $0\leq w_{T/r^{1/2},a_n}\leq 1$,
\begin{equation}\label{eq:M-L2}
M_{T,r}(f)^2
\leq \left(\frac{r^{1/2}}{T}\right)^n \frac{\tau^m}{\|\rho\|_{L^2}^2} \|F_\tau[f]\|_{L^2}^2.
\end{equation}
Furthermore, for $1\leq r'\leq r$,
\begin{equation}\label{eq:M-monotonicity}
M_{T,r'}(f)\lesssim M_{T,r}(f).
\end{equation}
To prove \eqref{eq:M-monotonicity},
set $B=T/r^{1/2}$, $B'=T/(r')^{1/2}\geq B$. For every $S\geq B$, a ball of radius $S$ can be covered by at most $C(S/B)^n$ balls of radius $B$.
Since $w_{B,a_n}$ is bounded below on $B^n(a_n,B)$, we have
\begin{equation*}
\sum_k\int_{B^n(a_n,S)}|f_k(y)|^2\rmd y
\lesssim \left(\frac SB\right)^n \sup_{z\in\mathbb{R}^n} \sum_k\int_{\mathbb{R}^n} w_{B,z}(y)|f_k(y)|^2\rmd y.
\end{equation*}
By decomposing $\mathbb{R}^n$ into dyadic annuli centered at $a_n$, we obtain
\begin{align*}
\sum_k\int_{\mathbb{R}^n} w_{B',a_n}(y)|f_k(y)|^2\rmd y
&\lesssim \sum_{j\geq0} \frac{1}{2^{2N_0j}} \sum_k\int_{B^n(a_n,2^{j+1}B')} |f_k(y)|^2\rmd y\\
&\lesssim \left(\frac{B'}{B}\right)^n \sum_{j\geq0} \frac{1}{2^{j(2N_0-n)}} \sup_{z\in\mathbb{R}^n}
\sum_k\int_{\mathbb{R}^n} w_{B,z}(y)|f_k(y)|^2\rmd y\\
&\lesssim \left(\frac{B'}{B}\right)^n \sup_{z\in\mathbb{R}^n} \sum_k\int_{\mathbb{R}^n} w_{B,z}(y)|f_k(y)|^2\rmd y,
\end{align*}
where the last inequality follows from $2N_0>n$. Taking the supremum over $a_n$ and multiplying by $1/(B')^{n}$, we obtain
\begin{equation*}
M_{T,r'}(f)^2
\lesssim \frac{1}{(B')^n}\left(\frac{B'}{B}\right)^n\sup_z\sum_k\int_{\mathbb{R}^n}w_{B,z}(y)|f_k(y)|^2\rmd y
=M_{T,r}(f)^2.
\end{equation*}

We also have the following estimates for wave packet selections.

\begin{lemma}\label{lem:packet-subsum}
Let $\mathcal W$ be any collection of the wave packets at $(r,t_0)$. For every $1\leq r'\leq r$, we have
\begin{equation}\label{eq:packet-average}
M_{T,r'}(f_{\mathcal W})
\lesssim M_{T,r}(f).
\end{equation}
Moreover, let $\{\mathcal W_i\}_{i\in I}$ be a finite family of subsets of the labels in the same wave packet decomposition
of $f$ at $(r,t_0)$, satisfying
\begin{equation*}
\sup_\iota \sum_{i\in I}1_{\mathcal W_i}(\iota)
\leq D.
\end{equation*}
Then
\begin{equation}\label{eq:packet-subsum-L2}
\sum_{i\in I} \|F_\tau[f_{\mathcal W_i}]\|_{L^2}^2
\lesssim D\|F_\tau[f]\|_{L^2}^2.
\end{equation}
\end{lemma}

\begin{proof}
For each $k,\theta$, set
\begin{equation*}
Q_{k,\theta}(y)
=\sum_{\substack{v\in r^{1/2}\mathbb{Z}^n\\(k,\theta,v)\in\mathcal W}}\psi_v(y).
\end{equation*}
By the spatial partition of unity, we have $0\leq Q_{k,\theta}\leq1$ with $\supp\widehat{Q_{k,\theta}}\subset B^n(0,c_1/r^{1/2})$.
Note that
\begin{equation*}
\mathcal{E}_n(t_0;k)g(y)
=e\big(t_0(\Phi(k,\xi_\theta)-\nabla_\xi\Phi(k,\xi_\theta)\cdot\xi_\theta)\big)U_\theta(t_0)g(y+t_0\nabla_\xi\Phi(k,\xi_\theta)).
\end{equation*}
Taking the $k$-th periodic Fourier coefficient of $f_{\mathcal{W}}$, we obtain
\begin{align*}
(\widehat{f_{\mathcal{W}}}(k,\cdot))^{\vee}
&=\sum_{\theta\in \Theta}\mathcal{E}_{n}(-t_0;k)\left(Q_{k,\theta}\mathcal{E}_n(t_0;k)\left(\varphi_\theta\widehat{f_k}\right)^\vee\right)\\
&=\sum_{\theta\in \Theta}U_\theta(-t_0)\left(Q_{k,\theta}(\cdot-t_0\nabla_\xi\Phi(k,\xi_\theta))U_\theta(t_0)(\varphi_\theta\widehat{f_k})^\vee\right).
\end{align*}
Fix $1\leq r'\leq r$ and set $B'=T/(r')^{1/2}\geq T/r^{1/2}$.
By \eqref{eq:weighted-cap-synthesis}, \eqref{eq:centered-propagation} and \eqref{eq:weighted-analysis}, we obtain
\begin{align*}
\sum_k\left\|(\widehat{f_{\mathcal{W}}}(k,\cdot))^{\vee}\right\|_{L^2(w_{B',a_n}\!)}^2
&\lesssim \sum_{k,\theta} \!\left\| U_\theta(-t_0)
\!\left(Q_{k,\theta}(\cdot\!-\!t_0\nabla_\xi\Phi(k,\xi_\theta)) U_\theta(t_0)(\varphi_\theta\widehat{f_k})^\vee \right)
\right\|_{L^2(w_{B',a_n}\!)}^2\\
&\lesssim \sum_{k,\theta} \left\| Q_{k,\theta}(\cdot-t_0\nabla_\xi\Phi(k,\xi_\theta)) U_\theta(t_0)
(\varphi_\theta\widehat{f_k})^\vee \right\|_{L^2(w_{B',a_n})}^2\\
&\leq \sum_{k,\theta} \left\| U_\theta(t_0) (\varphi_\theta\widehat{f_k})^\vee \right\|_{L^2(w_{B',a_n})}^2\\
&\lesssim \sum_{k,\theta} \|(\varphi_\theta\widehat{f_k})^\vee \|_{L^2(w_{B',a_n})}^2\\
&\lesssim \sum_k\|f_k\|_{L^2(w_{B',a_n})}^2.
\end{align*}
Multiplying by $((r')^{1/2}/T)^n$, taking the supremum over $a_n$, and applying \eqref{eq:M-monotonicity}, we obtain
\begin{equation*}
M_{T,r'}(f_{\mathcal W})
\lesssim M_{T,r'}(f)
\lesssim M_{T,r}(f).
\end{equation*}
Hence \eqref{eq:packet-average} follows.

For the second assertion, set
\begin{equation*}
Q_{i,k,\theta}(y)
=\sum_{\substack{v\in r^{1/2}\mathbb{Z}^n\\ (k,\theta,v)\in\mathcal W_i}} \psi_v(y).
\end{equation*}
Since $0\leq Q_{i,k,\theta}\leq1$,
the multiplicity assumption gives
\begin{equation*}
\sum_{i\in I}Q_{i,k,\theta}(y)^2
\leq\sum_{i\in I}Q_{i,k,\theta}(y)
=\sum_v\left(\sum_{i\in I}1_{\mathcal W_i}(k,\theta,v)\right)\psi_v(y)
\leq D.
\end{equation*}
By the bounded overlap of the enlarged frequency balls, and Plancherel's identity, we have
\begin{align*}
\sum_{i\in I}\sum_k\|(\widehat{f_{\mathcal{W}_i}}(k,\cdot))^{\vee}\|_{L^2}^2
&= \sum_{i\in I}\sum_k\left\|\sum_{\theta\in\Theta}Q_{i,k,\theta}
\left(\varphi_\theta e(t_0\Phi(k,\cdot))\widehat{f_k}\right)^\vee\right\|_{L^2}^2\\
&\lesssim \sum_{i\in I}\sum_{k,\theta}\left\|Q_{i,k,\theta}\left(\varphi_\theta e(t_0\Phi(k,\cdot))\widehat{f_k}\right)^\vee\right\|_{L^2}^2\\
&\leq D\sum_{k,\theta}\left\|\left(\varphi_\theta e(t_0\Phi(k,\cdot))\widehat{f_k}\right)^\vee\right\|_{L^2}^2\\
&\lesssim D\sum_k\|f_k\|_{L^2}^2.
\end{align*}
Finally, applying \eqref{eq:Ftau-L2}
to $f_{\mathcal W_i}$ and $f$, we obtain \eqref{eq:packet-subsum-L2}.
\end{proof}

\subsection{Wave Packet $L^2$ Estimates}
The localized average gives an $L^{2}$ bound for wave packets with nearby centers.

\begin{lemma}\label{lem:packet-L2}
Let $\mathcal W$ be any collection of wave packets at $(r,t_0)$.
For every $a_n\in\mathbb{R}^n$ and fixed $C_3>0$, we have
\begin{equation}\label{eq:packet-local-L2}
\sum_{\substack{\iota\in\mathcal W\\|v+t_0\nabla_\xi\Phi(k,\xi_\theta)-a_n|
\leq C_3T/r^{1/2}}}\|f_\iota^{(r,t_0)}\|_{L^2}^2
\lesssim \left(\frac{T}{r^{1/2}}\right)^n M_{T,r}(f)^2.
\end{equation}
In particular, suppose that $\mathcal W$ satisfies
\begin{equation*}
\bar\beta
\leq\frac{\|f_\iota^{(r,t_0)}\|_{L^2}}{r^{n/4}}
<2\bar\beta,
\qquad \iota\in\mathcal W.
\end{equation*}
Set
\begin{equation*}
P_r
=\sup_{a_n\in\mathbb{R}^n}\#\left\{\iota\in\mathcal W:|v+t_0\nabla_\xi\Phi(k,\xi_\theta)-a_n|\leq C_3T/r^{1/2}\right\}.
\end{equation*}
Then
\begin{equation}\label{eq:directional-bound}
\bar\beta^2P_r
\lesssim \left(\frac Tr\right)^n M_{T,r}(f)^2.
\end{equation}
\end{lemma}

\begin{proof}
Set $B=T/r^{1/2}\geq r^{1/2}$. For every $\theta\in\Theta$, $a_n\in\mathbb{R}^n$ and $v$ obeying $|v+t_0\nabla_\xi\Phi(k,\xi_\theta)-a_n|\leq C_3B$,
triangular inequality yields
\begin{equation*}
1+\frac{|y+t_0\nabla_\xi\Phi(k,\xi_\theta)-a_n|}{B}
\lesssim 1+\frac{|y-v|}{r^{1/2}}.
\end{equation*}
Thus, taking the decay order in \eqref{eq:wp-decay} sufficiently large, one gives
\begin{align*}
\sum_{\substack{v\in r^{1/2}\mathbb{Z}^n\\ |v+t_0\nabla_\xi\Phi(k,\xi_\theta)-a_n|\leq C_3B}}|\psi_v(y)|^2
&\lesssim \sum_{v\in r^{1/2}\mathbb{Z}^n}\frac{w_{B,a_n}(y+t_0\nabla_\xi\Phi(k,\xi_\theta))}{\left(1+{|y-v|}/{r^{1/2}}\right)^{N}}\\
&\lesssim w_{B,a_n}(y+t_0\nabla_\xi\Phi(k,\xi_\theta)).
\end{align*}
By the definition of $f_\iota^{(r,t_0)}$ and Plancherel's identity, we obtain
\begin{align*}
&\sum_{\substack{\iota\in\mathcal W\\ |v+t_0\nabla_\xi\Phi(k,\xi_\theta)-a_n|\leq C_3B}}\|f_\iota^{(r,t_0)}\|_{L^2}^2\\
&\quad=\sum_{k,\theta}\int_{\mathbb{R}^n}
\sum_{\substack{v:\,(k,\theta,v)\in\mathcal W\\|v+t_0\nabla_\xi\Phi(k,\xi_\theta)-a_n|\leq C_3B}}|\psi_v(y)|^2
\left|\left(e(t_0\Phi(k,\cdot))\varphi_\theta\widehat{f_k}\right)^\vee(y)\right|^2\rmd y\\
&\quad\lesssim\sum_{k,\theta}\int_{\mathbb{R}^n}w_{B,a_n}(y+t_0\nabla_\xi\Phi(k,\xi_\theta))
\left|\left(e(t_0\Phi(k,\cdot))\varphi_\theta\widehat{f_k}\right)^\vee(y)\right|^2\rmd y.
\end{align*}
Since
\begin{equation*}
\left(e(t_0\Phi(k,\cdot))\varphi_\theta\widehat{f_k}\right)^\vee(y)
= e\big(t_0(\Phi(k,\xi_\theta)-\nabla_\xi\Phi(k,\xi_\theta)\cdot\xi_\theta)\big)U_\theta(t_0)(\varphi_\theta\widehat{f_k})^\vee(y+t_0\nabla_\xi\Phi(k,\xi_\theta)),
\end{equation*}
changing variables $z=y+t_0\nabla_\xi\Phi(k,\xi_\theta)$ and applying \eqref{eq:centered-propagation} and \eqref{eq:weighted-analysis} yield
\begin{align*}
\sum_{\substack{\iota\in\mathcal W\\|v+t_0\nabla_\xi\Phi(k,\xi_\theta)-a_n|\leq C_3B}}\|f_\iota^{(r,t_0)}\|_{L^2}^2
&\lesssim \sum_{k,\theta}\left\|U_\theta(t_0)(\varphi_\theta\widehat{f_k})^\vee \right\|_{L^2(w_{B,a_n})}^2\\
&\lesssim \sum_{k,\theta}\|(\varphi_\theta\widehat{f_k})^\vee\|_{L^2(w_{B,a_n})}^2\\
&\lesssim \sum_k\|f_k\|_{L^2(w_{B,a_n})}^2\\
&\leq B^nM_{T,r}(f)^2.
\end{align*}
This establishes \eqref{eq:packet-local-L2}.

Under the assumption, every term counted in $P_r$ satisfies
\begin{equation*}
\|f_\iota^{(r,t_0)}\|_{L^2}^2
\geq \bar\beta^2r^{n/2}.
\end{equation*}
Therefore, by \eqref{eq:packet-local-L2},
\begin{equation*}
\bar\beta^2r^{n/2}P_r
\lesssim B^nM_{T,r}(f)^2
=\left(\frac{T}{r^{1/2}}\right)^n M_{T,r}(f)^2.
\end{equation*}
Dividing by $r^{n/2}$ gives \eqref{eq:directional-bound}.
\end{proof}

Moreover, the full collection of wave packets satisfies
\begin{equation}\label{eq:packet-L2}
\frac{\|\rho\|_{L^2}^2}{\tau^m}\sum_\iota\|f_\iota^{(r,t_0)}\|_{L^2}^2
\lesssim \|F_\tau[f]\|_{L^2}^2.
\end{equation}
Indeed, for each $\iota=(k,\theta,v)$, \eqref{eq:wp-continuous} gives
\begin{equation*}
\|F_\tau[f_{\{\iota\}}]\|_{L^2}^2
=\int_{\mathbb{R}^{m+n}}|\rho_\tau(\eta-k)|^{2}|f_\iota^{(r,t_0)}(\xi)|^{2}\rmd \eta\rmd \xi
=\frac{\|\rho\|_{L^{2}}^{2}}{\tau^{m}}\|f_\iota^{(r,t_0)}\|_{L^2}^2.
\end{equation*}
Applying \eqref{eq:packet-subsum-L2} with $D=1$ to any finite family of singleton label set $\{\iota\}$, 
and then taking the supremum over such families, proves \eqref{eq:packet-L2}.

\subsection{Coefficient Estimates}
We now derive pointwise bounds in terms of $M_{T,r}(f)$.

\begin{lemma}\label{lem:coefficient}
Let $f_k$ be the periodic Fourier coefficients of $f$.
For $t\sim T\geq1$ and $1\leq r\leq T$, we have
\begin{equation}\label{eq:coefficient-square}
\sup_{y\in\mathbb{R}^n} \left(\sum_k|\mathcal{E}_n(t;k)f_k(y)|^2\right)^{1/2}
\lesssim T^{n/2}M_{T,r}(f).
\end{equation}
Moreover, for every $\theta\in\Theta$,
\begin{equation}\label{eq:cap-square}
\sup_{y\in\mathbb{R}^n}
\left(\sum_{k}\left|\left(\varphi_\theta e(t\Phi(k,\cdot))\widehat{f_k}\right)^\vee(y)\right|^2\right)^{1/2}
\lesssim T^{n/2}M_{T,r}(f).
\end{equation}
\end{lemma}

\begin{proof}
Choose $\chi\in C_c^\infty(\mathbb{R}^n)$ equal to one on the Fourier supports of $f_k$, and set
\begin{equation*}
K_{t,k}(y)=\int_{\mathbb{R}^n}\chi(\xi)e(y\cdot\xi+t\Phi(k,\xi))\rmd\xi.
\end{equation*}
By Corollary \ref{cor:kernel}, for $t\sim T\geq1$, we have
\begin{equation*}
\sup_k|K_{t,k}(y)|\lesssim \frac{1}{T^{n/2}(1+|y|/T)^{2N_0}},
\qquad \sup_k\|K_{t,k}\|_{L^1}\lesssim T^{n/2}.
\end{equation*}
By H\"older's inequality, we obtain
\begin{align*}
\sum_k|\mathcal{E}_n(t;k)f_k(y)|^2
&=\sum_{k}\left|\int_{\mathbb{R}^{n}}f_k(z)K_{t,k}(y-z)\rmd z\right|^{2}\\
&\leq \sup_k\|K_{t,k}\|_{L^1} \int_{\mathbb{R}^n} \sup_k|K_{t,k}(y-z)|\sum_k|f_k(z)|^2\rmd z\\
&\lesssim \int_{\mathbb{R}^n} w_{T,y}(z)\sum_k|f_k(z)|^2\rmd z\\
&\leq T^nM_{T,1}(f)^2\\
&\lesssim T^nM_{T,r}(f)^2,
\end{align*}
where the last inequality follows from \eqref{eq:M-monotonicity}. This proves \eqref{eq:coefficient-square}.

For \eqref{eq:cap-square}, by the frequency partition, we have $\|\varphi_\theta^\vee\|_{L^1}\lesssim 1$. Since
\begin{equation*}
\left(\varphi_\theta e(t\Phi(k,\cdot))\widehat{f_k}\right)^\vee=\varphi_\theta^\vee*\mathcal{E}_n(t;k)f_k,
\end{equation*}
Minkowski's inequality gives
\begin{align*}
\left(\sum_{k\in\mathcal K}\left|\left(\varphi_\theta e(t\Phi(k,\cdot))\widehat{f_k}\right)^\vee(y)\right|^2\right)^{1/2}
&\leq \int_{\mathbb{R}^n}|\varphi_\theta^\vee(y-z)|\left(\sum_{k\in\mathcal K}|\mathcal{E}_n(t;k)f_k(z)|^2\right)^{1/2}\rmd z\\
&\leq \|\varphi_\theta^\vee\|_{L^1}\sup_{z\in\mathbb{R}^n}\left(\sum_k|\mathcal{E}_n(t;k)f_k(z)|^2\right)^{1/2}\\
&\lesssim T^{n/2}M_{T,r}(f),
\end{align*}
which proves \eqref{eq:cap-square}.
\end{proof}

The wave packets satisfy the following decay estimates.

\begin{lemma}\label{lem:packet-decay}
For every $\iota=(k,\theta,v)$, $|t|\lesssim R$ and $|t-t_0|\lesssim r$, we have
\begin{equation}\label{eq:packet-decay}
|\mathcal P_\iota(x,y,t)|
\lesssim\frac{T^{n/2}M_{T,r}(f)}
{(1+\tau|x|)^N\big(1+|y-v+(t-t_0)\nabla_\xi\Phi(k,\xi_\theta)|/r^{1/2}\big)^N}
\end{equation}
for every $N\geq0$.
Furthermore, for $|t|\lesssim R$,
\begin{equation}\label{eq:packet-infty}
|\mathcal P_\iota(x,y,t)|
\lesssim\frac{\|f_\iota^{(r,t_0)}\|_{L^2}}{r^{n/4}(1+\tau|x|)^N}.
\end{equation}
\end{lemma}

\begin{proof}
By the definition of $f_\iota^{(r,t_0)}$, \eqref{eq:cap-square} and \eqref{eq:wp-decay},
we have
\begin{equation}\label{eq:packer-decay1}
|(f_\iota^{(r,t_0)})^\vee(y)|
= |\psi_v(y)|\left|\left(\varphi_\theta e(t_0\Phi(k,\cdot))\widehat{f_k}\right)^\vee(y)\right|
\lesssim\frac{T^{n/2}M_{T,r}(f)}{\left(1+|y-v|/r^{1/2}\right)^N},
\end{equation}
for every integer $N\geq0$. 
Changing variables $\eta=k+\tau u$, we obtain
\begin{equation}\label{eq:packer-decay2}
\mathcal P_\iota(x,y,t)
=e(k\cdot x)\int_{\mathbb R^n}f_\iota^{(r,t_0)}(\xi)e(y\cdot\xi+(t-t_0)\Phi(k,\xi))A_{\tau,k,t}(x,\xi)\rmd\xi,
\end{equation}
where
\begin{equation*}
A_{\tau,k,t}(x,\xi)
=\int_{\mathbb R^m}\rho(u)e\big(\tau x\cdot u+t(\Phi(k+\tau u,\xi)-\Phi(k,\xi))\big)\rmd u.
\end{equation*}

Choose $\chi\in C_c^\infty(\mathbb R^n)$ equal to one on $B^n(0,2+c_1)$, so that $\chi(r^{1/2}(\xi-\xi_\theta))=1$ on $\supp f_\iota^{(r,t_0)}$. 
Define
\begin{equation*}
K_{t-t_0,k,\theta}(y)
=\int_{\mathbb R^n} \chi(r^{1/2}(\xi-\xi_\theta)) e(y\cdot\xi+(t-t_0)\Phi(k,\xi))\rmd\xi.
\end{equation*}
Let $\chi_0\in C_c^\infty(\mathbb R^n)$ be a fixed function equal to one on $\supp\chi(r^{1/2}(\cdot-\xi_\theta))$ for all $\theta\in \Theta$. Since
\begin{equation*}
t(\Phi(k+\tau u,\xi)-\Phi(k,\xi))
=t\tau\int_0^1 u\cdot\nabla_\eta\Phi(k+s\tau u,\xi)\rmd s
\end{equation*}
and $|t|\tau\lesssim1$, the functions $\rho(u)\chi_0(\xi)e\big(t(\Phi(k+\tau u,\xi)-\Phi(k,\xi))\big)$ 
form a uniformly bounded family in $C_c^\infty(\mathbb R^{m+n})$. The uniform Schwartz bounds give, for every $N,J\geq 0$,
\begin{equation*}
|A_{\tau,k,t}(x,\xi)|\lesssim \frac{1}{(1+\tau|x|)^{N}}, \qquad |(\chi_0 A_{\tau,k,t}(x,\cdot))^{\vee}(z)|\lesssim \frac{1}{(1+\tau|x|)^{N}(1+|z|)^{J}},
\end{equation*}
where the first estimate holds on the relevant frequency supports. In particular, for each $N,N_1\geq0$,
\begin{equation*}
\int_{\mathbb R^n}(1+|z|)^{N_1}\left|(\chi_0 A_{\tau,k,t}(x,\cdot))^\vee(z)\right|\rmd z
\lesssim\frac1{(1+\tau|x|)^N}.
\end{equation*}
By \eqref{eq:packer-decay2}, we have
\begin{equation*}
\mathcal P_\iota(x,y,t)
=e(k\cdot x) \big((f_\iota^{(r,t_0)})^\vee*K_{t-t_0,k,\theta}*(\chi_0 A_{\tau,k,t}(x,\cdot))^\vee\big)(y).
\end{equation*}
By Corollary \ref{cor:kernel}, we have
\begin{equation*}
|K_{t-t_0,k,\theta}(y)|
\lesssim \frac1{r^{n/2}\big(1+|y+(t-t_0)\nabla_\xi\Phi(k,\xi_\theta)|/r^{1/2}\big)^{N_1}}.
\end{equation*}
Since $r\geq1$, we have
\begin{equation*}
1+\frac{|y-z+(t-t_0)\nabla_\xi\Phi(k,\xi_\theta)|}{r^{1/2}}
\leq \left(1+\frac{|y-z-w+(t-t_0)\nabla_\xi\Phi(k,\xi_\theta)|}{r^{1/2}}\right)(1+|w|).
\end{equation*}
Consequently, by the preceding estimate, we obtain
\begin{align*}
&|\mathcal P_\iota(x,y,t)|\\
&\quad\leq \int_{\mathbb R^n}\int_{\mathbb R^n} |(f_\iota^{(r,t_0)})^\vee(z)||K_{t-t_0,k,\theta}(y-z-w)|
\left|(\chi_0 A_{\tau,k,t}(x,\cdot))^\vee(w)\right|\rmd w\rmd z\\
&\quad\lesssim \frac1{r^{n/2}}
\int_{\mathbb R^n}\frac{|(f_\iota^{(r,t_0)})^\vee(z)|}{\big(1+|y-z+(t-t_0)\nabla_\xi\Phi(k,\xi_\theta)|/r^{1/2}\big)^{N_1}}\\
&\qquad\qquad\times \int_{\mathbb R^n}(1+|w|)^{N_1}\left|(\chi_0 A_{\tau,k,t}(x,\cdot))^\vee(w)\right|\rmd w\rmd z\\
&\quad\lesssim \frac1{r^{n/2}(1+\tau|x|)^N}\int_{\mathbb R^n}\frac{|(f_\iota^{(r,t_0)})^\vee(z)|}
{\big(1+|y-z+(t-t_0)\nabla_\xi\Phi(k,\xi_\theta)|/r^{1/2}\big)^{N_1}}\rmd z.
\end{align*}
Choosing $N_1>N+n$ and applying \eqref{eq:packer-decay1} proves \eqref{eq:packet-decay}.

Clearly, 
\begin{equation*}
|A_{\tau,k,t}(x,\xi)|
\lesssim \frac{1}{(1+\tau|x|)^{N}}. 
\end{equation*}
For \eqref{eq:packet-infty}, \eqref{eq:packer-decay2} and H\"older's inequality give
\begin{equation*}
|\mathcal P_\iota(x,y,t)|
\lesssim \frac{\|f_\iota^{(r,t_0)}\|_{L^1}}{(1+\tau|x|)^N}
\lesssim \frac{\|f_\iota^{(r,t_0)}\|_{L^2}}{r^{n/4}(1+\tau|x|)^N}.
\end{equation*}
\end{proof}

The pointwise estimates extend to $E_{m+n}F_\tau[f]$ as follows.

\begin{corollary}\label{cor:terminal}
For $t\sim T$, $1\leq r\leq T$
and every $N\geq0$, we have
\begin{equation}\label{eq:terminal-pointwise}
|E_{m+n}F_\tau[f](x,y,t)|
\lesssim \frac{(L^mT^n)^{1/2}M_{T,r}(f)} {(1+\tau|x|)^N}.
\end{equation}
Consequently, for every $B_r(y_0,t_0)$ with $1\leq r\leq T$ and $t\sim T$ on $B_r(y_0,t_0)$, and every $p\geq2$,
\begin{equation}\label{eq:terminal}
\int_{\mathbb{R}^m\times B_r(y_0,t_0)}|E_{m+n}F_\tau[f](x,y,t)|^p\rmd x\rmd y\rmd t
\lesssim r(L^mT^n)^{p/2-1} \|F_\tau[f]\|_{L^2}^2 M_{T,r}(f)^{p-2}.
\end{equation}
\end{corollary}

\begin{proof}
By \eqref{eq:Ftau-extension} and the definition of $A_{\tau,k,t}$ with $\chi_0$ chosen as in the proof of Lemma \ref{lem:packet-decay}, we have
\begin{equation*}
E_{m+n}F_\tau[f](x,y,t)=\sum_k e(k\cdot x)\left(\mathcal{E}_n(t;k)f_k\ast (\chi_0 A_{\tau,k,t}(x,\cdot))^{\vee}\right)(y),
\end{equation*}
The uniform Schwartz bounds in that proof, Cauchy--Schwarz inequality, and \eqref{eq:coefficient-square} give
\begin{align*}
|E_{m+n}F_\tau[f](x,y,t)|
&\lesssim \frac{L^{m/2}}{(1+\tau|x|)^{N}}\int_{\mathbb{R}^{n}}\frac{\left(\sum_k|\mathcal{E}_n(t;k)f_k(y-z)|^{2}\right)^{1/2}}{(1+|z|)^{n+1}}\rmd z\\
&\lesssim \frac{(L^{m}T^{n})^{1/2}M_{T,r}(f)}{(1+\tau|x|)^{N}}.
\end{align*}
Hence \eqref{eq:terminal-pointwise} holds.

For $p\geq2$, using \eqref{eq:terminal-pointwise}
with $N=0$, we have
\begin{align*}
&\int_{\mathbb R^m\times B_r(y_0,t_0)}|E_{m+n}F_\tau[f](x,y,t)|^p\rmd x\rmd y\rmd t\\
&\quad\lesssim (L^mT^n)^{p/2-1}M_{T,r}(f)^{p-2}\int_{\mathbb R^m\times B_r(y_0,t_0)}|E_{m+n}F_\tau[f](x,y,t)|^2\rmd x\rmd y\rmd t.
\end{align*}
Since the projection of $B_r(y_0,t_0)$ onto the time axis has length at most $2r$, spatial $L^2$ unitarity gives
\begin{equation*}
\int_{\mathbb R^m\times B_r(y_0,t_0)}|E_{m+n}F_\tau[f](x,y,t)|^2\rmd x\rmd y\rmd t
\leq 2r\|F_\tau[f]\|_{L^2}^2.
\end{equation*}
Combining this with the previous estimate, we obtain \eqref{eq:terminal}.
\end{proof}

\section{Refined Decoupling}\label{sec:refined-decoupling}

In 2015, Bourgain and Demeter \cite{BD} proved the sharp $\ell^2$ decoupling theorem for compact $C^2$ hypersurfaces with positive definite second fundamental form.
Let $S\subset\mathbb{R}^{d+1}$ be such a hypersurface, and suppose that $\widehat g$ is supported in the $1/r$-neighborhood of $S$,
where $r\geq1$. Write $g=\sum_\theta g_\theta$, where $\theta$ ranges over finitely overlapping caps of diameter $1/r^{1/2}$ on $S$,
and each $\widehat{g_\theta}$ is supported in the $1/r$-neighborhood of $\theta$.
Then, for $2\leq p\leq 2(d+2)/d$ and every $\varepsilon>0$,
\begin{equation*}
\|g\|_{L^p}
\lesssim r^\varepsilon\left(\sum_\theta\|g_\theta\|_{L^p}^2\right)^{1/2}.
\end{equation*}

This estimate reflects the Fourier support information of the functions $g_\theta$,
but does not distinguish how they overlap in physical space. If their physical supports are disjoint,
an $\ell^p$ summation is available instead of the $\ell^2$ summation above.
This naturally raises the question of whether a stronger decoupling estimate can be obtained when the overlap in physical space is controlled.

Guth, Iosevich, Ou and Wang \cite{GIOW} established such a refinement and noted that the same result was independently obtained by Du and Zhang.
Related refined Strichartz estimates also appear in the work of Du and Zhang \cite{DZ}.
In this section, we establish a corresponding refined decoupling estimate adapted to the wave packet decomposition in Section \ref{sec:wave-packets}.

We begin with the refined decoupling theorem of Guth, Iosevich, Ou and Wang \cite[Theorem 4.2]{GIOW}. We first recall their notation.

Let $d\geq 1$ and $S\subset\mathbb{R}^{d+1}$ be a strictly convex $C^2$ hypersurface with Gaussian curvature $\sim 1$.
For $r\geq1$, decompose the $1/r$-neighborhood of $S$ into blocks $\vartheta$ of dimensions $1/r^{1/2}\times \cdots \times 1/r^{1/2}\times 1/r$.
For a sufficiently small $\delta>0$, let $\mathbb{T}_\vartheta$ be a finitely overlapping cover of $\mathbb{R}^{d+1}$
by tubes $T$ of length $\sim r^{1+\delta}$ and radius $\sim r^{(1+\delta)/2}$, whose long axes are normal to $S$ at $\vartheta$.
Write $\mathbb{T}=\bigcup_\vartheta\mathbb{T}_\vartheta$.
Each $T\in\mathbb{T}$ belongs to $\mathbb{T}_\vartheta$ for a unique $\vartheta$, denoted by $\vartheta(T)$.

We say that a function $g$ is microlocalized to $(T,\vartheta(T))$ if $g$ is essentially supported in $2T$ and $\widehat g$ is essentially
supported in $2\vartheta(T)$.

\begin{proposition}[{\cite[Theorem 4.2]{GIOW}}]\label{thm:GIOW}
Let $2\leq p\leq 2(d+2)/d$. For every $\varepsilon>0$, choose $0<\delta\ll\varepsilon$.
Let $\mathbb{W}\subset \{T\in\mathbb{T}:T\cap B_r\neq\varnothing\}$ be a finite collection and set $W=\#\mathbb{W}$.
Suppose that $g=\sum_{T\in\mathbb{W}}g_T$, where each $g_T$ is microlocalized to $(T,\vartheta(T))$.
Suppose that $\|g_T\|_{L^p}$ is roughly constant among all the $T\in \mathbb{W}$.
Let $Y$ be a union of $r^{1/2}$-cubes in $B_r$, each intersecting at most $M$ tubes $T\in \mathbb{W}$. Then
\begin{equation*}
\|g\|_{L^p(Y)}
\lesssim r^\varepsilon \left(\frac{M}{W}\right)^{1/2-1/p} \left(\sum_{T\in\mathbb{W}}\|g_T\|_{L^p}^2\right)^{1/2}.
\end{equation*}
\end{proposition}

We now establish the following refined decoupling estimate.

\begin{theorem}\label{thm:refined-decoupling}
Let $m\geq 0$, $n\geq 1$, $0<\varsigma<1$, $R>1$, $1\leq L\leq R$ and $0<\tau\leq c/R$. Fix a dyadic number $1\leq T\leq R$ and
a ball $B_r(y_0,t_0)\subset\mathbb{R}^{n+1}$ with $R^\varsigma\leq r\leq T$ and $t_0\sim T$, such that $t\sim T$ on its fixed enlargements.
Let $f$ be smooth in the periodic variables and Schwartz in the Euclidean variables, with
\begin{equation*}
\supp\widehat f\subset
\{(k,\xi)\in L^{-1}\mathbb Z^m\times\mathbb R^n:|k|\leq1,\ |\xi|\leq3\}.
\end{equation*}
Let $\mathcal W$ be a finite collection of its wave packet labels $\iota=(k,\theta,v)$ at $(r,t_0)$, satisfying
\begin{equation*}
|v-y_0|\leq C_4r,
\qquad \frac{\|f_\iota^{(r,t_0)}\|_{L^2}}{r^{n/4}}\leq \bar\beta,
\qquad \iota\in\mathcal W,
\end{equation*}
where $C_4>0$ is a sufficiently large fixed constant.

Let $\mathcal Q$ be a finite family of $r^{1/2}$-balls with centers in $B_{C_4r}(y_0,t_0)$, and set $X=\cup_{B\in\mathcal Q}B$.
Fix a sufficiently large constant $C_{5}>0$ and suppose that $M\geq 1$ satisfies
\begin{equation*}
\sup_{B\in \mathcal{Q}}\#\{\iota\in \mathcal{W}: \Pi_\iota\cap C_{5}B\neq\varnothing\}\leq M
\end{equation*}
counting different labels separately even when their tubes coincide.

Write $p_j^{\ast}=2+4/(n+j)$ for $0\leq j\leq m$. Then, for every $\varepsilon,N>0$, we have
\begin{equation}\label{eq:refined-decoupling}
\begin{aligned}
&\int_{\mathbb{R}^m\times X}|E_{m+n}F_\tau[f_{\mathcal W}](x,y,t)|^{p^{\ast}_j}\rmd x\rmd y\rmd t\\
&\quad\lesssim r^{1+\varepsilon}\left(L^{m-j}\left(1+\frac{L}{r^{1/2}}\right)^{j}\right)^{2/(n+j)}
(\bar\beta^2M)^{2/(n+j)}\frac{\|\rho\|_{L^2}^2}{\tau^m}\sum_{\iota\in\mathcal W}\|f_\iota^{(r,t_0)}\|_{L^2}^2\\
&\qquad+ \frac{1}{R^N}\|F_\tau[f]\|_{L^2}^2M_{T,r}(f)^{p_j^\ast-2}.
\end{aligned}
\end{equation}
\end{theorem}

\begin{proof}
We may assume $f\neq0$ and $\mathcal{W}\neq \varnothing$.
Since the right hand side of \eqref{eq:refined-decoupling} is increasing in $\bar \beta$ and $M$, 
we may take them to be the maximal normalized wave packet norm and the maximal tube multiplicity (or $1$), respectively.

\medskip
\noindent\textbf{1. The case $j=m$ ($m\geq 1$).}
\par\nopagebreak

\smallskip
\noindent\textit{(a) Wave packet construction.}
\par\nopagebreak

Partition $\mathbb R^m$ into cubes $\theta'$ of side length $1/r^{1/2}$.
By the Fourier support of $f$, it suffices to consider cubes containing $k\in L^{-1}\mathbb{Z}^{m}$ with $|k|\leq 1$.
For each $(\theta',\theta,v)$, put $\mathcal{W}(\theta',\theta,v)=\{k\in\theta':(k,\theta,v)\in\mathcal W\}$.
We classify the triples $(\theta',\theta,v)$ according to the cardinality of $\mathcal W(\theta',\theta,v)$.
For every dyadic number $q\geq 1$, let $\mathcal{W}_q$ consist of all labels $\iota=(k,\theta,v)\in \mathcal{W}$ for which
$q\leq \#\mathcal{W}(\theta',\theta,v)<2q$, where $\theta'$ is the unique cube containing $k$.
Every nonempty class satisfies
\begin{equation*}
1\leq q
\leq \#\mathcal{W}(\theta',\theta,v)
\leq\#\{k: k\in \theta'\cap L^{-1}\mathbb{Z}^{m}\}
\leq (1+L/r^{1/2})^{m}.
\end{equation*}
As $r\geq R^{\varsigma}$, the number of nonempty classes $\lesssim \log R\lesssim r^{\varepsilon_1}$, for every $\varepsilon_1>0$.
By the triangle inequality,
\begin{equation*}
\left\|E_{m+n}F_\tau[f_{\mathcal{W}}]\right\|_{L^{p_m^\ast}(\mathbb{R}^{m}\times X)}\leq \sum_q \left\|E_{m+n}F_\tau[f_{\mathcal{W}_q}]\right\|_{L^{p_m^\ast}(\mathbb{R}^{m}\times X)}.
\end{equation*}
By the pigeonhole principle, there exists a class $\mathcal{W}_{q_0}$ such that
\begin{equation}\label{eq:refined-selection1}
\left\|E_{m+n}F_\tau[f_{\mathcal{W}}]\right\|_{L^{p_m^\ast}(\mathbb{R}^{m}\times X)}\lesssim r^{\varepsilon_1} \left\|E_{m+n}F_\tau[f_{\mathcal{W}_{q_0}}]\right\|_{L^{p_m^\ast}(\mathbb{R}^{m}\times X)}.
\end{equation}
In what follows, we write $\mathcal{W}=\mathcal{W}_{q_0}$ and $q=q_0$ for simplicity.

Choose a nonnegative smooth partition of unity $\{\psi'_{v'}\}_{v'\in r^{1/2}\mathbb Z^m}$ on $\mathbb R^m$,
analogous to \eqref{eq:wp-decay}, such that
$\sum_{v'}\psi'_{v'}=1$, $\supp \widehat{\psi'_{v'}}\subset B^m(0,c/r^{1/2})$, and, for every $N\geq0$,
\begin{equation*}
|\psi'_{v'}(x)|
\lesssim \frac{1}{(1+{|x-v'|}/{r^{1/2}})^{N}}.
\end{equation*}
For $\alpha=(\theta',\theta,v',v)$, define
\begin{equation*}
F_\alpha(\cdot,\xi)
=\widehat{\psi'_{v'}}* \left[\sum_{k\in\mathcal W(\theta',\theta,v)}\rho_\tau(\cdot-k)
e\big(t_0[\Phi(\cdot,\xi)-\Phi(k,\xi)]\big) f_{k,\theta,v}^{(r,t_0)}(\xi)\right].
\end{equation*}
By the definition of $\mathcal P_\iota$ and \eqref{eq:wp-continuous},
\begin{equation*}
F_\alpha^\vee(x,y)
=\psi'_{v'}(x)\sum_{k\in\mathcal W(\theta',\theta,v)}\mathcal P_{(k,\theta,v)}(x,y,t_0)
\end{equation*}
Summing first over $v'$ and then using \eqref{eq:wp-evolution}, we have
\begin{equation*}
\sum_\alpha F_\alpha(\eta,\xi)
=e(t_0\Phi(\eta,\xi))F_\tau[f_{\mathcal W}](\eta,\xi).
\end{equation*}
and hence
\begin{equation*}
\sum_\alpha E_{m+n}F_\alpha(x,y,t-t_0)
=E_{m+n}F_\tau[f_{\mathcal W}](x,y,t).
\end{equation*}
Clearly, $\supp F_\alpha \subset B^m(c_{\theta'},C/r^{1/2})\times B^n(\xi_\theta,C/r^{1/2})$,
where $c_{\theta'}$ is the center of $\theta'$.

Applying Lemma \ref{lem:packet-decay} at $t=t_0$, and using
$\#\mathcal W(\theta',\theta,v)<2q$ and $\tau r^{1/2}\lesssim1$, gives, for every $N_1,N_2\geq0$,
\begin{equation}\label{eq:packet-initial-decay}
|F_\alpha^\vee(x,y)|
\lesssim \frac{qT^{n/2}M_{T,r}(f)}{(1+\tau|v'|)^{N_2}}\frac1{(1+|x-v'|/r^{1/2})^{N_1}(1+|y-v|/r^{1/2})^{N_1}}.
\end{equation}
Here we use $1+\tau|v'|\lesssim(1+\tau|x|)(1+|x-v'|/r^{1/2})$.

For every $s\in\mathbb R$, we claim that
\begin{equation}\label{eq:packet-propagation}
|E_{m+n}F_\alpha(x,y,s)|
\lesssim \frac{qT^{n/2}M_{T,r}(f)}{(1+\tau|v'|)^{N_2}}\frac{(1+|s|/r)^{N_1}}
{(1+|(x,y)-(v',v)+s\nabla\Phi(c_{\theta'},\xi_\theta)|/r^{1/2})^{N_1}}.
\end{equation}
Indeed, choose $\chi\in C_c^\infty(\mathbb R^{m+n})$ equal to one on a sufficiently large ball centered at the origin, such that
$\chi\big(r^{1/2}(\eta-c_{\theta'}), r^{1/2}(\xi-\xi_\theta)\big)=1$ on $\supp F_\alpha$.
Define
\begin{equation*}
K_s(x,y)
=\int_{\mathbb R^{m+n}}\chi\big(r^{1/2}(\eta-c_{\theta'}),r^{1/2}(\xi-\xi_\theta)\big)
e\big(x\cdot\eta+y\cdot\xi+s\Phi(\eta,\xi)\big)\rmd\eta\rmd\xi.
\end{equation*}
Corollary~\ref{cor:kernel} gives
\begin{equation*}
|K_s(x,y)|
\lesssim \frac1{(r+|s|)^{(m+n)/2}}\frac1{\left(1+\dfrac{|(x,y)+s\nabla\Phi(c_{\theta'},\xi_\theta)|}{r^{1/2}(1+|s|/r)}\right)^N}
\end{equation*}
for every $N\geq0$. Since $E_{m+n}F_\alpha(x,y,s)=F_\alpha^{\vee}\ast K_s(x,y)$, choosing $N>N_1+m+n$, we combine this estimate
with \eqref{eq:packet-initial-decay} to obtain
\begin{align*}
&|E_{m+n}F_\alpha(x,y,s)|\\
&\quad\lesssim \frac{qT^{n/2}M_{T,r}(f)}{(1+\tau|v'|)^{N_2}(r+|s|)^{(m+n)/2}}\\
&\qquad\times \int_{\mathbb R^{m+n}}\frac{\rmd\tilde x\rmd\tilde y}{\big(1+|(\tilde x,\tilde y)-(v',v)|/r^{1/2}\big)^N
\left(1+\dfrac{|(x-\tilde x,y-\tilde y)+s\nabla\Phi(c_{\theta'},\xi_\theta)|}{r^{1/2}(1+|s|/r)}\right)^N}\\
&\quad\lesssim \frac{qT^{n/2}M_{T,r}(f)}{(1+\tau|v'|)^{N_2}}\frac1{
\left(1+\dfrac{|(x,y)-(v',v)+s\nabla\Phi(c_{\theta'},\xi_\theta)|}{r^{1/2}(1+|s|/r)}\right)^{N_1}}\\
&\quad\lesssim \frac{qT^{n/2}M_{T,r}(f)}{(1+\tau|v'|)^{N_2}}\frac{(1+|s|/r)^{N_1}}
{\left(1+\dfrac{|(x,y)-(v',v)+s\nabla\Phi(c_{\theta'},\xi_\theta)|}{r^{1/2}}\right)^{N_1}}.
\end{align*}
This proves \eqref{eq:packet-propagation}.

Partition $\mathbb R^m$ into cubes $U$ of side length $r$, with centers $x_U$. The set $X$ lies in a ball of radius $(C_4+1)r$, 
so each $U\times X$ has diameter $O(r)$. On $U\times X$, boundedness of $\nabla\Phi$ gives
\begin{equation*}
|x-v'+(t-t_0)\nabla_\eta\Phi(c_{\theta'},\xi_\theta)|
\gtrsim |v'-x_U|\qquad\text{if }|v'-x_U|>Cr,
\end{equation*}
for a sufficiently large fixed $C$.
Using \eqref{eq:packet-propagation} with decay order $2N_1$ and $|t-t_0|\lesssim r$, we obtain
\begin{align*}
&|E_{m+n}F_\alpha(x,y,t-t_0)|\\
&\qquad\lesssim \frac{qT^{n/2}M_{T,r}(f)}{(1+\tau|v'|)^{m+1}}
\frac1{(1+|v'-x_U|/r^{1/2})^{N_1}}\frac1{\left(1+|y-v+(t-t_0)\nabla_\xi\Phi(c_{\theta'},\xi_\theta)|/r^{1/2}\right)^{N_1}}.
\end{align*}
Since $\tau r^{1/2}\lesssim1$, summing over $|v'-x_U|>Cr$ gives
\begin{equation*}
\begin{aligned}
&\sum_{\substack{v'\in r^{1/2}\mathbb Z^m\\|v'-x_U|>Cr}}\frac1{(1+\tau|v'|)^{m+1}(1+|v'-x_U|/r^{1/2})^{N_1}}\\
&\qquad \lesssim \frac{1}{(1+\tau|x_U|)^{m+1}}\sum_{\substack{v'\in r^{1/2}\mathbb Z^m\\|v'-x_U|>Cr}}\frac{1}{(1+|v'-x_U|/ r^{1/2})^{N_1-m-1}}\\
&\qquad\lesssim \frac1{r^{(N_1-2m-1)/2}(1+\tau|x_U|)^{m+1}},
\end{aligned}
\end{equation*}
provided $N_1>2m+1$, where we used for $b>m$, $A\geq 1$,
\begin{equation*}
\sum_{\substack{\ell\in \mathbb{Z}^{m}\\|\ell-a|>A}}\frac{1}{(1+|\ell-a|)^{b}}\lesssim A^{m-b}.
\end{equation*}
Moreover, for $N_1>n$,
\begin{equation*}
\sum_{v\in r^{1/2}\mathbb Z^n}\frac1{\left(1+|y-v+(t-t_0)\nabla_\xi\Phi(c_{\theta'},\xi_\theta)|/r^{1/2}\right)^{N_1}}
\lesssim 1.
\end{equation*}
Since the number of pairs $(\theta',\theta)$ is $\lesssim r^{(m+n)/2}$, we conclude that
\begin{align*}
\left|\sum_{\alpha:\,|v'-x_U|>Cr}E_{m+n}F_\alpha(x,y,t-t_0)\right|
&\lesssim \frac{qT^{n/2}M_{T,r}(f)}{r^{(N_1-3m-n-1)/2}(1+\tau|x_U|)^{m+1}}\\
&\lesssim \frac{M_{T,r}(f)}{R^{N_2}(1+\tau|x_U|)^{m+1}},
\end{align*}
where the last inequality follows by taking $N_1$ sufficiently large, using $r\geq R^\varsigma$, $T\leq R$ and $q\lesssim L^{m}\leq R^{m}$.
Summing the left-hand side above for $x\in U$, we denote the result by $\operatorname{Err_{far}}$.
Note that the cubes $U$ are pairwise disjoint and $|X|\lesssim r^{n+1}$. For $2\leq p\leq p_m^\ast$, \eqref{eq:M-L2} gives
\begin{equation}\label{eq:rapid-remainder}
\begin{aligned}
\int_{\mathbb{R}^{m}\times X}|\operatorname{Err_{far}}(x,y,t)|^{p}\rmd x\rmd y\rmd t
&\lesssim \frac{|X|M_{T,r}(f)^{p}}{R^{N_2p}}\sum_{U}\frac{r^{m}}{(1+\tau |x_U|)^{(m+1)p}}\\
&\lesssim \frac{r^{n+1}}{R^{N_2p}}\frac{M_{T,r}(f)^{p}}{\tau^m}\\
&\lesssim \frac{1}{R^{N}}\|F_\tau[f]\|_{L^{2}}^{2}M_{T,r}(f)^{p-2},
\end{aligned}
\end{equation}
provided $N_2$ is chosen sufficiently large in terms of the prescribed $N$. 

We henceforth consider the indices $\alpha$ with $|v'-x_U|\leq Cr$.
Take a Schwartz function $\chi_1$ on $\mathbb{R}$, with $\supp \widehat{\chi_1}$ contained in a sufficiently small neighborhood of the origin,
such that $|\chi_1(s)|\geq 1$ for $|s|\leq C$, where $|t-t_0|\leq Cr$ on $X$. For each $\alpha$,
define 
\begin{equation*}
g_\alpha(x,y,t)=\chi_1((t-t_0)/r)E_{m+n}F_\alpha(x,y,t-t_0),
\end{equation*}
By the support of $F_\alpha$ and H\"older's inequality, we have
\begin{equation*}
\left\|E_{m+n}F_\alpha(\cdot,\cdot, u)\right\|_{L^{\infty}}
\leq |\supp F_\alpha|^{1/2}\|F_\alpha\|_{L^{2}} 
\lesssim \frac{\|F_\alpha\|_{L^{2}}}{r^{(m+n)/4}}.
\end{equation*}
On the other hand, Plancherel's identity gives
\begin{equation*}
\left\|E_{m+n}F_\alpha(\cdot, u)\right\|_{L^{2}}=\|F_\alpha\|_{L^{2}}.
\end{equation*}
Combining these estimates with $p_m^\ast-2=4/(m+n)$, we obtain
\begin{equation}\label{eq:packet-Lp}
\begin{aligned}
\|g_\alpha\|_{L^{p_m^\ast}}^{p_m^\ast}
&\lesssim \int_{\mathbb{R}}\left|\chi_1\left(\frac{t-t_0}{r}\right)\right|^{p_m^\ast}\|E_{m+n}F_\alpha(\cdot,\cdot,t-t_0)\|_{L^{\infty}}^{p_m^\ast-2}
\|E_{m+n}F_\alpha(\cdot,\cdot,t-t_0)\|_{L^{2}}^{2}\rmd t\\
&\lesssim r^{1-(m+n)/4(p_m^\ast-2)}\|F_\alpha\|_{L^{2}}^{p_m^\ast}=\|F_\alpha\|_{L^{2}}^{p_m^\ast}.
\end{aligned}
\end{equation}

\smallskip
\noindent\textit{(b) Selection of wave packets.}
\par\nopagebreak

Fix an arbitrarily large integer $N>m+1$ and set  $A_U=M_{T,r}(f)/(1+\tau |x_U|)^{N}$.
We first consider the following $\alpha$ with small $L^{p_m^\ast}$ norms. For each cube $U$, let
\begin{equation*}
\mathcal I_U^{\mathrm{small}}
=\left\{\alpha:\ |v'-x_U|\leq Cr,\quad\|g_\alpha\|_{L^{p_m^\ast}}<\frac{A_U}{R^{N}}\right\}.
\end{equation*}
By the bounded Fourier support and the condition $|v-y_0|\leq C_4r$, the numbers of possible $\theta'$, $\theta$, $v'$, and $v$ satisfy
$\#\{\theta'\}\lesssim r^{m/2}$, $\#\{\theta\}\lesssim r^{n/2}$, $\#\{v':|v'-x_U|\leq Cr\}\lesssim r^{m/2}$ and
$\#\{v:|v-y_0|\leq C_4r\}\lesssim r^{n/2}$.
Therefore, $\#\mathcal I_U^{\mathrm{small}}\lesssim r^{m+n}$.
By the triangle inequality,
\begin{equation*}
\left\|\sum_{\alpha\in \mathcal{I}_U^{\mathrm{small}}}g_\alpha\right\|_{L^{p_m^\ast}(U\times X)}\lesssim \frac{r^{m+n}}{R^{N}}A_U.
\end{equation*}
For $2\leq p\leq p_m^\ast$, H\"older's inequality and $|U\times X|\lesssim r^{m+n+1}$ yield
\begin{align*}
\int_{U\times X}\left|\sum_{\alpha\in \mathcal{I}_U^{\mathrm{small}}}g_\alpha(x,y,t)\right|^p \rmd x \rmd y \rmd t
&\leq |U\times X|^{1-p/p_m^\ast}\left\|\sum_{\alpha\in \mathcal{I}_U^{\mathrm{small}}}g_\alpha\right\|_{L^{p_m^\ast}(U\times X)}^p\\
&\lesssim \frac{r^{(m+n+1)(1-p/p_m^\ast)+(m+n)p}}{R^{Np}}A_U^p\\
&\lesssim R^{C_{m,n}-Np}A_U^p.
\end{align*}
Since the cubes $U$ have side length $r$ and $\tau r\lesssim 1$, we have, for $Np>m$,
\begin{align*}
\sum_U A_U^p
&=\sum_U\frac{M_{T,r}(f)^p}{(1+\tau|x_U|)^{Np}}\\
&\lesssim \frac{M_{T,r}(f)^p}{(\tau r)^{m}}
\lesssim \frac{M_{T,r}(f)^p}{\tau^{m}}.
\end{align*}
Consequently, by \eqref{eq:M-L2},
\begin{align*}
\sum_{U}\int_{U\times X}\left|\sum_{\alpha\in \mathcal{I}_U^{\mathrm{small}}}g_\alpha(x,y,t)\right|^p \rmd x \rmd y \rmd t 
&\lesssim \frac{R^{C_{m,n}-Np}}{\tau^{m}}M_{T,r}(f)^{p}\\
&\lesssim R^{C_{m,n}-Np}\|F_\tau[f]\|_{L^{2}}^{2}M_{T,r}(f)^{p-2}.
\end{align*}
Since $N$ may be chosen arbitrarily large and $p\geq2$,
the corresponding sum satisfies \eqref{eq:rapid-remainder} for every prescribed decay order.

We now only need to consider the indices
\begin{equation*}
\mathcal I_U
=\left\{\alpha:\ |v'-x_ U|\leq Cr,\quad\|g_\alpha\|_{L^{p_m^\ast}}\geq A_U/R^{N}\right\}.
\end{equation*}
We next determine the range of their $L^{p_m^\ast}$ norms.
By \eqref{eq:packet-Lp} and \eqref{eq:packet-initial-decay}, taking the decay orders sufficiently large, we have
\begin{equation*}
\|g_\alpha\|_{L^{p_m^\ast}}
\lesssim \|F_\alpha\|_{L^2}
\lesssim \frac{qT^{n/2}r^{(m+n)/4}M_{T,r}(f)}
{(1+\tau|v'|)^N}
\lesssim q T^{n/2}r^{(m+n)/4}A_U
\lesssim R^{C}A_U,
\end{equation*}
where $C>0$ is a fixed constant. Here we use $|v'-x_U|\leq Cr$ and $\tau r\lesssim1$, which imply
\begin{equation*}
1+\tau|v'|\sim1+\tau|x_U|.
\end{equation*}
Consequently, for every $\alpha\in\mathcal I_U$,
\begin{equation*}
\frac{A_U}{R^N}
\leq\|g_\alpha\|_{L^{p_m^\ast}}
\lesssim R^{C}A_U.
\end{equation*}
For each dyadic $\lambda\geq 1$, let
\begin{equation*}
\mathcal I_U(\lambda)
=\left\{\alpha\in\mathcal I_U:\frac{\lambda A_U}{R^N}\leq\|g_\alpha\|_{L^{p_m^\ast}}
<\frac{2\lambda A_U}{R^N}\right\}.
\end{equation*}
The number of nonempty classes is $O(\log R)$.

We further divide these classes according to their frequency and position indices.
The cubes $\theta'$ and the centers $\xi_\theta$ lie on regular lattices of spacing comparable to $1/r^{1/2}$.
Denote their lattice indices by $j'\in\mathbb Z^m$ and $j''\in\mathbb Z^n$.
By the Fourier support of $F_\alpha$, we may choose a sufficiently large fixed integer $K_0$ such that distinct frequency regions
in the same congruence class
\begin{equation*}
(j',j'')\pmod{K_0}
\end{equation*}
have disjoint fixed enlargements.
Since $(v',v)\in r^{1/2}\mathbb Z^{m+n}$, write $(v',v)=r^{1/2}\ell$, where $\ell\in\mathbb Z^{m+n}$. Choose
\begin{equation*}
K_r=\left\lceil Cr^{\varepsilon_0/2}\right\rceil
\end{equation*}
with $C$ sufficiently large, and classify $\ell$ modulo $K_r$.
Distinct positions in the same class are separated by at least $K_rr^{1/2}\gtrsim r^{(1+\varepsilon_0)/2}$.

For each dyadic $\lambda\geq 1$, $\gamma\in(\mathbb Z/K_0\mathbb Z)^{m+n}$ and $b\in(\mathbb Z/K_r\mathbb Z)^{m+n}$, define
\begin{equation*}
\mathcal I_U(\lambda,\gamma,b)
=\big\{\alpha\in\mathcal I_U(\lambda):(j',j'')\equiv\gamma\pmod{K_0},\ell\equiv b\pmod{K_r}\big\}.
\end{equation*}
These sets partition $\mathcal I_U$, and their number is bounded by
\begin{equation*}
K_0^{m+n}K_r^{m+n}(1+\log R)
\lesssim r^{\varepsilon_1+(m+n)\varepsilon_0/2}
\end{equation*}
for every $\varepsilon_1>0$.
By the triangle inequality,
\begin{equation*}
\left\|\sum_{\alpha\in\mathcal I_U}g_\alpha\right\|_{L^{p_m^\ast}(U\times X)}
\leq\sum_{\lambda,\gamma,b}
\left\|\sum_{\alpha\in\mathcal I_U(\lambda,\gamma,b)}g_\alpha\right\|_{L^{p_m^\ast}(U\times X)}.
\end{equation*}
By the pigeonhole principle, there exists $\lambda_U,\gamma_U,b_U$, setting $\mathcal{I}_U^{\ast}=\mathcal I_U(\lambda_U,\gamma_U,b_U)$, such that
\begin{equation}\label{eq:packet-selection}
\left\|\sum_{\alpha\in\mathcal I_U}g_\alpha\right\|_{L^{p_m^\ast}(U\times X)}^{p_m^\ast}
\lesssim r^{p_m^\ast\varepsilon_1+(m+n)p_m^\ast\varepsilon_0/2}\left\|\sum_{\alpha\in\mathcal{I}_U^{\ast}}g_\alpha\right\|_{L^{p_m^\ast}(U\times X)}^{p_m^\ast}.
\end{equation}

\smallskip
\noindent\textit{(c) Microlocalization.}
\par\nopagebreak

We verify the hypotheses of Proposition \ref{thm:GIOW} for the selected functions.
By construction, their $L^{p_m^\ast}$ norms are comparable, and distinct frequency regions within $\mathcal{I}_U^{\ast}$ have
disjoint fixed enlargements.

For each $\alpha=(\theta',\theta,v',v)$, associate the tube
\begin{equation*}
\mathcal T_\alpha=\big\{(x,y,t):
|t-t_0|\leq Cr^{1+\varepsilon_0},
|(x,y)-(v',v)+(t-t_0)\nabla\Phi(c_{\theta'},\xi_\theta)|
\leq Cr^{(1+\varepsilon_0)/2}\big\},
\end{equation*}
where $C$ is a sufficiently large fixed constant. Its long axis is normal to the graph of $\Phi$ at $(c_{\theta'},\xi_\theta)$.
For each fixed pair $(\theta',\theta)$, the associated tubes are pairwise disjoint by the spatial congruence class selection.

Let $(\omega,\sigma)$ be the Fourier variables dual to $((x,y),t)$, where $\omega=(\eta,\xi)$. By the definition of $g_\alpha$, we have
\begin{equation*}
\widehat{g_\alpha}(\omega,\sigma)
=e(-t_0\sigma)F_\alpha(\omega)r\widehat{\chi_1}\left(r(\sigma-\Phi(\omega))\right).
\end{equation*}
Since $\widehat{\chi_1}$ has compact support and $\supp F_\alpha$ lies in a ball of radius $C/r^{1/2}$, we obtain
\begin{equation*}
\supp\widehat{g_\alpha}
\subset\big\{(\omega,\sigma):
|\omega-(c_{\theta'},\xi_\theta)|\leq C/r^{1/2},
|\sigma-\Phi(\omega)|\leq C/r^{1}\big\}.
\end{equation*}
Thus $g_\alpha$ has the required Fourier localization.

It remains to verify that it is essentially supported in $\mathcal T_\alpha$. Fix $\delta_1=\varepsilon_0/100$.
On $|t-t_0|\leq r^{1+\delta_1}$, \eqref{eq:packet-propagation} implies, for every sufficiently large $J$,
\begin{align*}
&\left\|g_\alpha1_{\{|t-t_0|\leq r^{1+\delta_1},\,|(x,y)-(v',v)+(t-t_0)\nabla\Phi(c_{\theta'},\xi_\theta)|\geq r^{1/2+\varepsilon_0/4}\}}\right\|_{L^{p_m^\ast}}\\
&\quad\lesssim \frac{qT^{n/2}M_{T,r}(f)}{(1+\tau|v'|)^{m+1}}\frac{r^{((m+n)/2+1+\delta_1+(m+n)\varepsilon_0/4)/p_m^\ast}}{r^{J(\varepsilon_0/4-\delta_1)}}
\lesssim\frac{A_U}{R^{2N_2}}.
\end{align*}
The last inequality follows by choosing $J$ large, since $r\geq R^\varsigma$.

For $|t-t_0|> r^{1+\delta_1}$, the Fourier support of $F_\alpha$ and Plancherel's identity imply, uniformly in $t-t_0$,
\begin{equation*}
\|E_{m+n}F_\alpha(\cdot,\cdot,t-t_0)\|_{L^{p_m^\ast}}^{p_m^\ast}
\leq \|E_{m+n}F_\alpha(\cdot,\cdot,t-t_0)\|_{L^\infty}^{p_m^\ast-2}\|E_{m+n}F_\alpha(\cdot,\cdot,t-t_0)\|_{L^2}^2\\
\lesssim \frac{\|F_\alpha\|_{L^2}^{p_m^\ast}}{r}.
\end{equation*}
By the Schwartz decay of $\chi_1$, for every $N_2>0$,
\begin{equation*}
\left\|g_\alpha 1_{\{|t-t_0|>r^{1+\delta_1}\}}\right\|_{L^{p_m^\ast}}^{p_m^\ast}
\lesssim \|F_\alpha\|_{L^2}^{p_m^\ast}\int_{|s|>r^{\delta_1}}|\chi_1(s)|^{p_m^\ast}\rmd s
\lesssim \frac{\|F_\alpha\|_{L^2}^{p_m^\ast}}{r^{N_2}}.
\end{equation*}
Using $\|F_\alpha\|_{L^2}\lesssim R^{C}A_U$ and choosing $N_2$ sufficiently large, we obtain
\begin{equation*}
\left\|g_\alpha 1_{\{|t-t_0|>r^{1+\delta_1}\}}\right\|_{L^{p_m^\ast}}
\lesssim \frac{A_U}{R^{2N_2}}.
\end{equation*}

The two estimates show that $g_\alpha$ is essentially supported in
\begin{equation*}
|t-t_0|\leq r^{1+\delta_1},
\qquad
|(x,y)-(v',v)+(t-t_0)\nabla\Phi(c_{\theta'},\xi_\theta)|
<r^{1/2+\varepsilon_0/4},
\end{equation*}
which is contained in $\mathcal T_\alpha$. Since $\|g_\alpha\|_{L^{p_m^\ast}}\geq {A_U}/{R^N}$ for $\alpha\in\mathcal{I}_U^{\ast}$,
we conclude that
\begin{equation}\label{eq:packet-microlocal}
\|g_\alpha\|_{L^{p_m^\ast}(\mathbb R^{m+n+1}\setminus\mathcal T_\alpha)}
\lesssim \frac{\|g_\alpha\|_{L^{p_m^\ast}}}{R^N}.
\end{equation}
This verifies the remaining localization condition in Proposition \ref{thm:GIOW}.

\smallskip
\noindent\textit{(d) Refined decoupling.}
\par\nopagebreak

We estimate the multiplicity of the selected tubes. Partition $\mathbb R^{m+n+1}$ into pairwise disjoint cubes of side length $r^{1/2}$,
and let $\mathcal C_U$ consist of those cubes meeting $U\times X$. For each $Q\in\mathcal C_U$, there exists $B\in\mathcal Q$ such that
$Q\cap(U\times B)\neq\varnothing$. Since both $Q$ and $B$ have diameter $O(r^{1/2})$, every $(x,y,t)\in Q$ satisfies $(y,t)\in C_5B$,
where $C_5$ is the sufficiently large fixed constant.

Fix $Q\in \mathcal{C}_U$ and such $B$. If $Q\cap\mathcal T_\alpha\neq\varnothing$, choose $(x,y,t)$ in this intersection.
Then $|t-t_0|\lesssim r$ and, for each $k\in\mathcal W(\theta',\theta,v)$,
\begin{align*}
|y-v+(t-t_0)\nabla_\xi\Phi(k,\xi_\theta)|
&\leq |y-v+(t-t_0)\nabla_\xi\Phi(c_{\theta'},\xi_\theta)|+C|t-t_0||k-c_{\theta'}|\\
&\lesssim r^{(1+\varepsilon_0)/2}+r^{1/2}
\lesssim r^{1/2+\varepsilon_0}.
\end{align*}
Thus, by the definition of $\Pi_\iota$ and the choice of the fixed constants, $(y,t)\in\Pi_{(k,\theta,v)}\cap C_5B$.
Call a triple $(\theta',\theta,v)$ contributing to $Q$ if there exists $v'$ such that $(\theta',\theta,v',v)\in\mathcal{I}_U^{\ast}$
and $\mathcal T_\alpha\cap Q\neq\varnothing$. Since distinct triples $(\theta',\theta,v)$ contain disjoint sets of at least $q$ original labels,
the hypothesis on $M$ implies that
\begin{equation*}
\begin{aligned}
M
&\geq\#\{\iota\in\mathcal W:\Pi_\iota\cap C_5B\neq\varnothing\}\\
&\geq\sum_{\substack{(\theta',\theta,v)\\\text{contributing to }Q}}\#\mathcal W(\theta',\theta,v)\\
&\geq q\,\#\{(\theta',\theta,v):(\theta',\theta,v)\text{ contributes to }Q\}.
\end{aligned}
\end{equation*}
Hence at most $M/q$ distinct triples contribute to $Q$.
It remains to count the possible $v'$ for each fixed triple. Let $(x_Q,y_Q,t_Q)$ be the center of $Q$.
If $\mathcal T_\alpha\cap Q\neq\varnothing$, then
\begin{align*}
&|v'-x_Q-(t_Q-t_0)\nabla_\eta\Phi(c_{\theta'},\xi_\theta)|\\
&\qquad\leq |v'-x-(t-t_0)\nabla_\eta\Phi(c_{\theta'},\xi_\theta)|+|x-x_Q|+|t-t_Q||\nabla_\eta\Phi(c_{\theta'},\xi_\theta)|\\
&\qquad\lesssim r^{(1+\varepsilon_0)/2}.
\end{align*}
For a fixed $(\theta',\theta,v)$, all the selected $v'$ belong to the same congruence class modulo $K_r$.
Their separation is therefore at least $K_rr^{1/2}$. Since $K_r\sim r^{\varepsilon_0/2}$, the lattice-point bound gives
\begin{equation*}
\#\{v':(\theta',\theta,v',v)\in\mathcal{I}_U^{\ast}, \mathcal T_\alpha\cap Q\neq\varnothing\}
\lesssim \left(1+\frac{r^{(1+\varepsilon_0)/2}}{K_rr^{1/2}}\right)^m
\lesssim 1.
\end{equation*}
Combining the two estimates, we obtain
\begin{equation}\label{eq:tube-multiplicity}
\#\{\alpha\in\mathcal{I}_U^{\ast}:
\mathcal T_\alpha\cap Q\neq\varnothing\}
\lesssim\frac Mq,
\qquad Q\in\mathcal C_U.
\end{equation}

Let $Y_U$ be the union of cubes $Q\in\mathcal C_U$ meeting at least one $\mathcal T_\alpha$ with $\alpha\in\mathcal{I}_U^{\ast}$.
Since $(U\times X)\setminus Y_U$ lies outside $\mathcal T_\alpha$ for every $\alpha\in\mathcal{I}_U^{\ast}$, 
the triangle inequality and \eqref{eq:packet-microlocal} give
\begin{align*}
\left\|\sum_{\alpha\in\mathcal{I}_U^{\ast}}g_\alpha\right\|_{L^{p_m^\ast}((U\times X)\setminus Y_U)}
&\leq\sum_{\alpha\in\mathcal{I}_U^{\ast}}
\|g_\alpha\|_{L^{p_m^\ast}(\mathbb R^{m+n+1}\setminus\mathcal T_\alpha)}\\
&\lesssim \frac{1}{R^N}
\sum_{\alpha\in\mathcal{I}_U^{\ast}}\|g_\alpha\|_{L^{p_m^\ast}}
\lesssim \frac{A_U}{R^N},
\end{align*}
where we used $\#\mathcal{I}_U^{\ast}\lesssim R^C$ and $\|g_\alpha\|_{L^{p_m^\ast}}\lesssim R^CA_U$, and chose $N$ sufficiently large. 
Summing over $U$ as in \eqref{eq:rapid-remainder} shows that this term satisfies the required remainder bound.

If $Y_U$ is empty, the estimate follows directly. Otherwise, \eqref{eq:tube-multiplicity} implies $M/q\gtrsim1$.
The cubes forming $Y_U$ lie in a ball of radius $Cr$, which may be covered by a bounded number of balls of radius $r$.
On each such ball, we retain only the selected tubes meeting it.
The omitted terms are controlled by \eqref{eq:packet-microlocal}, as above.

The remaining functions have comparable $L^{p_m^\ast}$ norms and satisfy the hypotheses of Proposition \ref{thm:GIOW}.
Applying Proposition \ref{thm:GIOW} and using \eqref{eq:tube-multiplicity}, we obtain, for every $\varepsilon_1>0$,
\begin{equation}\label{eq:refined-tube-bound}
\begin{aligned}
\left\|\sum_{\alpha\in\mathcal{I}_U^{\ast}}g_\alpha\right\|_{L^{p_m^\ast}(U\times X)}^{p_m^\ast}
&\lesssim r^\varepsilon_1
\left(\frac{M}{q\#\mathcal{I}_U^{\ast}}\right)^{p_m^\ast/2-1}
\left(\sum_{\alpha\in\mathcal{I}_U^{\ast}}
\|g_\alpha\|_{L^{p_m^\ast}}^2\right)^{p_m^\ast/2}
+\frac{A_U^{p_m^\ast}}{R^N}\\
&\lesssim r^\varepsilon_1
\left(\frac Mq\right)^{2/(m+n)}
\sum_{\alpha\in\mathcal{I}_U^{\ast}}
\|g_\alpha\|_{L^{p_m^\ast}}^{p_m^\ast}
+\frac{A_U^{p_m^\ast}}{R^N}.
\end{aligned}
\end{equation}
Here we use the comparability of the selected norms and $p_m^\ast/2-1=2/(m+n)$.

For the remaining $L^2$ estimate, use \eqref{eq:packer-decay2} at $t=t_0$. The bound $|A_{\tau,k,t_0}(x,\xi)|\lesssim1$ and Plancherel's identity in $y$ give
\begin{equation*}
\int_{\mathbb R^n}|\mathcal P_{(k,\theta,v)}(x,y,t_0)|^2\rmd y
\lesssim\|f_{k,\theta,v}^{(r,t_0)}\|_{L^2}^2.
\end{equation*}
Therefore, by Cauchy--Schwarz inequality and the amplitude condition,
\begin{align*}
\|F_\alpha\|_2^2
&=\int_{\mathbb{R}^{m}}\int_{\mathbb{R}^{n}}|\psi'_{v'}(x)|^{2}\left|\sum_{k\in \mathcal{W}(\theta',\theta,v)}\mathcal{P}_{(k,\theta,v)}(x,y,t_0)\right|^{2}\rmd y\rmd x\\
&\lesssim q\int_{\mathbb R^m}|\psi'_{v'}(x)|^2 \sum_{k\in \mathcal{W}(\theta',\theta,v)}\int_{\mathbb{R}^{n}}\left|\mathcal{P}_{(k,\theta,v)}(x,y,t_0)\right|^{2}\rmd y \rmd x\\
&\lesssim q\int_{\mathbb R^m}|\psi'_{v'}(x)|^2\sum_{k\in\mathcal W(\theta',\theta,v)}\|f_{k,\theta,v}^{(r,t_0)}\|_{L^2}^2\rmd x\\
&\lesssim q r^{m/2}\sum_{k\in\mathcal W(\theta',\theta,v)}\|f_{k,\theta,v}^{(r,t_0)}\|_{L^2}^2
\lesssim q^2\bar\beta^2r^{(m+n)/2}.
\end{align*}
Combining this with \eqref{eq:packet-Lp}, we obtain
\begin{equation*}
\sum_{\alpha\in\mathcal{I}_U^{\ast}}\|g_\alpha\|_{L^{p_m^\ast}}^{p_m^\ast}
\lesssim rq^{4/(m+n)}\bar\beta^{4/(m+n)}\sum_{\alpha\in\mathcal{I}_U^{\ast}}\|F_\alpha\|_{L^2}^2.
\end{equation*}
Thus \eqref{eq:packet-selection} and \eqref{eq:refined-tube-bound} imply
\begin{equation}\label{eq:refined-local}
\begin{aligned}
&\int_{U\times X}\left|\sum_{\alpha\in\mathcal I_U}E_{m+n}F_\alpha(x,y,t-t_0)\right|^{p_m^\ast}\rmd x\rmd y\rmd t\\
&\quad\lesssim r^{p_m^\ast\varepsilon_1+(m+n)p_m^\ast\varepsilon_0/2}\left\|\sum_{\alpha\in\mathcal{I}_U^{\ast}}g_\alpha\right\|_{L^{p_m^\ast}(U\times X)}^{p_m^\ast}\\
&\quad\lesssim r^{1+(p_m^\ast+1)\varepsilon_1+(m+n)p_m^\ast\varepsilon_0/2}q^{2/(m+n)}(\bar\beta^2M)^{2/(m+n)}
\sum_{\substack{\alpha:\\|v'-x_U|\leq Cr}}\|F_\alpha\|_{L^2}^2+\frac{A_U^{p_m^\ast}}{R^N}.
\end{aligned}
\end{equation}
For each fixed $v'$, at most $O(1)$ cubes $U$ satisfy $|v'-x_U|\leq Cr$. Since $\sum_{v'}|\psi'_{v'}|^2\leq1$, 
Plancherel's identity and disjointness of the supports of $\rho_\tau(\cdot-k)$ give
\begin{equation}\label{eq:packet-sum-L2}
\begin{aligned}
\sum_\alpha\|F_\alpha\|_{L^2}^2
&= \sum_{\theta',\theta,v}\sum_{v'}\int_{\mathbb{R}^{m}}\int_{\mathbb{R}^{n}}|\psi'_{v'}(x)|^{2}
\left|\sum_{k\in \mathcal{W}(\theta',\theta,v)}\mathcal{P}_{(k,\theta,v)}(x,y,t_0)\right|^{2}\rmd y\rmd x\\
&\leq \sum_{\theta',\theta,v}\int_{\mathbb{R}^{m}}\int_{\mathbb{R}^{n}}
\left|\sum_{k\in \mathcal{W}(\theta',\theta,v)}\mathcal{P}_{(k,\theta,v)}(x,y,t_0)\right|^{2}\rmd y\rmd x\\
&= \sum_{\theta',\theta,v}\int_{\mathbb{R}^{m}}\int_{\mathbb{R}^{n}} 
\left|\sum_{k\in \mathcal{W}(\theta',\theta,v)} \rho_\tau(\eta-k)e(t_0(\Phi(\eta,\xi)-\Phi(k,\xi))) f^{(r,t_0)}_{k,\theta,v}(\xi)\right|^{2}\rmd \xi\rmd \eta\\
&= \frac{\|\rho\|_{L^2}^2}{\tau^m}\sum_{\iota\in\mathcal W}\|f_\iota^{(r,t_0)}\|_{L^2}^2.
\end{aligned}
\end{equation}

Finally, summing \eqref{eq:refined-local} over $U$, we first note that each $v'$ satisfies $|v'-x_U|\leq Cr$ for at most $O(1)$ cubes $U$.
Hence, by \eqref{eq:packet-sum-L2},
\begin{equation*}
\sum_U\sum_{\alpha:\,|v'-x_U|\leq Cr}\|F_\alpha\|_{L^2}^2
\lesssim\sum_\alpha\|F_\alpha\|_{L^2}^2
\lesssim\frac{\|\rho\|_{L^2}^2}{\tau^m}\sum_{\iota\in\mathcal W_{q}}\|f_\iota^{(r,t_0)}\|_{L^2}^2.
\end{equation*}
Together with $\sum_U A_U^{p_m^\ast}
\lesssim M_{T,r}(f)^{p_m^\ast}/\tau^{m}$
and \eqref{eq:M-L2}, this gives
\begin{align*}
&\sum_U\int_{U\times X}\left|\sum_{\alpha\in\mathcal I_U}E_{m+n}F_\alpha(x,y,t-t_0)\right|^{p_m^\ast}\rmd x\rmd y\rmd t\\
&\quad\lesssim r^{1+(p_m^\ast+1)\varepsilon_1+(m+n)p_m^\ast\varepsilon_0/2}q^{2/(m+n)}(\bar\beta^2M)^{2/(m+n)}
\frac{\|\rho\|_{L^2}^2}{\tau^m}\sum_{\iota\in\mathcal W_{q}}\|f_\iota^{(r,t_0)}\|_{L^2}^2\\
&\qquad+ \frac1{R^N}\|F_\tau[f]\|_{L^2}^2M_{T,r}(f)^{p_m^\ast-2}.
\end{align*}
Since all the preceding remainder terms satisfy \eqref{eq:rapid-remainder}, the pigeonhole selections 
\eqref{eq:refined-selection1} and \eqref{eq:packet-selection}, together with $q\lesssim(1+L/r^{1/2})^m$, yield
\begin{align*}
&\int_{\mathbb R^m\times X}
|E_{m+n}F_\tau[f_{\mathcal W}](x,y,t)|^{p_m^\ast}
\rmd x\rmd y\rmd t\\
&\quad\lesssim
r^{1+(2p_m^\ast+1)\varepsilon_1+(m+n)p_m^\ast\varepsilon_0/2}
\left(1+\frac{L}{r^{1/2}}\right)^{2m/(m+n)}(\bar\beta^2M)^{2/(m+n)}\frac{\|\rho\|_{L^2}^2}{\tau^m}
\sum_{\iota\in\mathcal W}\|f_\iota^{(r,t_0)}\|_{L^2}^2\\
&\qquad+
\frac1{R^N}\|F_\tau[f]\|_{L^2}^2M_{T,r}(f)^{p_m^\ast-2}.
\end{align*}
Choosing $\varepsilon_1>0$ and $\varepsilon_0>0$ sufficiently small
so that
\begin{equation*}
(2p_m^\ast+1)\varepsilon_1+\frac{(m+n)p_m^\ast\varepsilon_0}{2}<\varepsilon,
\end{equation*}
we obtain \eqref{eq:refined-decoupling} for $j=m$ for $m\geq 1$. 

\medskip
\noindent\textbf{2. The case $j=0$ ($m\geq1$).}
\par\nopagebreak

We now prove the case $j=0$ for $m\geq 1$. For each fixed $k\in L^{-1}\mathbb{Z}^{m}$ and $|\eta-k|\leq \tau$, the integral
\begin{equation*}
\int_{\mathbb R^n}f_\iota^{(r,t_0)}(\xi) e\bigl(y\cdot\xi+t\Phi(\eta,\xi)-t_0\Phi(k,\xi)\bigr)\rmd\xi
\end{equation*}
is associated with a tube
\begin{equation*}
\Pi'_\iota=\left\{(y,t):
\begin{gathered}
|t-t_0|\leq C'r,\\
|y-v+(t-t_0)\nabla_\xi\Phi(k,\xi_\theta)
+t(\nabla_\xi\Phi(\eta,\xi_\theta)-\nabla_\xi\Phi(k,\xi_\theta))|\leq C''r^{1/2+\varepsilon_0}
\end{gathered}\right\}.
\end{equation*}
Since
\begin{equation*}
|t|\left|\nabla_\xi\Phi(\eta,\xi_\theta)-\nabla_\xi\Phi(k,\xi_\theta)\right|
\lesssim R\tau\lesssim 1,
\end{equation*}
we have $\Pi'_\iota\subset \Pi_\iota$ by choosing the fixed constants appropriately. 
Thus the multiplicity in the scalar argument is controlled by the original tubes, as in \eqref{eq:tube-multiplicity}.
Define $\mathcal{W}_k=\{\iota\in \mathcal{W}: \iota=(k,\theta,v)\}$ and
\begin{equation*}
A_k(\eta,y,t)=
\sum_{\iota\in\mathcal W_k}\int_{\mathbb R^n}f_\iota^{(r,t_0)}(\xi)e\bigl(y\cdot\xi+t\Phi(\eta,\xi)-t_0\Phi(k,\xi)\bigr)\rmd\xi.
\end{equation*}
The preceding argument applies in $\mathbb{R}^{n+1}$ with the periodic partitions omitted and $\rho_\tau=1$, $q=1$, we have 
\begin{equation}\label{eq:refinedj01}
\int_{X}\left|A_k(\eta,y,t)\right|^{p_0^{\ast}}\rmd y\rmd t\\
\lesssim r^{1+\varepsilon}(\bar \beta^{2}M)^{2/n}\sum_{\iota\in \mathcal{W}_k}\|f_\iota^{(r,t_0)}\|_{L^{2}}^{2}+\frac{\|f_k\|_{L^{2}}^{2}M_{T,r}(f)^{p_0^{\ast}-2}}{R^{N}}.
\end{equation}
Partition $\mathbb{R}^{n+1}$ into disjoint cubes of side length $r^{1/2}$ and let $\mathcal{C}$ consist of those cubes meeting $X$.
For each $Q\in \mathcal{C}$, choose $B_Q\in \mathcal{Q}$ with $Q\cap B_Q\neq \varnothing$ and set
\begin{equation*}
M_k(Q)=\#\{\iota\in \mathcal{W}_k:\Pi_\iota \cap C_5 B_Q\neq \varnothing\}.
\end{equation*}
Then $\sum_kM_k(Q)\leq M$. For dyadic $\mu\geq1$, put
\begin{equation*}
\mathcal C_k^\mu=\{Q\in\mathcal C:\mu\leq M_k(Q)<2\mu\}.
\end{equation*}
There are $O(\log R)$ nonempty classes, since $M\lesssim L^mr^n$.
On the union of the cubes in $\mathcal C_k^\mu$, the scalar argument giving \eqref{eq:refinedj01} applies with multiplicity $2\mu$.
Using that estimate with loss $\varepsilon/2$, and $r\geq R^\varsigma$, we obtain
\begin{equation}\label{eq:refinedj02}
\begin{aligned}
&\sum_{\substack{Q\in \mathcal{C}\\M_k(Q)>0}}\frac{1}{M_k(Q)^{2/n}}\int_{Q}\left|A_k(\eta,y,t)\right|^{p_0^{\ast}}\rmd y\rmd t\\
&\quad=\sum_\mu\sum_{Q\in \mathcal{C}^{\mu}_k}\frac{1}{M_k(Q)^{2/n}}\int_{Q}\left|A_k(\eta,y,t)\right|^{p_0^{\ast}}\rmd y\rmd t\\
&\quad\leq \sum_\mu\frac{1}{\mu^{2/n}}\sum_{Q\in \mathcal{C}_k^{\mu}}\int_{Q}\left|A_k(\eta,y,t)\right|^{p_0^{\ast}}\rmd y\rmd t\\
&\quad\lesssim \sum_\mu\frac{1}{\mu^{2/n}} \left(r^{1+\varepsilon}(\bar \beta^{2}\mu)^{2/n}\sum_{\iota\in \mathcal{W}_k}\|f_\iota^{(r,t_0)}\|_{L^{2}}^{2}+\frac{\|f_k\|_{L^{2}}^{2}M_{T,r}(f)^{p_0^{\ast}-2}}{R^{N}}\right)\\
&\quad\lesssim r^{1+\varepsilon}\bar \beta^{4/n}\sum_{\iota\in \mathcal{W}_k}\|f_\iota^{(r,t_0)}\|_{L^{2}}^{2}+\frac{\|f_k\|_{L^{2}}^{2}M_{T,r}(f)^{p_0^{\ast}-2}}{R^{N}}.
\end{aligned}
\end{equation}
By \eqref{eq:wp-continuous},
\begin{equation*}
E_{m+n}F_\tau[f_{\mathcal{W}}](x,y,t)
=\sum_k\int_{\mathbb{R}^{m}}\rho_\tau(\eta-k)e(x\cdot \eta)A_k(\eta,y,t)\rmd \eta.
\end{equation*}
Hence, 
\begin{equation*}
\supp (E_{m+n}F_\tau[f_{\mathcal{W}}](\cdot,y,t))^{\wedge}\subset \bigcup_{\substack{k\in L^{-1}\mathbb{Z}^{m}\\|k|\leq 1}}(k+\tau \supp \rho).
\end{equation*}
By H\"older's inequality, 
\begin{align*}
|E_{m+n}F_\tau[f_{\mathcal{W}}](x,y,t)|
&\leq \int_{\mathbb{R}^{m}}\left|(E_{m+n}F_\tau[f_{\mathcal{W}}](\cdot,y,t))^{\wedge}(\eta)\right|\rmd \eta\\
&\lesssim (L\tau)^{m/2} \left(\int_{\mathbb{R}^{m}}|(E_{m+n}F_\tau[f_{\mathcal{W}}](\cdot,y,t))^{\wedge}(\eta)|^{2}\rmd \eta\right)^{1/2}.
\end{align*}
Thus, by Plancherel's identity, orthogonality of the supports of $\rho_\tau(\cdot-k)$ and $p_0^{\ast}/2-1=2/n$, we have
\begin{equation}\label{eq:refinedj03}
\begin{aligned}
&\int_{\mathbb{R}^{m}}\left|E_{m+n}F_\tau[f_{\mathcal{W}}](x,y,t)\right|^{p_0^{\ast}}\rmd x\\
&\quad\leq \int_{\mathbb{R}^{m}}\left|E_{m+n}F_\tau[f_{\mathcal{W}}](x,y,t)\right|^{2}\rmd x \cdot \sup_x |E_{m+n}F_\tau[f_{\mathcal{W}}](x,y,t)|^{p_0^{\ast}-2}\\
&\quad\lesssim (L\tau)^{2m/n}\left(\int_{\mathbb{R}^{m}}\left|(E_{m+n}F_\tau[f_{\mathcal{W}}](\cdot,y,t))^{\wedge}(\eta)\right|^{2}\rmd \eta\right)^{p_0^{\ast}/2}\\
&\quad=(L\tau)^{2m/n}\left(\sum_k\int_{\mathbb{R}^{m}}|\rho_\tau(\eta-k)|^{2}\left|A_k(\eta,y,t)\right|^{2} \rmd \eta\right)^{p_0^{\ast}/2}
\end{aligned}
\end{equation}
For $(y,t)\in Q$, the terms with $M_k(Q)=0$ satisfy the required remainder estimate after integration, by \eqref{eq:packet-decay}, 
the label count and \eqref{eq:M-L2}. We apply \eqref{eq:refinedj03} to the sum over the remaining $k$.
Since $p_0^\ast/2-1=2/n$, H\"older's inequality yields
\begin{equation}\label{eq:refinedj04}
\begin{aligned}
&\left(\sum_{k:M_k(Q)>0}\int_{\mathbb{R}^{m}}|\rho_\tau(\eta-k)|^{2}\left|A_k(\eta,y,t)\right|^{2} \rmd \eta\right)^{p_0^{\ast}/2}\\
&\quad\leq \left(\sum_{k:M_k(Q)>0}\!\!\left(\int_{\mathbb{R}^{m}}\!\!|\rho_\tau(\eta-k)|^{2}\rmd \eta\right)^{(p_0^{\ast}-2)/p_0^{\ast}}\!\!
\left(\int_{\mathbb{R}^{m}}\!\!\!|\rho_\tau(\eta-k)|^{2}\left|A_k(\eta,y,t)\right|^{p_0^{\ast}}\rmd \eta\right)^{2/p_0^{\ast}}\right)^{p_0^{\ast}/2}\\
&\quad \leq \left(\sum_{k:M_k(Q)>0}M_k(Q)\frac{\|\rho\|_{L^{2}}^{2}}{\tau^{m}}\right)^{2/n}\!\!\sum_{k:M_k(Q)>0}\frac{1}{M_k(Q)^{2/n}}\int_{\mathbb{R}^{m}}|\rho_\tau(\eta-k)|^{2}\left|A_k(\eta,y,t)\right|^{p_0^{\ast}}\rmd \eta\\
&\quad \leq \left(M\frac{\|\rho\|_{L^{2}}^{2}}{\tau^{m}}\right)^{2/n}\sum_{k:M_k(Q)>0}\frac{1}{M_k(Q)^{2/n}}\int_{\mathbb{R}^{m}}|\rho_\tau(\eta-k)|^{2}\left|A_k(\eta,y,t)\right|^{p_0^{\ast}}\rmd \eta.
\end{aligned}
\end{equation}
Integrating \eqref{eq:refinedj03} over $X$, using \eqref{eq:refinedj04} and \eqref{eq:refinedj02}, we obtain
\begin{align*}
&\int_{\mathbb{R}^{m}\times X}\left|E_{m+n}F_\tau[f_{\mathcal{W}}](x,y,t)\right|^{p_0^{\ast}}\rmd x\rmd y\rmd t\\
&\quad \leq L^{2m/n}M^{2/n}\|\rho\|_{L^{2}}^{4/n}\sum_k \!\!\sum_{\substack{Q\in \mathcal{C}\\M_k(Q)>0}}\frac{1}{M_k(Q)^{2/n}}\int_{\mathbb{R}^{m}}|\rho_\tau(\eta-k)|^{2}\int_Q |A_k(\eta,y,t)|^{p_0^{\ast}}\rmd y\rmd t\rmd \eta\\
&\quad \lesssim L^{2m/n}M^{2/n}\frac{\|\rho\|_{L^{2}}^{4/n+2}}{\tau^{m}}\sum_k \left(r^{1+\varepsilon}\bar{\beta}^{4/n}\sum_{\iota\in \mathcal{W}_k}\|f_\iota^{(r,t_0)}\|_{L^{2}}^{2}+\frac{\|f_k\|_{L^{2}}^{2}M_{T,r}(f)^{p_0^{\ast}-2}}{R^{N}}\right)\\
&\quad \lesssim r^{1+\varepsilon}L^{2m/n}(\bar{\beta}^{2}M)^{2/n}\frac{\|\rho\|_{L^{2}}^{2}}{\tau^{m}}\sum_{\iota\in \mathcal{W}}\|f_{\iota}^{(r,t_0)}\|_{L^{2}}^{2}+\frac{\|F_\tau[f]\|_{L^{2}}^{2}M_{T,r}(f)^{p_0^{\ast}-2}}{R^{N}}.
\end{align*}
The last step uses $L\leq R$ and $M\lesssim L^{m}r^{n}$. This proves \eqref{eq:refined-decoupling} for $j=0$ and $m\geq 1$.
For $m=0$, the scalar argument giving \eqref{eq:refinedj01}, without $k$ and $\eta$, proves \eqref{eq:refined-decoupling} directly.

\medskip
\noindent\textbf{3. Interpolation.}
\par\nopagebreak

For the case $0<j<m$, set
\begin{equation*}
\vartheta=\frac{j(m+n)}{m(n+j)}.
\end{equation*}
Then $0<\vartheta<1$ and $p_j^{\ast}=\vartheta p_m^{\ast}+(1-\vartheta)p_0^{\ast}$. By H\"older's inequality, we have
\begin{align*}
&\int_{\mathbb{R}^{m}\times X}\left|E_{m+n}F_\tau[f_{\mathcal{W}}](x,y,t)\right|^{p_j^{\ast}}\rmd x\rmd y\rmd t\\
&\quad \leq \left(\int_{\mathbb{R}^{m}\times X}\left|E_{m+n}F_\tau[f_{\mathcal{W}}]\right|^{p_m^{\ast}}\rmd x\rmd y\rmd t\right)^{\vartheta}
\left(\int_{\mathbb{R}^{m}\times X}\left|E_{m+n}F_\tau[f_{\mathcal{W}}]\right|^{p_0^{\ast}}\rmd x\rmd y\rmd t\right)^{1-\vartheta}\\
&\quad\lesssim\left(\begin{aligned}
&r^{1+\varepsilon}\left(1+\frac L{r^{1/2}}\right)^{2m/(m+n)}(\bar\beta^2M)^{2/(m+n)}\frac{\|\rho\|_{L^2}^2}{\tau^m}
\sum_{\iota\in\mathcal W}\|f_\iota^{(r,t_0)}\|_{L^2}^2\\
&\qquad+\frac{\|F_\tau[f]\|_{L^2}^2M_{T,r}(f)^{p_m^\ast-2}}{R^N}
\end{aligned}\right)^\vartheta\\
&\qquad \times \left(r^{1+\varepsilon}L^{2m/n}(\bar{\beta}^{2}M)^{2/n}\frac{\|\rho\|_{L^{2}}^{2}}{\tau^{m}}\sum_{\iota\in \mathcal{W}}\|f_{\iota}^{(r,t_0)}\|_{L^{2}}^{2}+\frac{\|F_\tau[f]\|_{L^{2}}^{2}M_{T,r}(f)^{p_0^{\ast}-2}}{R^{N}}\right)^{1-\vartheta}.
\end{align*}
Since $\vartheta/(m+n)+(1-\vartheta)/n=1/(n+j)$, the product of the main term is bounded by
\begin{equation*}
r^{1+\varepsilon}\left(L^{m-j}\left(1+\frac{L}{r^{1/2}}\right)^{j}\right)^{2/(n+j)}(\bar{\beta}^{2}M)^{2/(n+j)}\frac{\|\rho\|_{L^{2}}^{2}}{\tau^{m}}\sum_{\iota\in \mathcal{W}}\|f_\iota^{(r,t_0)}\|_{L^{2}}^{2}.
\end{equation*}
To control the other terms, \eqref{eq:packet-local-L2} and $M\lesssim L^{m}r^{n}$ imply that
\begin{equation*}
\bar{\beta}^{2}M\lesssim L^{m}T^{n}M_{T,r}(f)^{2}.
\end{equation*}
Combining this and \eqref{eq:packet-L2}, the remaining terms are bounded by 
\begin{equation*}
\frac{\|F_\tau[f]\|_{L^{2}}^{2}M_{T,r}(f)^{p_j^{\ast}-2}}{R^{N}}.
\end{equation*}
This completes the proof of \eqref{eq:refined-decoupling}.
\end{proof}

\section{Proof of Theorem \ref{thm:main}}\label{sec:local-estimate}

We prove \eqref{eq:local-reduction} by establishing the following stronger estimate.

\begin{theorem}\label{thm:local-estimate}
Let $R>1$, $L=R^{1/a_0}$ and $0<\tau\leq c/R$. Fix a dyadic number $1\leq T\leq R$ and a ball $B_r(y_0,t_0)\subset\mathbb{R}^{n+1}$
with $1\leq r\leq T$ and $t_0\sim T$, such that $t\sim T$ on its fixed enlargements.
For every $\zeta>0$ and every $f$ satisfying
\begin{equation*}
\supp\widehat f\subset
\{(k,\xi)\in L^{-1}\mathbb Z^m\times\mathbb R^n:|k|\leq1,\ |\xi|\leq2+c_1/r^{1/2}\},
\end{equation*}
we have
\begin{equation}\label{eq:local-estimate}
\begin{aligned}
&\int_{\mathbb R^m\times B_r(y_0,t_0)}|E_{m+n}F_\tau[f](x,y,t)|^{p_0}\rmd x\rmd y\rmd t\\
&\qquad\lesssim R^\zeta\big(L^mT^n\big)^{p_0/2-1}\|F_\tau[f]\|_{L^2}^2M_{T,r}(f)^{p_0-2}.
\end{aligned}
\end{equation}
\end{theorem}

\begin{proof}
Fix $\zeta>0$ and assume $f\neq 0$, the zero case being immediate. 
Let $\varepsilon,\varsigma>0$ be sufficiently small, and choose $\varepsilon_0>0$ sufficiently small in terms of $\varepsilon$. 
All the estimates below are uniform for $0<\tau\leq c/R$. A term denoted by $\operatorname{Err}_p$ satisfies, for every $N>0$,
\begin{equation}\label{eq:negligible-remainder}
\operatorname{Err}_p\lesssim \frac{1}{R^N}\|F_\tau[f]\|_{L^2}^2M_{T,r}(f)^{p-2},\qquad 2\leq p\leq 2+4/n.
\end{equation}
The decay orders in the estimates below are chosen after $N$. By Corollary \ref{cor:terminal}, we may assume that $R$ is sufficiently large.

\medskip
\noindent\textbf{Step 1. One-step reduction.}
\nopagebreak

By \eqref{eq:terminal}, we only need to consider $r>R^\varsigma$. For $\iota=(k,\theta,v)$, recall
\begin{equation*}
\mathcal P_\iota(x,y,t)
=E_{m+n}\left[\rho_\tau(\eta-k)e(-t_0\Phi(k,\xi))f_\iota^{(r,t_0)}(\xi)\right](x,y,t).
\end{equation*}
By \eqref{eq:wp-evolution}, $E_{m+n}F_\tau[f]=\sum_\iota\mathcal P_\iota$.

We first discard the wave packets with $|v-y_0|>C_4r$.
For $(y,t)\in B_r(y_0,t_0)$, boundedness of $\nabla_\xi\Phi$ gives, for sufficiently large $C_4$,
\begin{equation*}
|y-v+(t-t_0)\nabla_\xi\Phi(k,\xi_\theta)|
\geq |v-y_0|-|y-y_0|-|t-t_0||\nabla_\xi\Phi(k,\xi_\theta)|
\geq \tfrac12|v-y_0|.
\end{equation*}
For every $N_1>n$, summing over $v\in r^{1/2}\mathbb Z^n$ yields
\begin{align*}
\sum_{\substack{v\in r^{1/2}\mathbb Z^n\\|v-y_0|>C_4r}} \frac{1}{\left(1+{|y-v+(t-t_0)\nabla_\xi\Phi(k,\xi_\theta)|}/{r^{1/2}}\right)^{N_1}}
&\lesssim \sum_{\substack{v\in r^{1/2}\mathbb Z^n\\|v-y_0|>C_4r}}\frac{1}{\left(1+{|v-y_0|}/{r^{1/2}}\right)^{N_1}}\\
&\lesssim\frac{1}{r^{(N_1-n)/2}}.
\end{align*}
Since $\#\{k:|k|\leq1\}\lesssim L^m$ and $\#\Theta\lesssim r^{n/2}$, Lemma \ref{lem:packet-decay} yields, for every $N_2\geq 0$,
\begin{align*}
&\left|\sum_{\iota:\,|v-y_0|>C_4r}\mathcal P_\iota(x,y,t)\right|\\
&\quad\lesssim \frac{T^{n/2}M_{T,r}(f)}{(1+\tau|x|)^{N_2}}\sum_{k,\theta}\sum_{\substack{v\in r^{1/2}\mathbb Z^n\\|v-y_0|>C_4r}}
\frac{1}{\left(1+{|y-v+(t-t_0)\nabla_\xi\Phi(k,\xi_\theta)|}/{r^{1/2}}\right)^{N_1}}\\
&\quad\lesssim \frac{L^mT^{n/2}r^{n/2}M_{T,r}(f)}{r^{(N_1-n)/2}(1+\tau|x|)^{N_2}}\\
&\quad \lesssim \frac{M_{T,r}(f)}{R^{N+n+1}(1+\tau|x|)^{N_2}}.
\end{align*}
where the last inequality follows by taking $N_1$ sufficiently large, since $r\geq R^\varsigma$, $L=R^{1/a_0}$ and $T\leq R$.
Taking $N_2p>m$ and using \eqref{eq:M-L2}, for $2\leq p\leq 2+4/n$, we have
\begin{align*}
\int_{\mathbb R^m\times B_r(y_0,t_0)} \left|\sum_{\iota:\,|v-y_0|>C_4r}\mathcal P_\iota(x,y,t)\right|^p \rmd x\rmd y\rmd t
&\lesssim \frac{r^{n+1}M_{T,r}(f)^p}{R^{p(N+n+1)}}\int_{\mathbb R^m}\frac{\rmd x}{(1+\tau|x|)^{N_2p}}\\
&\lesssim \frac{1}{R^N}\|F_\tau[f]\|_{L^2}^2M_{T,r}(f)^{p-2}.
\end{align*}
This satisfies the required remainder estimate \eqref{eq:negligible-remainder}. Thus we only need to consider labels satisfying $|v-y_0|\leq C_4r$.

By the bounded Fourier support, the number of remaining labels satisfies
\begin{equation*}
\#\{\iota=(k,\theta,v):|v-y_0|\leq C_4r\}
\lesssim L^m r^n\lesssim R^{m+n}.
\end{equation*}
For the labels $\iota$ obeying $\|f_\iota^{(r,t_0)}\|_{L^{2}}/r^{n/4}<M_{T,r}(f)/R^N$, \eqref{eq:packet-infty} gives
\begin{equation*}
\left|\sum_{\iota}\mathcal P_\iota(x,y,t)\right|
\lesssim \frac{1}{(1+\tau|x|)^{N_2}} \sum_{\iota} \frac{\|f_\iota^{(r,t_0)}\|_{L^2}}{r^{n/4}}
\lesssim \frac{L^mr^nM_{T,r}(f)}{R^N(1+\tau|x|)^{N_2}}.
\end{equation*}
Consequently, for $2\leq p\leq 2+4/n$ and $N_2p>m$, we obtain
\begin{align*}
\int_{\mathbb R^m\times B_r(y_0,t_0)}\left|\sum_{\iota}\mathcal P_\iota(x,y,t)\right|^p\rmd x\rmd y\rmd t
&\lesssim\frac{L^{mp}r^{np+n+1}}{R^{Np}}\frac{M_{T,r}(f)^p}{\tau^m}\\
&\lesssim\frac{R^{C_{m,n}}}{R^{Np}}\|F_\tau[f]\|_{L^2}^2M_{T,r}(f)^{p-2}.
\end{align*}
Since $N$ may be chosen arbitrarily large, the last expression is bounded by \eqref{eq:negligible-remainder} for every prescribed decay order.

For all the remaining labels, \eqref{eq:packet-local-L2} yields
\begin{equation*}
\frac{M_{T,r}(f)}{R^N}
\leq\frac{\|f_\iota^{(r,t_0)}\|_{L^2}}{r^{n/4}}
\lesssim\left(\frac Tr\right)^{n/2}M_{T,r}(f)
\lesssim R^CM_{T,r}(f).
\end{equation*}
For each dyadic $\bar\beta$, let $\mathcal W(\bar\beta)$ be the set of labels satisfying
\begin{equation*}
\bar\beta\leq\frac{\|f_\iota^{(r,t_0)}\|_{L^2}}{r^{n/4}}<2\bar\beta.
\end{equation*}
There are $O(\log R)$ nonempty classes. By the triangle inequality and the pigeonhole principle, one class $\mathcal W_0$ satisfies
\begin{equation}\label{eq:amplitude-selection}
\begin{aligned}
&\int_{\mathbb R^m\times B_r(y_0,t_0)}|E_{m+n}F_\tau[f](x,y,t)|^{p_0}\rmd x\rmd y\rmd t\\
&\qquad\lesssim(\log R)^C\int_{\mathbb R^m\times B_r(y_0,t_0)}|E_{m+n}F_\tau[f_{\mathcal W_0}](x,y,t)|^{p_0}\rmd x\rmd y\rmd t+\operatorname{Err}_{p_0}.
\end{aligned}
\end{equation}

For a ball $B\subset\mathbb R^{n+1}$, write
\begin{equation*}
\mathcal N_B(G)=\left(\tau^m\int_{\mathbb R^m\times B}|G(x,y,t)|^{p_0}\rmd x\rmd y\rmd t\right)^{1/p_0}.
\end{equation*}
Cover $B_r(y_0,t_0)$ by $r^{1/2}$-balls whose centers lie on a regular $cr^{1/2}$-lattice, with $c>0$ fixed sufficiently small.
Denote this bounded-overlap family by $\mathcal B$.
By \eqref{eq:packet-decay},
\begin{equation*}
\mathcal N_B(E_{m+n}F_\tau[f_{\mathcal W_0}])\lesssim R^CM_{T,r}(f).
\end{equation*}
Also, $\#\mathcal B\lesssim r^{(n+1)/2}\leq R^{(n+1)/2}$.
For the balls obeying
\begin{equation*}
\mathcal N_B(E_{m+n}F_\tau[f_{\mathcal W_0}])<\frac{M_{T,r}(f)}{R^N},
\end{equation*}
\eqref{eq:M-L2} yields
\begin{align*}
\sum_{B}\int_{\mathbb R^m\times B}|E_{m+n}F_\tau[f_{\mathcal W_0}](x,y,t)|^{p_0}\rmd x\rmd y\rmd t
&\lesssim\frac{\#\mathcal B}{\tau^mR^{Np_0}}M_{T,r}(f)^{p_0}\\
&\lesssim\frac{1}{R^{Np_0-C}}\|F_\tau[f]\|_{L^2}^2M_{T,r}(f)^{p_0-2}.
\end{align*}
Hence these terms satisfy the remainder estimate \eqref{eq:negligible-remainder}. We therefore restrict attention to the remaining balls, which obey
\begin{equation*}
\frac{M_{T,r}(f)}{R^N}\lesssim\mathcal N_B(E_{m+n}F_\tau[f_{\mathcal W_0}])\lesssim R^CM_{T,r}(f).
\end{equation*}
Denote this subfamily by $\mathcal{B}_0$.

If $\mathcal{B}_0=\varnothing$, the estimate follows directly. Otherwise, since $\mathcal{B}$ cover $B_r(y_0, t_0)$, we have
\begin{align*}
&\int_{\mathbb{R}^{m}\times B_r(y_0,t_0)}\left|E_{m+n}F_\tau[f_{\mathcal W_0}](x,y,t)\right|^{p_0}\rmd x\rmd y\rmd t\\
&\qquad\lesssim \sum_{B\in \mathcal{B}_0}\int_{\mathbb{R}^{m}\times B}\left|E_{m+n}F_\tau[f_{\mathcal W_0}](x,y,t)\right|^{p_0}\rmd x\rmd y\rmd t+\operatorname{Err}_{p_0}.
\end{align*}
Since $\mathcal{B}_0$ has bounded overlap, we can partition $\mathcal{B}_0$ into $O(1)$ subfamilies of disjoint balls.
By the pigeonhole principle, one of these subfamilies $\mathcal{B}_1$ satisfies
\begin{align*}
&\sum_{B\in \mathcal{B}_0}\int_{\mathbb{R}^{m}\times B}\left|E_{m+n}F_\tau[f_{\mathcal W_0}](x,y,t)\right|^{p_0}\rmd x\rmd y\rmd t\\
&\qquad\lesssim \sum_{B\in \mathcal{B}_1}\int_{\mathbb{R}^{m}\times B}\left|E_{m+n}F_\tau[f_{\mathcal W_0}](x,y,t)\right|^{p_0}\rmd x\rmd y\rmd t.
\end{align*}
Partition $\mathcal{B}_1$ into classes $\{B\in \mathcal{B}_1:2^{j}\leq \mathcal{N}_B(E_{m+n}F_\tau[f_{\mathcal W_0}])<2^{j+1}\}$ for $j\in\mathbb Z$.
By the pigeonhole principle, one can select a family of balls $\mathcal{Q}\subset \mathcal{B}_1$
such that the values $\mathcal{N}_B(E_{m+n}F_\tau[f_{\mathcal W_0}])$ associated with the balls in $\mathcal{Q}$ are comparable, and
\begin{align*}
&\sum_{B\in \mathcal{B}_1}\int_{\mathbb{R}^{m}\times B}\left|E_{m+n}F_\tau[f_{\mathcal W_0}](x,y,t)\right|^{p_0}\rmd x\rmd y\rmd t\\
&\qquad\lesssim (\log R)^C\sum_{B\in \mathcal{Q}}\int_{\mathbb{R}^{m}\times B}\left|E_{m+n}F_\tau[f_{\mathcal W_0}](x,y,t)\right|^{p_0}\rmd x\rmd y\rmd t.
\end{align*}
Let $X=\cup_{B\in \mathcal{Q}}B$, so that
\begin{equation}\label{eq:ball-selection}
\begin{aligned}
&\int_{\mathbb R^m\times B_r(y_0,t_0)}|E_{m+n}F_\tau[f_{\mathcal W_0}](x,y,t)|^{p_0}\rmd x\rmd y\rmd t\\
&\qquad\lesssim(\log R)^C\int_{\mathbb R^m\times X}|E_{m+n}F_\tau[f_{\mathcal W_0}](x,y,t)|^{p_0}\rmd x\rmd y\rmd t+\operatorname{Err}_{p_0}.
\end{aligned}
\end{equation}

\medskip
\noindent\textbf{Step 2. Two-ends reduction.}
\nopagebreak

We next select the wave packets and balls so that their incidences satisfy the counting estimates needed for the two-ends argument.

For $\sigma\in[0,r^{1-\varepsilon^2})$, let $I_j^\sigma=[\sigma+jr^{1-\varepsilon^2},\sigma+(j+1)r^{1-\varepsilon^2})$ with $j\in\mathbb Z$.
For each $B\in\mathcal Q$, denote by $t_B$ the time coordinate of its center. For each fixed $B\in \mathcal{Q}$, we have
\begin{equation*}
\frac{1}{r^{1-\varepsilon^2}}\int_0^{r^{1-\varepsilon^2}} 1_{\{\operatorname{dist}(t_B,\sigma+r^{1-\varepsilon^2}\mathbb Z)
\geq r^{1-\varepsilon^2}/4\}}\rmd\sigma
=\frac12,
\end{equation*}
summing over $B\in\mathcal Q$ gives
\begin{equation*}
\frac{1}{r^{1-\varepsilon^2}}\int_0^{r^{1-\varepsilon^2}}\#\left\{B\in\mathcal Q:
\operatorname{dist}(t_B,\sigma+r^{1-\varepsilon^2}\mathbb Z)\geq \frac{r^{1-\varepsilon^2}}{4}\right\}
\rmd\sigma
=\frac12\#\mathcal Q.
\end{equation*}
By averaging, there exists $\sigma\in[0,r^{1-\varepsilon^2})$ such that
\begin{equation*}
\#\left\{B\in\mathcal Q:
\operatorname{dist}(t_B,\sigma+r^{1-\varepsilon^2}\mathbb Z)\geq \frac{r^{1-\varepsilon^2}}{4}\right\}
\geq\frac12\#\mathcal Q.
\end{equation*}
Since the local norms $\mathcal N_B(E_{m+n}F_\tau[f_{\mathcal W_0}])$ are comparable on $\mathcal Q$, such balls satisfy
\begin{equation*}
\begin{aligned}
&\int_{\mathbb R^m\times X}|E_{m+n}F_\tau[f_{\mathcal W_0}](x,y,t)|^{p_0}\rmd x\rmd y\rmd t\\
&\quad\lesssim \sum_{\substack{B\in\mathcal Q\\\operatorname{dist}(t_B,\sigma+r^{1-\varepsilon^2}\mathbb Z)\geq r^{1-\varepsilon^2}/4}}
\int_{\mathbb R^m\times B}|E_{m+n}F_\tau[f_{\mathcal W_0}](x,y,t)|^{p_0}\rmd x\rmd y\rmd t.
\end{aligned}
\end{equation*}
We retain these balls and continue to denote their family and union by $\mathcal Q$ and $X$. Write $I_j=I_j^\sigma$.
Since $r^{1/2}\ll r^{1-\varepsilon^2}$, each $B\in\mathcal Q$, together with its fixed enlargements,
is contained in $\mathbb R^n\times I_j$ for a unique $j$.

For each $\iota=(k,\theta,v)$, define
\begin{equation*}
\Pi_\iota^*
=\left\{(y,t): |t-t_0|\leq C_6r,\quad|y-v+(t-t_0)\nabla_\xi \Phi(k,\xi_\theta)|\leq C_7r^{1/2+\varepsilon_0}\right\},
\end{equation*}
where $C_6,C_7$ are sufficiently large fixed constants. By the definition of $\Pi_\iota$ and the boundedness of $\nabla_\xi \Phi(k,\xi_\theta)$,
they may be chosen so that $\Pi_\iota\cap C_5B\neq\varnothing$ implies $B\subset\Pi_\iota^*$ for $B\in\mathcal Q$.

For every $\iota$, partition $\Pi^{\ast}_\iota=\sqcup_j J_{\iota,j}$, where $J_{\iota,j}=\Pi_\iota^*\cap(\mathbb R^n\times I_j)$.
For each nonempty $J_{\iota,j}$, define
$N_X(J_{\iota,j})=\#\{B\in\mathcal Q:B\cap J_{\iota,j}\neq\varnothing\}$.
For every dyadic number $\lambda\in2^{\mathbb Z}$, let $\mathcal J_\lambda(\iota)$ consist of those $J_{\iota,j}$ satisfying
\begin{equation*}
\lambda r^{1/2}
\leq N_X(J_{\iota,j})
<2\lambda r^{1/2}.
\end{equation*}
Since
\begin{equation*}
1\leq N_X(J_{\iota,j})
\leq\#\mathcal Q
\lesssim r^{(n+1)/2}
\lesssim R^{(n+1)/2}
\end{equation*}
whenever $N_X(J_{\iota,j})>0$, there are $O(\log R)$ nonempty choices of $\lambda$.

For each such $\lambda$, define
\begin{equation*}
G_{\tau,\lambda}(x,y,t)
=\sum_{\iota\in\mathcal W_0}\mathcal P_\iota(x,y,t)\sum_{J\in\mathcal J_\lambda(\iota)}1_J(y,t).
\end{equation*}
On $\mathbb R^m\times X$, we have
\begin{equation*}
E_{m+n}F_\tau[f_{\mathcal W_0}]-\sum_\lambda G_{\tau,\lambda}
=\sum_{\iota\in\mathcal W_0}\mathcal P_\iota 1_{(\Pi_\iota^*)^c}.
\end{equation*}
By \eqref{eq:packet-decay} and $r\geq R^\varsigma$, for every $N>0$,
\begin{equation*}
\mathcal N_B\left(\sum_{\iota\in\mathcal W_0}\mathcal P_\iota 1_{(\Pi_\iota^*)^c}\right)
\lesssim \frac{M_{T,r}(f)}{R^N}, \qquad B\in\mathcal Q.
\end{equation*}
Increasing the decay order makes this remainder at most $\mathcal N_B(E_{m+n}F_\tau[f_{\mathcal W_0}])/2$.
Thus, the triangle inequality implies
\begin{equation*}
\mathcal N_B(E_{m+n}F_\tau[f_{\mathcal W_0}])
\lesssim
\sum_\lambda\mathcal N_B(G_{\tau,\lambda}).
\end{equation*}
Hence, by the pigeonhole principle, for each $B\in \mathcal{Q}$, there exists a $\lambda(B)$ such that
\begin{equation*}
\mathcal N_B(E_{m+n}F_\tau[f_{\mathcal W_0}])
\lesssim \log R \,\mathcal{N}_B(G_{\tau,\lambda(B)}).
\end{equation*}

We classify the balls in $\mathcal{Q}$ by the value of $\lambda(B)$. By the choice of $\lambda$, such number of classes $\lesssim \log R$.
By the pigeonhole principle, there exists $\mathcal{Q}'\subset \mathcal{Q}$, denoting $X'=\cup_{B\in \mathcal{Q}'}B$, such that
for every $B_1, B_2\in \mathcal{Q}'$, $\lambda(B_1)=\lambda(B_2)$ and
\begin{equation*}
\int_{\mathbb{R}^{m}\times X}|E_{m+n}F_\tau[f_{\mathcal W_0}](x,y,t)|^{p_0}\rmd x\rmd y\rmd t
\lesssim \log R\int_{\mathbb{R}^{m}\times X'}|E_{m+n}F_\tau[f_{\mathcal W_0}](x,y,t)|^{p_0}\rmd x\rmd y\rmd t.
\end{equation*}
We retain the balls satisfying $\lambda(B)=\lambda$.

For every $B\in \mathcal{Q}'$, the preceding estimate and the bounds on $\mathcal N_B(E_{m+n}F_\tau[f_{\mathcal W_0}])$ give
\begin{equation*}
\frac{M_{T,r}(f)}{R^{C_N}}
\lesssim \mathcal N_B(G_{\tau,\lambda})
\lesssim R^CM_{T,r}(f),
\end{equation*}
where $C_N>0$. We partition $\mathcal{Q}'=\sqcup_j \mathcal{Q}'_j$, where $\mathcal{Q}_j'=\{B\in \mathcal{Q}':
2^{j}\leq \mathcal{N}_{B}(G_{\tau,\lambda})<2^{j+1}\}$.
The number of nonempty classes $\mathcal{Q}'_j$ $\lesssim \log R$. By the pigeonhole principle,
there exists $j_0\in \mathbb{Z}$, setting $\mathcal{Q}_1=\mathcal{Q}_{j_0}'$ and $X_1=\cup_{B\in \mathcal{Q}_1}B$, such that
\begin{equation*}
\int_{\mathbb{R}^{m}\times X}|E_{m+n}F_\tau[f_{\mathcal W_0}](x,y,t)|^{p_0}\rmd x\rmd y\rmd t
\lesssim (\log R)^{C} \int_{\mathbb{R}^{m}\times X_1}|G_{\tau,\lambda}(x,y,t)|^{p_0}\rmd x\rmd y\rmd t.
\end{equation*}

With $\lambda$ fixed, for every dyadic number $\beta=2^j$, $j\geq0$, let
$\mathcal W_{\lambda,\beta}=\{\iota\in\mathcal W_0: \beta\leq\#\mathcal J_\lambda(\iota)<2\beta\}$.
Since $\Pi_\iota^*$ has time length $O(r)$ and each $I_j$ has length $r^{1-\varepsilon^2}$, we have
\begin{equation*}
1\leq\#\mathcal J_\lambda(\iota)
\lesssim 1+\frac{r}{r^{1-\varepsilon^2}}
\lesssim r^{\varepsilon^2}
\end{equation*}
whenever $\mathcal J_\lambda(\iota)\neq\varnothing$. Thus there are $O(\log R)$ nonempty choices of $\beta$.
For each such $\beta$, define
\begin{equation*}
G_{\tau,\lambda,\beta}(x,y,t)
=\sum_{\iota\in\mathcal W_{\lambda,\beta}}\mathcal P_\iota(x,y,t)\sum_{J\in\mathcal J_\lambda(\iota)} 1_J(y,t).
\end{equation*}
Clearly,
\begin{equation*}
G_{\tau,\lambda}
=\sum_\beta G_{\tau,\lambda,\beta}.
\end{equation*}

We repeat the preceding three selections. There exist $\beta_0\in 2^{\mathbb{Z}}$ and $\mathcal{Q}_2\subset \mathcal{Q}_1$,
denoting $X_2=\cup_{B\in \mathcal{Q}_2}B$, such that
$\#\mathcal{J}_{\lambda}(\iota')\sim \#\mathcal{J}_{\lambda}(\iota'')$ for all $\iota',\iota''\in \mathcal{W}_{\lambda,\beta_0}$,
$\mathcal{N}_{B'}(G_{\tau, \lambda, \beta_0})\sim \mathcal{N}_{B''}(G_{\tau, \lambda, \beta_0})$, where $B',B''\in \mathcal{Q}_2$,
$\#\mathcal{Q}\lesssim (\log R)^C \# \mathcal{Q}_2$ and
\begin{equation*}
\int_{\mathbb{R}^{m}\times X_1}|G_{\tau,\lambda}(x,y,t)|^{p_0}\rmd x\rmd y\rmd t
\lesssim (\log R)^{C}\int_{\mathbb{R}^{m}\times X_2}|G_{\tau,\lambda,\beta_0}(x,y,t)|^{p_0}\rmd x\rmd y\rmd t.
\end{equation*}

We now write $\mathcal W=\mathcal W_{\lambda,\beta_0}$, $\beta=\beta_0$ and $G_\tau=G_{\tau,\lambda,\beta_0}$.
Combining the preceding selections with \eqref{eq:amplitude-selection} and \eqref{eq:ball-selection}, we obtain
\begin{equation}\label{eq:packet-reduction}
\begin{aligned}
&\int_{\mathbb R^m\times B_r(y_0,t_0)}|E_{m+n}F_\tau[f](x,y,t)|^{p_0}\rmd x\rmd y\rmd t\\
&\qquad\lesssim (\log R)^C\int_{\mathbb R^m\times X_2}|G_\tau(x,y,t)|^{p_0}\rmd x\rmd y\rmd t+\operatorname{Err}_{p_0}.
\end{aligned}
\end{equation}

For every $\iota\in\mathcal W$, define
$\mathcal Q(\iota)=\{B\in\mathcal Q:B\cap J\neq\varnothing\text{ for some }J\in\mathcal J_\lambda(\iota)\}$.
Since every $B\in\mathcal Q$ lies in a unique $\mathbb R^n\times I_j$,
the sets of balls counted by distinct pieces $J\in\mathcal J_\lambda(\iota)$ are disjoint. Hence
\begin{equation}\label{eq:twoends-count}
\#\mathcal Q(\iota)
=\sum_{J\in\mathcal J_\lambda(\iota)}N_X(J)
\sim\lambda r^{1/2}\#\mathcal J_\lambda(\iota)
\sim\lambda\beta r^{1/2}.
\end{equation}
For any interval $I\subset\mathbb R$ of length at most $Cr^{1-\varepsilon^2}$, there are only $O(1)$ intervals $I_j$ meeting $I$.
By the definition of $\mathcal J_\lambda(\iota)$,
\begin{equation}\label{eq:twoends-local-count}
\#\{B\in\mathcal Q(\iota):t_B\in I\}
\leq\sum_{\substack{J_{\iota,j}\in\mathcal J_\lambda(\iota)\\I_j\cap I\neq\varnothing}}N_X(J_{\iota,j})
\lesssim\lambda r^{1/2}.
\end{equation}

Furthermore, the balls in $\mathcal Q(\iota)$ are pairwise disjoint and meet $\Pi_\iota^*$.
They are therefore contained in a fixed enlargement of a tube of length $O(r)$ and radius $O(r^{1/2+\varepsilon_0})$.
Comparing volumes gives
\begin{equation*}
\#\mathcal Q(\iota)
\lesssim\frac{r(r^{1/2+\varepsilon_0})^n}{r^{(n+1)/2}}
\lesssim r^{1/2+n\varepsilon_0}.
\end{equation*}
Since $\mathcal J_\lambda(\iota)$ is nonempty for every $\iota\in\mathcal W$, we also have $\#\mathcal Q(\iota)\geq1$.
Combining this with \eqref{eq:twoends-count}, we conclude that
\begin{equation}\label{eq:twoends-range}
\frac1{r^{1/2}}
\lesssim\lambda\beta\lesssim r^{n\varepsilon_0},
\qquad
1\leq\beta\lesssim r^{\varepsilon^2}.
\end{equation}

\medskip
\noindent\textbf{Step 3. Geometric estimate.}
\nopagebreak

By \eqref{eq:twoends-count}, we can choose $u\in [0,1]$ such that
\begin{equation*}
\frac{1}{r^{u/2}}\lesssim \lambda\beta\lesssim \frac{r^{n\varepsilon_0}}{r^{u/2}}.
\end{equation*}
Set
\begin{equation*}
\gamma_n(u)=
\begin{cases}
\displaystyle\frac{u}{2},&n=1,\\[8pt]
\displaystyle\min\left\{1,\frac{1}{2}+\frac{3}{4}u\right\},&n=2,\\[8pt]
\displaystyle\min\left\{\frac{n}{2},\frac{3n}{7}+\frac{2n+9}{7}u\right\},&n\geq3.
\end{cases}
\end{equation*}
Let $P_r$ be as in Lemma \ref{lem:packet-L2}, for the selected collection $\mathcal W$.

\begin{lemma}\label{lem:good-balls}
Suppose that $\beta\geq r^{\varepsilon^4}$. There exists a subfamily $\mathcal Q_3\subset\mathcal Q$,
writing $X_3=\bigcup_{B\in\mathcal Q_3}B$, such that
\begin{equation*}
|X\setminus X_3|
\leq\frac{|X|}{r^{2\varepsilon^2}},
\end{equation*}
and, for every $B\in\mathcal Q_3$,
\begin{equation*}
\#\{\iota\in\mathcal W:B\in\mathcal Q(\iota)\}\leq M,
\end{equation*}
where
\begin{equation}\label{eq:good-multiplicity}
M=C_{\varepsilon,n}r^{C_n\varepsilon^2}
P_rr^{\gamma_n(u)/2}.
\end{equation}
\end{lemma}

\begin{proof}
Consider the change of variables
\begin{equation*}
\Psi_T(y,t)
=\left(\frac{T(y-y_0)}{t},
\frac{T^2}{t}-\frac{T^2}{t_0}\right).
\end{equation*}
Since $|y-y_0|\lesssim r\leq T$ and $t\sim T$, direct differentiation shows that both $\Psi_T$ and
its inverse have uniformly bounded derivatives on the regions under consideration. Moreover,
$|\det D\Psi_T(y,t)|=\left(T/t\right)^{n+2}\sim1$.
For $\iota=(k,\theta,v)$, $\Psi_T$ transforms the axis $y=v-(t-t_0)\nabla_\xi \Phi(k,\xi_\theta)$ into the line
\begin{equation*}
\ell_\iota
=\left\{\left(\frac {T}{t_0}(v-y_0)+\frac{v-y_0+t_0\nabla_\xi \Phi(k,\xi_\theta)}{T}s,\ s\right):s\in\mathbb R\right\},
\end{equation*}
where $s=T^2/t-T^2/t_0$. We retain the original labels even when different labels determine the same line.

Now, we construct the shadings associated with the lines $\ell_\iota$.
Fix sufficiently large constants $C_9\gg C_8\gg C_7$. For each $B\in\mathcal Q(\iota)$ with center $(y_B,t_B)$, set
\begin{equation*}
s_B=\frac{T^2}{t_B}-\frac{T^2}{t_0}
\end{equation*}
and define
\begin{equation*}
Y(\ell_\iota)
=\bigcup_{B\in\mathcal Q(\iota)}B^{n+1}\left(\ell_\iota(s_B),C_8r^{1/2+\varepsilon_0}\right).
\end{equation*}
Since $B\cap\Pi_\iota^*\neq\varnothing$ and the radius of $B$ is $r^{1/2}$, we have
\begin{equation*}
|y_B-v+(t_B-t_0)\nabla_\xi \Phi(k,\xi_\theta)|
\lesssim r^{1/2+\varepsilon_0}.
\end{equation*}
On the other hand, direct computation gives
\begin{equation*}
|\Psi_T(y_B,t_B)-\ell_\iota(s_B)|
=\frac{T}{t_B}|y_B-v+(t_B-t_0)\nabla_\xi \Phi(k,\xi_\theta)|
\lesssim r^{1/2+\varepsilon_0}.
\end{equation*}
Hence, taking $C_8$ sufficiently large,
\begin{equation*}
\Psi_T(B)
\subset B^{n+1}\left(\ell_\iota(s_B),C_8r^{1/2+\varepsilon_0}\right)
\subset B^{n+1}\left(\Psi_T(y_B,t_B),Cr^{1/2+\varepsilon_0}\right).
\end{equation*}

We normalize the lines and shadings by $\widetilde\ell_\iota={\ell_\iota}/{C_9r}$, $Y(\widetilde\ell_\iota)={Y(\ell_\iota)}/{C_9r}$,
and define $\delta={C_8}/(C_9r^{1/2-\varepsilon_0})$.
Choosing $C_9$ sufficiently large, all the normalized shadings lie in $B^{n+1}(0,1/2)$ and consist of radius-$\delta$ balls centered on their respective lines.
Clearly,
\begin{equation*}
|N_\delta(\widetilde\ell_\iota)\cap B^{n+1}(0,1)|\sim\delta^n.
\end{equation*}

We now verify that the normalized shadings are $\Lambda$-dense.
The balls in $\mathcal Q$ are pairwise disjoint, and the Jacobian of $\Psi_T$ is comparable to one. Thus, by \eqref{eq:twoends-count},
\begin{equation}\label{eq:shading-lower}
|Y(\widetilde\ell_\iota)|
\gtrsim \frac{|\cup_{B\in \mathcal{Q}(\iota)}\Psi_T(B)|}{r^{n+1}}
\sim\frac{\#\mathcal Q(\iota)\,r^{(n+1)/2}}{r^{n+1}}
\gtrsim\frac{\lambda\beta}{r^{n/2}},
\end{equation}
On the other hand,
\begin{equation}\label{eq:shading-upper}
\left|\bigcup_{\iota\in\mathcal W}Y(\widetilde\ell_\iota)\right|
\lesssim\#\mathcal Q\,\delta^{n+1}.
\end{equation}
Consequently, by \eqref{eq:shading-lower},
\begin{equation*}
\frac{|Y(\widetilde\ell_\iota)|}{\left|N_\delta(\widetilde\ell_\iota)\cap B^{n+1}(0,1)\right|}
\gtrsim \frac{\lambda\beta}{r^{n/2}\delta^n}
\gtrsim\frac{\lambda\beta}{r^{n\varepsilon_0}}.
\end{equation*}
Taking
\begin{equation*}
\Lambda=\frac{c\lambda\beta}{r^{n\varepsilon_0}},
\end{equation*}
with $c>0$ sufficiently small, the shadings are $\Lambda$-dense. By \eqref{eq:twoends-range}, we have $0<\Lambda\leq1$.

We next verify the two-ends condition. Set $\varepsilon_1=4\varepsilon^2$ and $\varepsilon_2=\varepsilon^4$.
Let $J\subset N_\delta(\widetilde\ell_\iota)$ be a tube of length $\delta^{\varepsilon_1}$ and radius $\delta$.
If a ball forming $Y(\widetilde\ell_\iota)$ intersects $J$, its center has last coordinate in an interval of length $O(\delta^{\varepsilon_1}+\delta)$.
Indeed, the last coordinate of the center corresponding to $B$ is $s_B/(C_9r)$. Thus, for any two such balls $B,B'$, we have
\begin{equation*}
|t_B-t_{B'}|
\sim |s_B-s_{B'}|
\lesssim r(\delta^{\varepsilon_1}+\delta),
\end{equation*}
where we used $t_B,t_{B'}\sim T$.
The corresponding times $t_B$ therefore belong to an interval of length
\begin{equation*}
Cr(\delta^{\varepsilon_1}+\delta)
\lesssim r^{1-2\varepsilon^2+4\varepsilon^2\varepsilon_0}+r^{1/2+\varepsilon_0}
\lesssim r^{1-\varepsilon^2}.
\end{equation*}
By \eqref{eq:twoends-local-count}, at most $C\lambda r^{1/2}$ balls from $\mathcal Q(\iota)$ have centers in this interval. Hence
\begin{equation*}
\frac{|Y(\widetilde\ell_\iota)\cap J|}{|Y(\widetilde\ell_\iota)|}
\lesssim\frac{\lambda r^{1/2}\delta^{n+1}}{\lambda\beta/r^{n/2}}
\lesssim\frac{r^{(n+1)\varepsilon_0}}{\beta}
\lesssim r^{(n+1)\varepsilon_0-\varepsilon^4}
\lesssim\delta^{\varepsilon_2},
\end{equation*}
provided $\varepsilon_0$ is sufficiently small in terms of $\varepsilon$.
Thus all the normalized shadings satisfy the $(\varepsilon_1,\varepsilon_2)$-two-ends condition.

We select a maximal $C\delta$-separated family of directions among the lines $\widetilde\ell_\iota$,
and denote the corresponding lines by $\mathcal L_0$. Recall the slope of $\widetilde\ell_\iota$ is $(v-y_0+t_0\nabla_\xi\Phi(k,\xi_\theta))/{T}$.
By maximality, for every $\iota\in\mathcal W$, there exists a selected label $\iota_0=(k_0,\theta_0,v_0)$ such that
\begin{equation*}
\left|v+t_0\nabla_\xi\Phi(k,\xi_\theta)-v_0-t_0\nabla_\xi\Phi(k_0,\xi_{\theta_0})\right|
\lesssim T\delta
\lesssim T r^{\varepsilon_0}/r^{1/2}.
\end{equation*}
Therefore, for each fixed selected label $\iota_0$, all labels $\iota$ whose directions differ from that of $\iota_0$ by at most $C\delta$
have their centers lying in a ball of radius $O(Tr^{\varepsilon_0}/r^{1/2})$. This ball can be covered by $O(r^{n\varepsilon_0})$ balls of radius
$C_3 T/r^{1/2}$. By the definition of $P_r$, there are at most $O(r^{n\varepsilon_0}P_r)$ such labels, that is
\begin{equation}\label{eq:direction-count}
\#\mathcal W
\lesssim r^{n\varepsilon_0}P_r\#\mathcal L_0.
\end{equation}

We apply the geometric estimates of Section \ref{sec:preliminaries} to $\mathcal L_0$.
The family $\mathcal L_0$ is $1$-parallel by construction.
Since its directions are $C\delta$-separated, we have $\delta^{n}\#\mathcal{L}_0\lesssim 1$. 
By the $\Lambda$-density and Propositions \ref{prop:twoends-plane}, \ref{prop:twoends-three-d} and \ref{prop:twoends-general}, for every $\eta>0$, 
\begin{equation*}
\left|\bigcup_{\iota\in\mathcal W}Y(\widetilde\ell_\iota)\right|
\geq \left|\bigcup_{\widetilde\ell \in \mathcal{L}_0}Y(\widetilde\ell)\right|
\gtrsim \delta^{\eta+C_n\varepsilon_1}\Lambda\delta^{n}\#\mathcal{L}_0
\begin{cases}
\displaystyle \Lambda^{1/2},&n=1,\\[2pt]
\displaystyle \max\{\delta,\delta^{1/2}\Lambda^{3/4}\},&n=2,\\[2pt]
\displaystyle \delta^{n/2},&n\geq3.
\end{cases}
\end{equation*}
For $n=2$, the two bounds come from Propositions \ref{prop:twoends-three-d} and \ref{prop:twoends-general} respectively. 
For $n\geq 3$, we use $\delta^n\#\mathcal L_0\lesssim1$.
If $n\geq 3$ and $\Lambda\geq\delta^{1/4}$, Proposition \ref{prop:twoends-dense} further yields
\begin{align*}
\left|\bigcup_{\iota\in\mathcal W}Y(\widetilde\ell_\iota)\right|
&\gtrsim\delta^{\eta+C_n\varepsilon_1+3n/7}\Lambda^{(2n+12)/7}\left(\sum_{\ell\in\mathcal L_0}|Y(\ell)|\right)^{4/7}\\
&\gtrsim\delta^{\eta+C_n\varepsilon_1+3n/7}\Lambda\delta^n\#\mathcal L_0\Lambda^{(2n+9)/7},
\end{align*}
where the last inequality follows from the $\Lambda$-density and $\delta^{n}\#\mathcal{L}_0\lesssim 1$.
Recall $\delta\sim 1/r^{1/2-\varepsilon_0}$, $\Lambda\sim \lambda\beta/r^{n\varepsilon_0}$ and $u$ obeying $1/r^{u/2}\lesssim \lambda\beta\lesssim r^{n\varepsilon_0-u/2}$.
For $n\geq 3$, by direct calculation, the above estimate gives
\begin{equation*}
\Lambda\delta^{n}\#\mathcal{L}_0
\lesssim r^{n/4+(\eta+C_n\varepsilon_1)/2}\left|\bigcup_{\iota\in \mathcal{W}}Y(\widetilde{\ell}_\iota)\right|.
\end{equation*}
When $\Lambda\geq \delta^{1/4}$, we also have
\begin{equation*}
\Lambda\delta^{n}\#\mathcal{L}_0
\lesssim r^{3n/14+(2n+9)u/14+2n(n+3)\varepsilon_0/7+(\eta+C_n\varepsilon_1)/2}\left|\bigcup_{\iota\in \mathcal{W}}Y(\widetilde{\ell}_\iota)\right|.
\end{equation*}
Clearly, $n/4=3n/14+(2n+9)u/14$ if $u=u_n=n/(2(2n+9))<1/4$. Thus, the lower bound for $\lambda\beta$ ensures $\Lambda\geq \delta^{1/4}$ whenever $u\leq u_n$, 
provided $\varepsilon_0$ is sufficiently small. Therefore, we have
\begin{equation*}
\Lambda\delta^{n}\#\mathcal{L}_0
\lesssim r^{\gamma_n(u)/2+C'_n\varepsilon_0+(\eta+C_n\varepsilon_1)/2}\left|\bigcup_{\iota\in \mathcal{W}}Y(\widetilde{\ell}_\iota)\right|.
\end{equation*}
Taking $\eta=\varepsilon^{2}/10$ and selecting $\varepsilon_0$ sufficiently small in terms of $\varepsilon$, we have in all dimensions
\begin{equation}\label{eq:twoends-geometry}
\Lambda\delta^n\#\mathcal L_0
\lesssim r^{\gamma_n(u)/2+C_n\varepsilon^2}
\left|\bigcup_{\iota\in\mathcal W}Y(\widetilde\ell_\iota)\right|.
\end{equation}
For each $B\in\mathcal Q$, write
$\mathcal W_B=\{\iota\in\mathcal W:B\in\mathcal Q(\iota)\}$.
By \eqref{eq:twoends-count}, \eqref{eq:direction-count}, \eqref{eq:twoends-geometry} and \eqref{eq:shading-upper},
\begin{align*}
\sum_{B\in\mathcal Q}\#\mathcal W_B
&=\sum_{\iota\in\mathcal W}\#\mathcal Q(\iota)\\
&\lesssim\lambda\beta r^{1/2+n\varepsilon_0}P_r\#\mathcal L_0\\
&\lesssim P_r r^{(n+1)/2+n\varepsilon_0}\Lambda\delta^n\#\mathcal L_0\\
&\lesssim P_r\#\mathcal Q\,r^{\gamma_n(u)/2+c_n\varepsilon^2}.
\end{align*}
The last second inequality follows from $\lambda\beta r^{1/2}/(\Lambda\delta^{n})\sim r^{(n+1)/2}$.

Choose $C_n>c_n+2$ in \eqref{eq:good-multiplicity}, increasing $C_{\varepsilon,n}$ if necessary. Then
\begin{equation*}
\sum_{B\in\mathcal Q}\#\mathcal W_B
\leq\frac{M\#\mathcal Q}{r^{2\varepsilon^2}}.
\end{equation*}
Set
\begin{equation*}
\mathcal Q_3
=\{B\in\mathcal Q:\#\mathcal W_B\leq M\},
\qquad
X_3=\bigcup_{B\in\mathcal Q_3}B.
\end{equation*}
It follows that
\begin{equation*}
\#(\mathcal Q\setminus\mathcal Q_3)
\leq\frac1M
\sum_{B\in\mathcal Q}\#\mathcal W_B
\leq\frac{\#\mathcal Q}{r^{2\varepsilon^2}}.
\end{equation*}
Since all balls in $\mathcal Q$ are pairwise disjoint and have the same radius,
\begin{equation*}
|X\setminus X_3|
\leq\frac{|X|}{r^{2\varepsilon^2}}.
\end{equation*}
This proves the lemma.
\end{proof}

We now select $j\in [0,m]$ in Theorem \ref{thm:refined-decoupling}. Set
\begin{equation*}
u_n=\begin{cases}
\displaystyle 1, &n=1,\\[4pt]
\displaystyle \frac{2}{3}, &n=2,\\[4pt]
\displaystyle \frac{n}{2(2n+9)}, &n\geq3.
\end{cases}
\end{equation*}
For $m\geq 1$, we choose
\begin{equation*}
j=\begin{cases}
\displaystyle m, &1<a_0\leq 2,\\[4pt]
\displaystyle \min\left\{m,\frac{a_0n}{2(a_0-2)}\right\}, &a_0>2, a_0u_n\leq 4,\\[4pt]
\displaystyle 0, &a_0>2, a_0u_n>4.
\end{cases}
\end{equation*}
For $m=0$, we take $j=0$ and use $p_E$ in place of $p_0$. We place the detailed selection process in Appendix \ref{sec:p_0}.

\medskip
\noindent\textbf{Step 4. The two-ends case.}
\nopagebreak

We first establish an $L^2$ estimate. For each $B\in\mathcal Q$, recall $\mathcal W_B=\{\iota\in\mathcal W:B\in\mathcal Q(\iota)\}$.
Since $B$ lies in a unique $\mathbb{R}^{n}\times I_j$, we can rewrite $G_\tau=\sum_{\iota\in \mathcal{W}_B}\mathcal{P}_\iota 1_{\Pi_\iota^{\ast}}$. Thus
\begin{equation*}
G_\tau-E_{m+n}F_\tau[f_{\mathcal W_B}]
=-\sum_{\iota\in\mathcal W_B}\mathcal P_\iota1_{(\Pi_\iota^*)^c}
\qquad\text{on }\mathbb R^m\times B.
\end{equation*}
By \eqref{eq:packet-decay} and $\#\mathcal{W}\lesssim R^{m+n}$, for every $N_1,N_2>0$, we have
\begin{equation*}
\left|\sum_{\iota\in \mathcal{W}}\mathcal{P}_\iota(x,y,t) 1_{(\Pi_\iota^*)^c}(y,t)\right|
\lesssim \frac{R^{m+n}T^{n/2}M_{T,r}(f)}{r^{N_1\varepsilon_0}(1+\tau|x|)^{N_2}}
\lesssim \frac{R^{m+3n/2}M_{T,r}(f)}{r^{N_1\varepsilon_0}(1+\tau|x|)^{N_2}}.
\end{equation*}
Since $r>R^{\varsigma}$, we can choose $N_1$ sufficiently large such that $R^{m+3n/2}/r^{N_1\varepsilon_0}\lesssim 1/R^{N+n+2}$.
Taking $2N_2>m$, for every union $Y$ of balls from $\mathcal{Q}$, \eqref{eq:M-L2} yields
\begin{align*}
\int_{\mathbb{R}^{m}\times Y} \left(\sum_{\iota\in \mathcal{W}}|\mathcal{P}_\iota(x,y,t)| 1_{(\Pi_\iota^*)^c}(y,t)\right)^{2} \rmd x\rmd y \rmd t
&\lesssim \frac{M_{T,r}(f)^{2}}{R^{2(N+n+2)}}|Y|\int_{\mathbb{R}^{m}}\frac{\rmd x}{(1+\tau|x|)^{2N_2}}\\
&\lesssim \frac{1}{R^{2(N+n+2)-n-1}}\frac{M_{T,r}(f)^{2}}{\tau^{m}}\\
&\lesssim \frac{1}{R^{N}}\|F_\tau[f]\|_{L^2}^{2},
\end{align*}
which implies that those terms satisfy \eqref{eq:negligible-remainder}.
By Plancherel's identity, \eqref{eq:packet-subsum-L2} and \eqref{eq:twoends-count}, we have
\begin{equation}\label{eq:twoends-L2}
\begin{aligned}
\int_{\mathbb R^m\times Y}|G_\tau(x,y,t)|^2\rmd x\rmd y\rmd t
&\lesssim r^{1/2}\sum_{B\subset Y}\|F_\tau[f_{\mathcal W_B}]\|_{L^2}^2+\operatorname{Err}_2\\
&\lesssim r^{1/2}\sup_{\iota\in\mathcal W}\#\mathcal Q(\iota)\|F_\tau[f]\|_{L^2}^2+\operatorname{Err}_2\\
&\lesssim\lambda\beta r\|F_\tau[f]\|_{L^2}^2+\operatorname{Err}_2.
\end{aligned}
\end{equation}

To apply Theorem \ref{thm:refined-decoupling} where the geometric multiplicity bound holds, and to handle the remaining case by induction,
we localize the selected wave packets to smaller balls. Set
$r'={r^{1-\varepsilon^2}}/{100}$.
Choose a maximal $r'/2$-separated subset of the centers of $\mathcal Q$, and let $B_i$ be the balls of radius $r'$ centered at these points.
Clearly, $\{B_i\}$ is the bounded overlap cover of $\mathcal{Q}$ and each lies in a unique $\mathbb{R}^{n}\times I_j$.
By the maximality and $r^{1/2}\ll r'$, for any $B\in \mathcal{Q}$, there exists at least one ball $B_i$ such that $B\subset B_i$.

Let $\mathcal{W}_i=\{\iota\in \mathcal{W}: J\cap B_i\neq \varnothing \text{ for some }J\in \mathcal{J}_\lambda(\iota)\}$.
Since each $J$ has diameter $O(r^{1-\varepsilon^2})$, it meets at most $O(1)$ balls $B_i$.
Recall $\#\mathcal{J}_{\lambda}(\iota)\lesssim \beta$, so
\begin{equation*}
\sup_{\iota\in\mathcal W}\#\{i:\iota\in\mathcal W_i\}
\lesssim\beta.
\end{equation*}
Consequently,
\begin{equation*}
\sum_i\sum_{\iota\in\mathcal W_i}\|f_\iota^{(r,t_0)}\|_{L^2}^2
\lesssim\beta\sum_{\iota\in\mathcal W}\|f_\iota^{(r,t_0)}\|_{L^2}^2.
\end{equation*}
By \eqref{eq:packet-subsum-L2} and \eqref{eq:packet-average},
\begin{equation}\label{eq:cover-L2}
\sum_i\|F_\tau[f_{\mathcal W_i}]\|_{L^2}^2
\lesssim\beta\|F_\tau[f]\|_{L^2}^2,
\end{equation}
\begin{equation}\label{eq:cover-average}
M_{T,r'}(f_{\mathcal W_i})
\lesssim M_{T,r}(f).
\end{equation}
As $B_i$ lies in a unique interval $I_j$, the definition of $\mathcal W_i$ gives
\begin{equation}\label{eq:cover-localization}
G_\tau-E_{m+n}F_\tau[f_{\mathcal W_i}]
=-\sum_{\iota\in\mathcal W_i}\mathcal P_\iota 1_{(\Pi_\iota^*)^c}
\qquad\text{on }\mathbb R^m\times B_i.
\end{equation}
Similarly to \eqref{eq:twoends-L2}, the last sum satisfies \eqref{eq:negligible-remainder}.

Now suppose that $\beta\geq r^{\varepsilon^4}$. Since $r>R^{\varsigma}$, Lemma \ref{lem:good-balls} and the selections in Step 2 give
\begin{equation*}
\frac{\#(\mathcal Q_2\setminus\mathcal Q_3)}{\#\mathcal Q_2}
\lesssim\frac{(\log R)^C}{r^{2\varepsilon^2}}
\lesssim\frac{(\log r)^C}{r^{2\varepsilon^2}}
\leq \frac 12.
\end{equation*}
Since the local norms $\mathcal N_B(G_\tau)$ are comparable on $\mathcal Q_2$, it follows that
\begin{equation}\label{eq:good-ball-reduction}
\int_{\mathbb R^m\times X_2}|G_\tau(x,y,t)|^{p_0}\rmd x \rmd y \rmd t
\lesssim\int_{\mathbb R^m\times(X_2\cap X_3)}|G_\tau(x,y,t)|^{p_0}\rmd x \rmd y \rmd t.
\end{equation}

Assign each $B\in\mathcal Q_2\cap\mathcal Q_3$ to one $B_i$ containing it, and let $Z_i$ be the union of the assigned balls.
Then $Z_i\subset B_i$, the sets $Z_i$ are pairwise disjoint, and
$X_2\cap X_3=\cup_i Z_i$.
For every $B\subset Z_i$, if $\iota\in \mathcal{W}_i$ and $\Pi_\iota\cap C_5B\neq \varnothing$, it implies $B\in \mathcal{Q}(\iota)$.
Since $B\in X_3$, Lemma \ref{lem:good-balls} yields that
\begin{equation*}
\#\{\iota\in\mathcal W_i:\Pi_\iota\cap C_5B\neq\varnothing\}
\leq \#\{\iota\in\mathcal W:B\in\mathcal Q(\iota)\}\\
\leq M.
\end{equation*}
By \eqref{eq:cover-localization}, Theorem \ref{thm:refined-decoupling} and \eqref{eq:packet-L2},
\begin{equation}\label{eq:twoends-high1}
\begin{aligned}
&\int_{\mathbb R^m\times(X_2\cap X_3)}|G_\tau(x,y,t)|^{p_j^\ast}\rmd x \rmd y \rmd t\\
&\quad\lesssim \sum_i\int_{\mathbb R^m\times Z_i}
|E_{m+n}F_\tau[f_{\mathcal W_i}](x,y,t)|^{p_j^\ast}\rmd x \rmd y \rmd t+\operatorname{Err}_{p_j^\ast}\\
&\quad\lesssim r^{1+\varepsilon^2}\left(L^{m-j}\left(1+\frac{L}{r^{1/2}}\right)^{j}\right)^{2/(n+j)}
(\bar\beta^2M)^{2/(n+j)}\frac{\|\rho\|_{L^2}^2}{\tau^m}\sum_i\sum_{\iota\in\mathcal W_i}\|f_\iota^{(r,t_0)}\|_{L^2}^2
+\operatorname{Err}_{p_j^\ast}\\
&\quad\lesssim r^{1+\varepsilon^2}\beta\left(L^{m-j}\left(1+\frac{L}{r^{1/2}}\right)^{j}\right)^{2/(n+j)}
(\bar\beta^2M)^{2/(n+j)}\|F_\tau[f]\|_{L^2}^2
+\operatorname{Err}_{p_j^\ast}.
\end{aligned}
\end{equation}
On the other hand, \eqref{eq:directional-bound} and \eqref{eq:good-multiplicity} give
\begin{equation*}
\bar\beta^2M
\lesssim r^{C_n\varepsilon^2}\bar\beta^2P_r r^{\gamma_n(u)/2}
\lesssim r^{C_n\varepsilon^2}\frac{T^{n}}{r^{n-\gamma_n(u)/2}}M_{T,r}(f)^{2}.
\end{equation*}
Using $\beta\lesssim r^{\varepsilon^2}$, \eqref{eq:twoends-high1} becomes
\begin{equation}\label{eq:twoends-high2}
\begin{aligned}
&\int_{\mathbb R^m\times(X_2\cap X_3)}|G_\tau(x,y,t)|^{p_j^\ast}\rmd x \rmd y \rmd t\\
&\quad\lesssim r^{1+C_{m,n}\varepsilon^2}\left(L^{m-j}\left(1+\frac{L}{r^{1/2}}\right)^{j}\frac{T^{n}}{r^{n-\gamma_n(u)/2}}\right)^{2/(n+j)}\|F_\tau[f]\|_{L^{2}}^{2}M_{T,r}(f)^{p_j^{\ast}-2}
+\operatorname{Err}_{p_j^\ast}.
\end{aligned}
\end{equation}

Finally, set
\begin{equation*}
\vartheta
=\frac{p_0-2}{p_j^\ast-2}
=\frac{(n+j)(p_0-2)}{4}
\leq 1.
\end{equation*}
H\"older's inequality, \eqref{eq:twoends-high2} and \eqref{eq:twoends-L2} give
\begin{equation}\label{eq:twoends-interpolation}
\begin{aligned}
&\int_{\mathbb R^m\times(X_2\cap X_3)}|G_\tau(x,y,t)|^{p_0}\rmd x \rmd y \rmd t\\
&\quad\leq\left(\int_{\mathbb R^m\times(X_2\cap X_3)}|G_\tau(x,y,t)|^{p_j^\ast}\rmd x \rmd y \rmd t\right)^\vartheta
\left(\int_{\mathbb R^m\times(X_2\cap X_3)}|G_\tau(x,y,t)|^2\rmd x \rmd y \rmd t\right)^{1-\vartheta}\\
&\quad\lesssim r^{C_{m,n}\varepsilon^2}(L^{m}T^{n})^{p_0/2-1}\|F_\tau[f]\|_{L^{2}}^{2}M_{T,r}(f)^{p_0-2}\\
&\qquad \times \left(\frac{1}{L}+\frac{1}{r^{1/2}}\right)^{j(p_0-2)/2}r^{1-(2n-\gamma_n(u))(p_0-2)/4}(\lambda\beta)^{1-\vartheta}
+\operatorname{Err}_{p_0}.
\end{aligned}
\end{equation}
Since $r\leq R=L^{a_0}$ and $\lambda\beta\lesssim r^{n\varepsilon_0-u/2}$, the last factor in \eqref{eq:twoends-interpolation} satisfies
\begin{align*}
&\left(\frac{1}{L}+\frac{1}{r^{1/2}}\right)^{j(p_0-2)/2}r^{1-(2n-\gamma_n(u))(p_0-2)/4}(\lambda\beta)^{1-\vartheta}\\
&\qquad\lesssim r^{n\varepsilon_0}r^{1-u/2-(2n-\gamma_n(u)-nu/2+(\min\{1,2/a_0\}-u/2)j)(p_0-2)/4}.
\end{align*}
By our choice of $j$, we have
\begin{equation*}
\left(1-\frac{u}{2}\right)\frac{4}{p_0-2}\leq 2n-\gamma_n(u)-\frac{nu}{2}+\left(\min\left\{1,\frac{2}{a_0}\right\}-\frac{u}{2}\right)j, \qquad 0\leq u\leq 1.
\end{equation*}
Consequently, the last factor in \eqref{eq:twoends-interpolation} is bounded by $r^{n\varepsilon_0}$.
Choosing $\varepsilon_0$ sufficiently small in terms of $\varepsilon$ and taking sufficiently large decay orders in the remainder estimates, we conclude from
\eqref{eq:good-ball-reduction} and \eqref{eq:twoends-interpolation} that
\begin{equation}\label{eq:twoends-estimate}
\begin{aligned}
&\int_{\mathbb R^m\times X_2}|G_\tau(x,y,t)|^{p_0}\rmd x \rmd y \rmd t\\
&\qquad\lesssim r^{C_{m,n}\varepsilon^2}(L^{m}T^n)^{p_0/2-1}\|F_\tau[f]\|_{L^2}^2M_{T,r}(f)^{p_0-2}
+\operatorname{Err}_{p_0}.
\end{aligned}
\end{equation}

\medskip
\noindent\textbf{Step 5. Putting the estimates together.}
\nopagebreak

We complete the proof by induction on $r$. Fix a sufficiently large constant $A>1$,
such that $A>C_{m,n}$ in \eqref{eq:twoends-estimate}. We choose the parameters $\varepsilon,\varsigma$ at the beginning
of the proof sufficiently small so that
\begin{equation*}
\varsigma+A\varepsilon^2<\zeta.
\end{equation*}
We prove the stronger estimate
\begin{equation}\label{eq:scale-induction}
\begin{aligned}
&\int_{\mathbb R^m\times B_r(y_0,t_0)}|E_{m+n}F_\tau[f](x,y,t)|^{p_0}\rmd x\rmd y\rmd t\\
&\qquad\lesssim R^\varsigma r^{A\varepsilon^2}(L^{m}T^n)^{p_0/2-1}\|F_\tau[f]\|_{L^2}^2M_{T,r}(f)^{p_0-2}.
\end{aligned}
\end{equation}

For $r\leq R^\varsigma$, \eqref{eq:terminal} gives
\begin{equation*}
\begin{aligned}
&\int_{\mathbb R^m\times B_r(y_0,t_0)}|E_{m+n}F_\tau[f](x,y,t)|^{p_0}\rmd x\rmd y\rmd t\\
&\quad\lesssim r(L^{m}T^n)^{p_0/2-1}\|F_\tau[f]\|_{L^2}^2M_{T,r}(f)^{p_0-2}\\
&\quad\lesssim R^\varsigma r^{A\varepsilon^2}(L^{m}T^n)^{p_0/2-1}\|F_\tau[f]\|_{L^2}^2M_{T,r}(f)^{p_0-2}.
\end{aligned}
\end{equation*}
Thus \eqref{eq:scale-induction} holds.

Suppose that $r>R^\varsigma$ and that \eqref{eq:scale-induction} holds at every smaller radius. By \eqref{eq:packet-reduction},
\begin{equation*}
\begin{aligned}
&\int_{\mathbb R^m\times B_r(y_0,t_0)}|E_{m+n}F_\tau[f](x,y,t)|^{p_0}\rmd x\rmd y\rmd t\\
&\qquad\lesssim(\log R)^C\int_{\mathbb R^m\times X_2}|G_\tau(x,y,t)|^{p_0}\rmd x\rmd y\rmd t
+\operatorname{Err}_{p_0}.
\end{aligned}
\end{equation*}

If $\beta\geq r^{\varepsilon^4}$, by \eqref{eq:twoends-estimate},
\begin{equation*}
\begin{aligned}
&\int_{\mathbb R^m\times B_r(y_0,t_0)}|E_{m+n}F_\tau[f](x,y,t)|^{p_0}\rmd x\rmd y\rmd t\\
&\quad\lesssim (\log R)^C r^{C_{m,n}\varepsilon^2}(L^{m}T^n)^{p_0/2-1}\|F_\tau[f]\|_{L^2}^2M_{T,r}(f)^{p_0-2}
+\operatorname{Err}_{p_0}.
\end{aligned}
\end{equation*}
Since $r>R^\varsigma$ and $A>C_{m,n}$, we have
\begin{equation*}
\frac{(\log R)^C}{r^{(A-C_{m,n})\varepsilon^2}}
\lesssim \frac{(\log r)^C}{r^{(A-C_{m,n})\varepsilon^2}}
\lesssim 1.
\end{equation*}
Hence \eqref{eq:scale-induction} follows in this case.

Now suppose that $\beta<r^{\varepsilon^4}$. The covering $\{B_i\}$ constructed in Step 4 contains $X_2$.
Since $r'\leq r/4$, and $\supp\widehat{\psi_v}\subset B^{n}(0,c_1/r^{1/2})$, we have
\begin{equation*}
\supp \widehat{f_{\mathcal{W}_i}}(k,\cdot)\subset B^{n}(0,2+2c_1/r^{1/2})\subset B^{n}(0,2+c_1/(r')^{1/2}),
\end{equation*}
so the induction hypothesis applies.
By \eqref{eq:packet-reduction}, \eqref{eq:cover-localization}, the induction hypothesis \eqref{eq:scale-induction} at radius $r'$, and
\eqref{eq:cover-L2}, \eqref{eq:cover-average}, we obtain
\begin{align*}
&\int_{\mathbb R^m\times B_r(y_0,t_0)}|E_{m+n}F_\tau[f](x,y,t)|^{p_0}\rmd x\rmd y\rmd t\\
&\quad\lesssim (\log R)^C\sum_i\int_{\mathbb R^m\times B_i}|E_{m+n}F_\tau[f_{\mathcal W_i}](x,y,t)|^{p_0}\rmd x\rmd y\rmd t
+\operatorname{Err}_{p_0}\\
&\quad\lesssim (\log R)^CR^\varsigma(r')^{A\varepsilon^2}(L^{m}T^n)^{p_0/2-1}
\sum_i\|F_\tau[f_{\mathcal W_i}]\|_{L^2}^2M_{T,r'}(f_{\mathcal W_i})^{p_0-2}
+\operatorname{Err}_{p_0}\\
&\quad\lesssim (\log R)^CR^\varsigma(r')^{A\varepsilon^2}\beta(L^{m}T^n)^{p_0/2-1}
\|F_\tau[f]\|_{L^2}^2M_{T,r}(f)^{p_0-2}
+\operatorname{Err}_{p_0}.
\end{align*}
By the definition of $r'$ and $\beta<r^{\varepsilon^4}$,
\begin{equation*}
\beta(r')^{A\varepsilon^2}\leq r^{\varepsilon^4}\left(\frac{r^{1-\varepsilon^2}}{100}\right)^{A\varepsilon^2}
=\frac{r^{A\varepsilon^2}}{100^{A\varepsilon^2}r^{(A-1)\varepsilon^4}}.
\end{equation*}
Since $A>1$ and $r>R^\varsigma$, we have
\begin{equation*}
\frac{(\log R)^C\beta(r')^{A\varepsilon^2}}{r^{A\varepsilon^2}}
\lesssim \frac{(\log R)^C}{r^{(A-1)\varepsilon^4}}
\lesssim \frac{(\log R)^C}{R^{\varsigma(A-1)\varepsilon^4}}\longrightarrow 0,
\qquad R\to\infty.
\end{equation*}
Together with \eqref{eq:negligible-remainder}, this establishes \eqref{eq:scale-induction} and completes the induction.
Finally, since $r\leq T\leq R$ and $\varsigma+A\varepsilon^2<\zeta$, we have
\begin{equation*}
R^\varsigma r^{A\varepsilon^2}
\leq R^{\varsigma+A\varepsilon^2}
\leq R^\zeta.
\end{equation*}
This proves \eqref{eq:local-estimate}, and Theorem \ref{thm:main} follows from the reductions in Section \ref{sec:reductions}.
\end{proof}

\begin{proof}[Proof of Corollary \ref{cor:Euclidean}]
When $m=0$, the periodic variables and the continuous extension are absent. 
Taking $j=0$ in the preceding argument gives \eqref{eq:local-estimate} with $p_0=p_E$, as computed in Appendix \ref{sec:p_0}. 
The reductions in Section \ref{sec:reductions} then give \eqref{eq:Euclidean}.
\end{proof}

\section{Proof of Theorem \ref{thm:sharpness}}\label{sec:sharpness}
Rogers \cite[Section 2]{R} constructed examples to obtain necessary Sobolev regularity conditions for the Schrödinger local smoothing estimate on $\mathbb R^n$. 
Rogers and Seeger \cite[Section 2]{RS} considered the corresponding problem for the phases $|\xi|^{a_0}$ with $a_0>1$, 
while Chen, Guo, Shen and Yan \cite[Section 5]{CGSY} established necessary conditions on $\mathbb T^m\times\mathbb R^n$. 
We now prove Theorem \ref{thm:sharpness} for the phases considered in Theorem \ref{thm:main}. 

\begin{proof}
We begin with the necessary condition that $s\geq 0$. Let $e_1=(1,0,\cdots,0)\in \mathbb{R}^{n}$, and choose a nonzero nonnegative function $\psi\in C_c^{\infty}(B^{n}(0,1))$.
For $N\geq 2$, define
\begin{equation*}
\widehat{f}(k,\xi)=1_{\{k=0\}}\psi(N^{a_0}(\xi-Ne_1)).
\end{equation*}
Then
\begin{equation*}
f(x,y)=\frac{1}{N^{a_0n}}e(Ny_1)\psi^{\vee}\left(\frac{y}{N^{a_0}}\right). 
\end{equation*}
Hence, we have
\begin{equation}\label{eq:reg01}
\|f\|_{L^{p}_s}\lesssim N^{s}\|f\|_{L^{p}}\sim N^{s+a_0n(1/p-1)}.
\end{equation}
By the assumptions on $\Phi$, $|\nabla\Phi(\eta,\xi)|\lesssim |(\eta,\xi)|^{a_0-1}$ for $|(\eta,\xi)|\geq 1$. Thus, for $\xi\in B^{n}(0,1)$,
\begin{equation*}
\left|\Phi\left(0,Ne_1+\frac{\xi}{N^{a_0}}\right)-\Phi(0,Ne_1)\right|
\leq \left|\frac{\xi}{N^{a_0}}\right|\left|\nabla_\xi\Phi(0,Ne_1+\frac{\theta\xi}{N^{a_0}})\right|
\lesssim \frac{1}{N},
\end{equation*}
where $0\leq \theta\leq 1$. Hence, for $0\leq t\leq 1$ and $|y|\leq cN^{a_0}$,
\begin{equation*}
\left|\frac{y\cdot\xi}{N^{a_0}}+t\left(\Phi\left(0,Ne_1+\frac{\xi}{N^{a_0}}\right)-\Phi(0,Ne_1)\right)\right|\leq \frac{1}{6}.
\end{equation*}
provided $c>0$ is sufficiently small and $N$ sufficiently large. In this region, we obtain
\begin{align*}
&|U_\Phi(t)f(x,y)|\\
&\quad=\frac{1}{N^{a_0n}}\left|\int_{\mathbb{R}^{n}}\psi(\xi)e\left(\frac{y\cdot\xi}{N^{a_0}}+t\left(\Phi\left(0,Ne_1+\frac{\xi}{N^{a_0}}\right)-\Phi(0,Ne_1)\right)\right) d\xi\right|\\
&\quad\geq \frac{1}{N^{a_0n}}\int_{\mathbb{R}^{n}}\psi(\xi)\cos\left(2\pi \left(\frac{y\cdot\xi}{N^{a_0}}+t\left(\Phi\left(0,Ne_1+\frac{\xi}{N^{a_0}}\right)-\Phi(0,Ne_1)\right)\right)\right) d\xi\\
&\quad\geq \frac{1}{2N^{a_0n}}\int_{\mathbb{R}^{n}}\psi(\xi)\rmd \xi
\gtrsim \frac{1}{N^{a_0n}}. 
\end{align*}
Thus, we have
\begin{equation}\label{eq:reg02}
\|U_\Phi(t)f\|^{p}_{L^{p}(\mathbb{T}^{m}\times \mathbb{R}^{n}\times [0,1])}
\geq \int_0^{1}\int_{\mathbb{T}^{m}}\int_{|y|\leq cN^{a_0}}|U_\Phi(t)f(x,y)|^{p}\rmd y\rmd x\rmd t
\gtrsim N^{a_0n(1-p)}.
\end{equation}
Consequently, combining \eqref{eq:reg01} and \eqref{eq:reg02} with \eqref{eq:thm:main} implies
\begin{equation*}
N^{a_0n(1/p-1)}
\lesssim \|U_\Phi(t)f\|_{L^{p}(\mathbb{T}^{m}\times \mathbb{R}^{n}\times [0,1])}
\lesssim \|f\|_{L^{p}_s}\lesssim N^{s+a_0n(1/p-1)}.
\end{equation*}
Letting $N\to \infty$ gives $s\geq 0$.

We turn to \eqref{eq:regularity1}. We may assume $p>2$, since the case $p=2$ follows from $s\geq 0$.
Choose a nonzero, nonnegative $b\in C_c^{\infty}(B^{m+n}((0,e_1),c_0))$, with $0<c_0<1/4$. For $N\geq 2$, define
\begin{equation*}
K_{N,t}(x,y)=\sum_{k\in \mathbb{Z}^{m}}\int_{\mathbb{R}^{n}}b\left(\frac{k}{N},\frac{\xi}{N}\right)e(k\cdot x+y\cdot\xi+t\Phi(k,\xi))\rmd \xi.
\end{equation*}
Set $I=[1/4,3/4]$ and $p_{\ast}=2+4/(m+n)$. We have, for every $\varepsilon>0$,
\begin{equation}\label{eq:reg11}
\left\|K_{N,t}\right\|_{L^{p_{\ast}}(\mathbb{T}^{m}\times \mathbb{R}^{n}\times I)}\lesssim \frac{N^{(m+n)/2+\varepsilon}}{N^{(a_0-1)n(1/2-1/p_{\ast})}}.
\end{equation}
The proof of this estimate is given in Appendix \ref{sec:a-kernel}. By Plancherel's identity,
\begin{equation*}
\left\|K_{N,t}\right\|^{2}_{L^{2}(\mathbb{T}^{m}\times \mathbb{R}^{n}\times I)}
=|I|\sum_{k\in \mathbb{Z}^{m}}\int_{\mathbb{R}^{n}}\left|b\left(\frac{k}{N},\frac{\xi}{N}\right)\right|^{2}\rmd \xi
\sim N^{m+n}.
\end{equation*}
Let $p'=p/(p-1)$, and choose $0<\vartheta<1$ such that
\begin{equation*}
\frac{1}{2}=\frac{\vartheta}{p'}+\frac{1-\vartheta}{p_{\ast}}.
\end{equation*}
H\"older's inequality and \eqref{eq:reg11} imply
\begin{equation}\label{eq:reg12}
\left\|K_{N,t}\right\|_{L^{p'}(\mathbb{T}^{m}\times \mathbb{R}^{n}\times I)}
\geq \frac{\left\|K_{N,t}\right\|_{L^{2}(\mathbb{T}^{m}\times \mathbb{R}^{n}\times I)}^{1/\vartheta}}
{\left\|K_{N,t}\right\|_{L^{p_{\ast}}(\mathbb{T}^{m}\times \mathbb{R}^{n}\times I)}^{(1-\vartheta)/\vartheta}}
\gtrsim N^{(m+n)/2+(a_0-1)n(1/2-1/p)-\varepsilon}.
\end{equation}
Now let $g\in C_c^{\infty}(\mathbb{T}^{m}\times \mathbb{R}^{n})$, and define
\begin{equation*}
\widehat{f}(k,\xi)=b\left(\frac{k}{N},\frac{\xi}{N}\right)\widehat{g}(k,\xi).
\end{equation*}
By Young's inequality, we have
\begin{align*}
\|f\|_{L^{p}_s}
&=\left\|\left((1+|k|^{2}+|\xi|^{2})^{s/2}b\left(\frac{k}{N},\frac{\xi}{N}\right)\right)^{\vee}\ast g\right\|_{L^{p}}\\
&\leq \left\|\left((1+|k|^{2}+|\xi|^{2})^{s/2}b\left(\frac{k}{N},\frac{\xi}{N}\right)\right)^{\vee}\right\|_{L^{1}}\|g\|_{L^{p}}\\
&\lesssim N^{s}\|g\|_{L^{p}}.
\end{align*}
Applying \eqref{eq:thm:main} gives
\begin{equation}\label{eq:reg13}
\|U_\Phi(t)f\|_{L^{p}(\mathbb{T}^{m}\times \mathbb{R}^{n}\times [0,1])}
\lesssim N^{s}\|g\|_{L^{p}}.
\end{equation}
Note that 
\begin{equation*}
\partial_t U_\Phi(t)f=2\pi iU_\Phi(t)\left(\Phi(k,\xi)\widehat{f}(k,\xi)\right)^{\vee}.
\end{equation*}
Similarly, we have
\begin{equation}\label{eq:reg14}
\|\partial_t U_\Phi(t)f\|_{L^{p}(\mathbb{T}^{m}\times \mathbb{R}^{n}\times [0,1])}
\lesssim N^{s+a_0}\|g\|_{L^{p}}.
\end{equation}
By the fundamental theorem of calculus, for every $s,t\in [0,1]$, we have
\begin{align*}
\|U_\Phi(t)f\|_{L^{p}(\mathbb{T}^{m}\times \mathbb{R}^{n})}^{p}
&=\|U_\Phi(s)f\|_{L^{p}(\mathbb{T}^{m}\times \mathbb{R}^{n})}^{p}
+\int_{s}^{t}\frac{\rmd}{\rmd \tau} \|U_\Phi(\tau)f\|_{L^{p}(\mathbb{T}^{m}\times \mathbb{R}^{n})}^{p}\rmd \tau\\
&\leq \|U_\Phi(s)f\|_{L^{p}(\mathbb{T}^{m}\times \mathbb{R}^{n})}^{p}
+\int_{0}^{1}\left|\frac{\rmd}{\rmd \tau} \|U_\Phi(\tau)f\|_{L^{p}(\mathbb{T}^{m}\times \mathbb{R}^{n})}^{p}\right|\rmd \tau.
\end{align*}
Thus, integrating in $s$ and applying H\"older's inequality, \eqref{eq:reg13} and \eqref{eq:reg14}, we obtain
\begin{align*}
\sup_{0\leq t\leq 1}\|U_\Phi(t)f\|_{L^{p}(\mathbb{T}^{m}\times \mathbb{R}^{n})}^{p}
&\lesssim \|U_\Phi(t)f\|_{L^{p}(\mathbb{T}^{m}\times \mathbb{R}^{n}\times [0,1])}^{p}\\
&\quad+\|U_\Phi(t)f\|_{L^{p}(\mathbb{T}^{m}\times \mathbb{R}^{n}\times [0,1])}^{p-1}\|\partial_t U_\Phi(t)f\|_{L^{p}(\mathbb{T}^{m}\times \mathbb{R}^{n}\times [0,1])}\\
&\lesssim N^{sp}\|g\|_{L^{p}}^{p}+N^{s(p-1)}N^{s+a_0}\|g\|_{L^{p}}^{p}\\
&\lesssim N^{sp+a_0}\|g\|_{L^{p}}^{p}.
\end{align*}
By the definition of $\widehat{f}$, we have $U_\Phi(t)f=K_{N,t}\ast g$. Hence, Bernstein's inequality gives
\begin{align*}
\sup_{t\in I}\left|\left(K_{N,t}\ast g\right)(0,0)\right|
&=\sup_{t\in I}\left|U_\Phi(t)f(0,0)\right|\\
&\lesssim N^{(m+n)/p}\sup_{0\leq t\leq 1}\left\|U_\Phi(t)f\right\|_{L^{p}}\\
&\lesssim N^{s+(m+n+a_0)/p}\|g\|_{L^{p}}.
\end{align*}
Taking the supremum over $g$ with $\|g\|_{L^{p}}=1$ and using duality, we obtain
\begin{equation*}
\|K_{N,t}\|_{L^{p'}(\mathbb{T}^{m}\times \mathbb{R}^{n}\times I)}
\lesssim \sup_{t\in I}\|K_{N,t}\|_{L^{p'}(\mathbb{T}^{m}\times \mathbb{R}^{n})}
\lesssim N^{s+(m+n+a_0)/p}.
\end{equation*}
Combining this with \eqref{eq:reg12} yields
\begin{equation*}
N^{(m+n)/2+(a_0-1)n(1/2-1/p)-\varepsilon}
\lesssim \|K_{N,t}\|_{L^{p'}(\mathbb{T}^{m}\times \mathbb{R}^{n}\times I)}
\lesssim N^{s+(m+n+a_0)/p}.
\end{equation*}
Letting $N\to \infty$ and then $\varepsilon\to 0$ proves \eqref{eq:regularity1}.

It remains to prove \eqref{eq:regularity2}. The case $a_0=2$ follows from $s\geq 0$, so assume $1<a_0<2$ and set $K=N^{2-a_0}$.
Choose a nonnegative function $\chi_1\in C_c^{\infty}(B^{m+n}(0,1))$, positive near the origin, and define
\begin{equation*}
\widehat f(k,\xi)
=\frac1{K^{(m+n)/2}}e\left(-\frac12\Phi(k,\xi)\right)\chi_1\left(\frac{k}{K},\frac{\xi-Ne_1}{K}\right).
\end{equation*}
For
\begin{equation*}
\Psi_N(\eta,\xi)
=\frac{\Phi(K\eta,Ne_1+K\xi)-\Phi(0,Ne_1)}K-\nabla\Phi(0,Ne_1)\cdot(\eta,\xi),
\end{equation*}
we have $\Psi_N(0,0)=0$, $\nabla\Psi_N(0,0)=0$ and
\begin{equation*}
\nabla^2\Psi_N(\eta,\xi)\sim KN^{a_0-2}I=I,
\end{equation*}
\begin{equation*}
\left|D^\beta\Psi_N(\eta,\xi)\right|
\lesssim\frac1{N^{(a_0-1)(|\beta|-2)}}\lesssim1,
\qquad |\beta|\geq3.
\end{equation*}
Let $g$ be the function on $\mathbb{R}^{m+n}$ with
\begin{equation*}
\widehat g(\eta,\xi)
=\frac1{K^{(m+n)/2}}e\left(-\frac12\Phi(\eta,\xi)\right)\chi_1\left(\frac\eta K,\frac{\xi-Ne_1}K\right).
\end{equation*}
Changing variables and applying Corollary \ref{cor:kernel}, we get
\begin{align*}
|g(x,y)|
&=K^{(m+n)/2}\left|\int_{\mathbb{R}^{m+n}}\chi_1(\eta,\xi)e\left(-\frac K2\Psi_N(\eta,\xi)\right)\right.\\
&\qquad\left.\times e\left(K\left((x,y)-\frac12\nabla\Phi(0,Ne_1)\right)\cdot(\eta,\xi)\right)\rmd\eta\rmd\xi\right|\\
&\lesssim \frac1{\left(1+\left|(x,y)-\nabla\Phi(0,Ne_1)/2\right|\right)^J},
\end{align*}
where $J$ is sufficiently large. Poisson summation gives
\begin{equation*}
\sum_{\ell\in \mathbb{Z}^{m}}g(x+\ell,y)
=\sum_{k\in \mathbb{Z}^{m}}\widehat{g(\cdot,y)}(k)e(k\cdot x)
=\sum_{k\in \mathbb{Z}^{m}}\int_{\mathbb{R}^{n}}\widehat{g}(k,\xi)e(k\cdot x+y\cdot\xi)\rmd \xi
=f(x,y).
\end{equation*}
Thus, 
\begin{equation*}
|f(x,y)|
\leq\sum_{\ell\in\mathbb Z^m}|g(x+\ell,y)|
\lesssim\sum_{\ell\in\mathbb Z^m}
\frac1{\left(1+\left|(x+\ell,y)-\nabla\Phi(0,Ne_1)/2\right|\right)^J}
\end{equation*}
Consequently, we obtain
\begin{equation}\label{eq:reg21}
\|f\|_{L^{p}_s}\lesssim N^{s}\|f\|_{L^{p}}\lesssim N^{s}.
\end{equation}
For $t=1/2+t'$, we have
\begin{align*}
&U_\Phi(t)f(x,y)\\
&\quad =\frac{1}{K^{(m+n)/2}}\sum_{k\in \mathbb{Z}^{m}}\int_{\mathbb{R}^{n}}\chi_1\left(\frac{k}{K},\frac{\xi-Ne_1}{K}\right)
e(k\cdot x+y\cdot\xi+t'\Phi(k,\xi))\rmd \xi\\
&\quad =\frac{K^{n}}{K^{(m+n)/2}}e\left(Ny_1+t'\Phi(0,Ne_1)\right)\sum_{k\in \mathbb{Z}^{m}}\int_{\mathbb{R}^{n}}\chi_1\left(\frac{k}{K},\xi\right)\\
&\qquad \times e\left(k\cdot\left(x+t'\nabla_\eta\Phi(0,Ne_1)\right)+K(y+t'\nabla_\xi\Phi(0,Ne_1))\cdot\xi+Kt'\Psi_N\left(\frac{k}{K},\xi\right)\right)\rmd \xi.
\end{align*}
Since $\Psi_N(0,0)=\nabla\Psi_N(0,0)=0$ and $\nabla^{2}\Psi_N$ is uniformly bounded, we have
\begin{equation*}
\left|\Psi_N(\eta,\xi)\right|\lesssim 1, \qquad (\eta,\xi)\in \supp \chi_1.
\end{equation*}
Read $x+t'\nabla_\eta\Phi(0,Ne_1)$ modulo $\mathbb Z^m$, using its representative of smallest Euclidean norm.
Hence, for $|t'|\leq c_1/K$, $|x+t'\nabla_\eta\Phi(0,Ne_1)|\leq c_1/K$ and $|y+t'\nabla_\xi\Phi(0,Ne_1)|\leq c_1/K$,
we have, whenever $(k/K,\xi)\in \supp\chi_1$,
\begin{align*}
&\left|k\cdot\left(x+t'\nabla_\eta\Phi(0,Ne_1)\right)+K(y+t'\nabla_\xi\Phi(0,Ne_1))\cdot\xi+Kt'\Psi_N\left(\frac{k}{K},\xi\right)\right|\\
&\quad \lesssim K(|x+t'\nabla_\eta\Phi(0,Ne_1)|+|y+t'\nabla_\xi\Phi(0,Ne_1)|+|t'|)\\
&\quad \lesssim c_1\leq \frac{1}{6},
\end{align*}
provided $c_1>0$ is sufficiently small. Consequently, we obtain
\begin{equation*}
|U_\Phi(t)f(x,y)|
\gtrsim \frac{K^{n}}{2K^{(m+n)/2}}\sum_{k\in \mathbb{Z}^{m}}\int_{\mathbb{R}^{n}}\chi_1\left(\frac{k}{K},\xi\right)\rmd \xi
\gtrsim K^{(m+n)/2}.
\end{equation*}
Integrating over the above region yields
\begin{equation}\label{eq:reg22}
\|U_\Phi(t)f\|_{L^{p}(\mathbb{T}^{m}\times \mathbb{R}^{n}\times [0,1])}^{p}\gtrsim \frac{K^{(m+n)p/2}}{K^{m+n+1}}=K^{(m+n)(p/2-1)-1}.
\end{equation}
Therefore, combining \eqref{eq:reg21} and \eqref{eq:reg22} with \eqref{eq:thm:main} implies
\begin{align*}
N^{(2-a_0)((m+n)(1/2-1/p)-1/p)}
&=K^{(m+n)(1/2-1/p)-1/p}\\
&\lesssim \|U_\Phi(t)f\|_{L^{p}(\mathbb{T}^{m}\times \mathbb{R}^{n}\times [0,1])}\\
&\lesssim \|f\|_{L^{p}_s}\lesssim N^{s}.
\end{align*}
Letting $N\to \infty$ proves \eqref{eq:regularity2}.

For the estimate to hold for all $s$ in the range \eqref{eq:regularity1}, the necessary condition $s\geq 0$ requires
\begin{equation*}
p\geq 2+\frac{2a_0}{m+a_0n}.
\end{equation*}
For $1<a_0\leq 2$, \eqref{eq:regularity2} gives the stronger restriction
\begin{equation*}
p\geq 2+\frac{4}{m+2n}.
\end{equation*}
Thus \eqref{eq:necessary-p} follows.
\end{proof}

\appendix
\section{Optimization of the exponent $p_0$}\label{sec:p_0}
We explain the choice of  $j$ and the resulting exponent $p_0$ in Step $4$ of the proof of Theorem \ref{thm:local-estimate}.
For $0\leq j\leq m$, consider an exponent $2<p\leq p_j^{\ast}$. To obtain an $L^{p}$ estimate by 
interpolating the $L^{2}$ estimate \eqref{eq:twoends-L2} and the $L^{p_j^{\ast}}$ estimate \eqref{eq:twoends-high2}, we take
\begin{equation*}
\vartheta=\frac{p-2}{p_j^{\ast}-2}=\frac{(n+j)(p-2)}{4}.
\end{equation*}
As in \eqref{eq:twoends-interpolation}, after extracting the desired factor, we are left with
\begin{equation*}
\left(\frac{1}{L}+\frac{1}{r^{1/2}}\right)^{j(p-2)/2}r^{1-(2n-\gamma_n(u))(p-2)/4}(\lambda\beta)^{1-\vartheta}.
\end{equation*}
Our aim is to bound this factor by $r^{n\varepsilon_0}$ uniformly for $u\in [0,1]$, 
while making $p$ as small as possible. Set $\alpha=\min\{1,2/a_0\}$ and $q=4/(p-2)$. Using $r\leq L^{a_0}$ and $\lambda \beta\lesssim r^{n\varepsilon_0-u/2}$,
the desired bound follows if
\begin{equation*}
\left(1-\frac{u}{2}\right)q\leq 2n-\gamma_n(u)-\frac{nu}{2}+\left(\alpha-\frac{u}{2}\right)j, \qquad 0\leq u\leq 1.
\end{equation*}
Moreover, $p\leq p_j^{\ast}$ is equivalent to $q\geq n+j$. Since minimizing $p$ is equivalent to maximizing $q$, we are led to the following linear programming problem:
\begin{equation}\label{eq:linear-optimization}
\begin{aligned}
&\text{max }q;\\
&\text{s. t. } \quad 0\leq j\leq m,\\
&\phantom{\text{s. t. }} \quad q\geq n+j,\\
&\phantom{\text{s. t. }} \quad \left(1-\frac{u}{2}\right)q\leq 2n-\gamma_n(u)-\frac{nu}{2}+\left(\alpha-\frac{u}{2}\right)j,\qquad 0\leq u\leq 1.
\end{aligned}
\end{equation}
For $n\geq 2$, let $u_n$ be the point where the two terms in the minimum defining $\gamma_n(u)$ coincide, and set $u_1=1$. Explicitly,
\begin{equation*}
u_n=\begin{cases}
\displaystyle 1, &n=1,\\[2pt]
\displaystyle \frac{2}{3}, &n=2,\\[2pt]
\displaystyle \frac{n}{2(2n+9)}, &n\geq3.
\end{cases}
\end{equation*}
Since $\gamma_n$ is piecewise affine, it suffices to verify that the last constraint in \eqref{eq:linear-optimization} holds at $u=0,1,u_n$. 
Taking $u=1$ in the last constraint of \eqref{eq:linear-optimization} and using $q\geq n+j$, we obtain
\begin{equation}\label{eq:constraint-j}
2(1-\alpha)j\leq n.
\end{equation}
Note that for every $j$ obeying the above constraint, the function
\begin{equation*}
\frac{2n-\gamma_n(u)-nu/2+(\alpha-u/2)j}{1-u/2}
\end{equation*}
is nonincreasing on $[0,u_n]$ and nondecreasing on $[u_n,1]$. Consequently, the constraint at $u=u_n$ is the most restrictive, 
and for each fixed $j$, the largest admissible value of $q$ is
\begin{equation}\label{eq:solution-q}
q=\frac{n(3-u_n)+(2\alpha-u_n)j}{2-u_n}.
\end{equation}
This choice also satisfies $q\geq n+j$, since
\begin{equation*}
q-(n+j)=\frac{n-2(1-\alpha)j}{2-u_n}\geq 0.
\end{equation*}
It remains to maximize the right-hand side of \eqref{eq:solution-q} over $0\leq j\leq m$ subject to \eqref{eq:constraint-j}.
Since this expression is nondecreasing in $j$ when $2\alpha-u_n\geq 0$ and decreasing otherwise, we take
\begin{equation*}
j=\begin{cases}
\displaystyle m, &1<a_0\leq 2,\\[4pt]
\displaystyle \min\left\{m,\frac{a_0n}{2(a_0-2)}\right\}, &a_0>2, a_0u_n\leq 4,\\[4pt]
\displaystyle 0, &a_0>2, a_0u_n>4.
\end{cases}
\end{equation*}
Substituting this choice of $j$ into \eqref{eq:solution-q}, we obtain
\begin{equation*}
p_0=2+\frac{4(2-u_n)}{n(3-u_n)+(2\alpha-u_n)j}.
\end{equation*}
This gives the exponent $p_0$ stated in Theorem \ref{thm:main}.

\section{A kernel estimate}\label{sec:a-kernel}
We prove \eqref{eq:reg11}. Recall that $p_\ast=2+4/(m+n)$, $I=[1/4,3/4]$ and 
\begin{equation*}
K_{N,t}(x,y)=\sum_{k\in \mathbb{Z}^{m}}\int_{\mathbb{R}^{n}}b\left(\frac{k}{N},\frac{\xi}{N}\right)e(k\cdot x+y\cdot\xi+t\Phi(k,\xi))\rmd \xi,
\end{equation*}
where $N\geq 2$. Let $f$ on $\mathbb{T}_N^{m}\times \mathbb{R}^{n}$ satisfy $\widehat{f}(k,\xi)=b(k,\xi)$, $k\in N^{-1}\mathbb{Z}^{m}$.
As in Section \ref{sec:reductions}, we use the normalized phase $\Phi(N\eta,N\xi)/N^{a_0}$, still denoted by $\Phi$.
Let $F_\tau[f]$ be as in Section \ref{sec:reductions}, with $0<\tau\leq c/N^{a_0}$. The scaling in Lemma \ref{lem:LP-reduction} and Lemma \ref{lem:Ftau-limit} give
\begin{align*}
&\|K_{N,t}\|_{L^{p_\ast}(\mathbb{T}^{m}\times \mathbb{R}^{n}\times I)}^{p_\ast}\\
&\quad=\frac{N^{np_\ast}}{N^{m+n+a_0}}\|\mathcal{E}_{m,n}(t)f\|_{L^{p_\ast}(\mathbb{T}_N^{m}\times \mathbb{R}^{n}\times N^{a_0}I)}^{p_\ast}\\
&\quad=\frac{N^{np_\ast}}{N^{m+n+a_0}\left\|\rho^{\vee}\right\|_{L^{p_\ast}}^{p_\ast}}\lim_{\tau\to 0}(\tau N)^{m}
\int_{\mathbb{R}^{m+n}\times N^{a_0}I}\left|E_{m+n}F_\tau[f](x,y,t)\right|^{p_\ast}\rmd x\rmd y\rmd t\\
\end{align*}
Hence, it suffices to show that, for every $\varepsilon>0$,
\begin{equation*}
\lim_{\tau\to 0}(\tau N)^m\int_{\mathbb R^{m+n}\times N^{a_0}I}|E_{m+n}F_\tau[f](x,y,t)|^{p_\ast}\rmd x\rmd y\rmd t
\lesssim N^{a_0+\varepsilon}N^{2m}\left(\frac{N^m}{N^{a_0n}}\right)^{2/(m+n)}.
\end{equation*}
As in the proof of Lemma \ref{lem:time-reduction}, we only need to prove that
\begin{equation}\label{eq:K-estimate}
\lim_{\tau\to 0}(\tau N)^m\int_{\mathbb R^{m}\times B_r(y_0,t_0)}|E_{m+n}F_\tau[f](x,y,t)|^{p_\ast}\rmd x\rmd y\rmd t
\lesssim r^{1+\varepsilon}N^{2m}\left(\frac{N^m}{r^{n}}\right)^{2/(m+n)},
\end{equation}
where $r\sim t_0\sim N^{a_0}$. We consider the two ranges of $a_0$ separately.

Suppose first that $a_0\geq 2$. Let $\mathcal{Q}$ be a family of $r^{1/2}$-balls covering $B_r(y_0,t_0)$, whose centers lie in $B_{C_4r}(y_0,t_0)$.
Denote $X=\cup_{B\in \mathcal{Q}}B$. We use the wave packet decomposition of $f$ at $(r,t_0)$ from Section \ref{sec:wave-packets}.
For $\iota=(k,\theta,v)$, Plancherel's identity gives
\begin{equation*}
\|f_\iota^{(r,t_0)}\|_{L^{2}}^{2}
=\int_{\mathbb{R}^{n}}|\psi_v(y)|^{2}\left|(e(t_0\Phi(k,\cdot))b(k,\cdot)\varphi_\theta)^{\vee}(y)\right|^{2}\rmd y.
\end{equation*}
Since $t_0\sim r$, by Corollary \ref{cor:kernel}, we have for every $J>0$,
\begin{equation*}
|(e(t_0\Phi(k,\cdot))b(k,\cdot)\varphi_\theta)^{\vee}(y)|
\lesssim \frac{1}{r^{n/2}\left(1+|y+t_0\nabla_\xi\Phi(k,\xi_\theta)|/r^{1/2}\right)^{J}}.
\end{equation*}
Note that 
\begin{equation*}
1+\frac{|v+t_0\nabla_\xi\Phi(k,\xi_\theta)|}{r^{1/2}}
\leq \left(1+\frac{|y-v|}{r^{1/2}}\right)\left(1+\frac{|y+t_0\nabla_\xi\Phi(k,\xi_\theta)|}{r^{1/2}}\right).
\end{equation*}
Applying \eqref{eq:wp-decay}, we obtain
\begin{equation}\label{eq:K-decay}
\frac{\|f_\iota^{(r,t_0)}\|_{L^2}}{r^{n/4}}
\lesssim \frac1{r^{n/2}}
\frac1{\left(1+{|v+t_0\nabla_\xi\Phi(k,\xi_\theta)|}/{r^{1/2}}\right)^J}.
\end{equation}
As in Step $1$ of Section \ref{sec:local-estimate}, we retain only the labels satisfying
\begin{equation*}
|v-y_0|\leq C_4r,\qquad
|v+t_0\nabla_\xi\Phi(k,\xi_\theta)|\leq r^{1/2+\varepsilon_0},
\end{equation*}
and denote their collection by $\mathcal W$. By \eqref{eq:K-decay} and \eqref{eq:packet-decay}, the omitted packets have arbitrarily rapid decay on $X$.
If $\mathcal{W}$ is empty, the required estimate follows. Otherwise, set
$\bar \beta=\max_{\iota\in \mathcal{W}}\|f_\iota^{(r,t_0)}\|_{L^{2}}/r^{n/4}$. 
Clearly, by \eqref{eq:K-decay}, 
\begin{equation}\label{eq:K-beta1}
\bar \beta^{2} \lesssim \frac1{r^{n}}.
\end{equation}
Applying \eqref{eq:packet-L2} and \eqref{eq:Ftau-L2}, we have
\begin{equation}\label{eq:K-norms}
\sum_{\iota\in\mathcal W}\|f_\iota^{(r,t_0)}\|_{L^2}^2
\lesssim\sum_{k\in N^{-1}\mathbb Z^m}\|b(k,\cdot)\|_{L^2}^2
\lesssim N^m.
\end{equation}
As in Theorem \ref{thm:refined-decoupling}, fix a sufficiently large constant $C_5>0$ and let
\begin{equation*}
M=\max\left\{1,\sup_{B\in \mathcal{Q}}\#\{\iota\in \mathcal{W}:\Pi_\iota\cap C_5B\neq \varnothing\}\right\}.
\end{equation*}
Fix a $B\in \mathcal{Q}$ with center at $(y_B,t_B)$ and suppose that $\Pi_\iota\cap C_5B\neq \varnothing$.
Choose $(y,t)\in \Pi_\iota\cap C_5B$. By the definition of $\Pi_\iota$ and the retained condition,
\begin{align*}
|y_B+t_B\nabla_\xi\Phi(k,\xi_\theta)|
&\leq |y_B-y|+|t_B-t||\nabla_\xi \Phi(k,\xi_\theta)|\\
&\quad+|y-v+(t-t_0)\nabla_\xi\Phi(k,\xi_\theta)|+|v+t_0\nabla_\xi\Phi(k,\xi_\theta)|\\
&\lesssim r^{1/2}+r^{1/2+\varepsilon_0}\lesssim r^{1/2+\varepsilon_0}.
\end{align*}  
Since $t_B\sim r$, it implies
\begin{equation*}
\left|\nabla_\xi \Phi(k,\xi_\theta)+\frac{y_B}{t_B}\right|\lesssim \frac{1}{r^{1/2-\varepsilon_0}}.
\end{equation*}
By the assumption $\nabla_\xi^{2}\Phi(k,\xi)\sim I_n$, any $\xi_\theta$ and $\xi_\theta'$ obeys that
\begin{equation*}
|\xi_\theta-\xi_\theta'|\lesssim |\nabla_\xi \Phi(k,\xi_\theta)-\nabla_\xi \Phi(k,\xi_\theta')|.
\end{equation*}
Since $\xi_\theta$ lies on a $c/r^{1/2}$-lattice, the number of $\theta$ such that $\Pi_\iota\cap C_5B\neq \varnothing$ is bounded by $r^{n\varepsilon_0}$.
Similarly, the number of $v$ is also bounded by $r^{n\varepsilon_0}$. Hence, we have
\begin{equation}\label{eq:K-M1}
M\lesssim N^{m}r^{2n\varepsilon_0}.
\end{equation}
Applying Theorem \ref{thm:refined-decoupling} with $j=m$ and loss $\varepsilon/2$ , multiplying by $(\tau N)^m$, 
and using \eqref{eq:K-beta1}, \eqref{eq:K-norms} and \eqref{eq:K-M1}, we obtain
\begin{align*}
&(\tau N)^m\int_{\mathbb R^m\times X}|E_{m+n}F_\tau[f_{\mathcal W}](x,y,t)|^{p_\ast}\rmd x\rmd y\rmd t\\
&\quad\lesssim r^{1+\varepsilon/2}N^m(\bar\beta^2M)^{2/(m+n)}
\sum_{\iota\in\mathcal W}\|f_\iota^{(r,t_0)}\|_{L^2}^2+\frac1{N^A}\\
&\quad\lesssim r^{1+\varepsilon/2+4n\varepsilon_0/(m+n)}N^{2m}\left(\frac{N^m}{r^{n}}\right)^{2/(m+n)}+\frac1{N^A}.
\end{align*}
The remainder is $O(1/N^{A})$ for any $A>0$, by \eqref{eq:Ftau-L2} and $M_{T,r}(f)\lesssim\|f\|_{L^\infty}\lesssim N^m$, with $T\sim r$.
Choosing $\varepsilon_0$ sufficiently small and restoring the omitted packets proves \eqref{eq:K-estimate} for $a_0\geq2$.

Now suppose that $1<a_0<2$. By Lemma \ref{lem:Ftau-limit}, it suffices to estimate $\mathcal{E}_{m,n}f$ on $\mathbb{T}^{m}_N\times B_{r}(y_0,t_0)$. Set
$\Omega=[-Cr,Cr]^{m}\times B_r(y_0,t_0)$, where $C>0$ is a sufficiently large fixed constant. By periodicity,
\begin{equation}\label{eq:K-periodic-average}
\int_{\mathbb T_N^m\times B_r(y_0,t_0)}
|\mathcal E_{m,n}(t)f(x,y)|^{p_\ast}\rmd x\rmd y\rmd t
\lesssim\left(\frac Nr\right)^m \int_\Omega|\mathcal E_{m,n}(t)f(x,y)|^{p_\ast}\rmd x\rmd y\rmd t.
\end{equation}
Poisson summation gives
\begin{equation*}
\mathcal E_{m,n}(t)f(x,y)
=N^m\sum_{\ell\in\mathbb Z^m}\int_{\mathbb R^{m+n}}b(\eta,\xi)e\bigl((x+N\ell)\cdot\eta+y\cdot\xi+t\Phi(\eta,\xi)\bigr)\rmd\eta\rmd\xi.
\end{equation*}
Since $t\sim r$ on $\Omega$, Corollary \ref{cor:kernel} shows that the terms
with $|\ell|>C'r/N$ are $O(1/N^{A})$ on $\Omega$, for every $A>0$, if $C'$ is sufficiently large.
Define $h$ on $\mathbb R^{m+n}$ by
\begin{equation*}
\widehat h(\eta,\xi)
=N^m b(\eta,\xi)\sum_{|\ell|\leq C'r/N}e(N\ell\cdot\eta).
\end{equation*}
Then
\begin{equation}\label{eq:K-poisson}
\mathcal E_{m,n}(t)f(x,y)
=E_{m+n}\widehat h(x,y,t)+O(1/N^{A}).
\end{equation}
Cover $\Omega$ by a bounded number of balls in $\mathbb R^{m+n+1}$
of radius comparable to $r$, with time coordinates comparable to $r$ on their fixed enlargements.
We apply the wave packet decomposition of Section \ref{sec:wave-packets} and
Theorem \ref{thm:refined-decoupling} with no periodic variables and Euclidean dimension $m+n$.
Fix one such ball, denote its center by $(z_0,t_0)$, and use $r$ for its radius.
Let $\mathcal{Q}$ be a covering of this ball by $r^{1/2}$-balls, and set $X=\cup_{B\in \mathcal{Q}}B$.
Its wave packet labels are $\iota=(\theta,v)$, where $\theta$ is centered at
$(\eta_\theta,\xi_\theta)$ and $v\in r^{1/2}\mathbb Z^{m+n}$.
As in the derivation of \eqref{eq:K-decay}, for every $J>0$,
\begin{equation}\label{eq:K-euclidean-decay}
\frac{\|h_\iota^{(r,t_0)}\|_{L^2}}{r^{(m+n)/4}}
\lesssim \frac{N^m}{r^{(m+n)/2}}\sum_{|\ell|\leq C'r/N}\frac{1}{\left(1+|v+(N\ell,0)+t_0\nabla\Phi(\eta_\theta,\xi_\theta)|/r^{1/2}\right)^J}.
\end{equation}
Retain only those labels satisfying
\begin{equation*}
|v-z_0|\leq C_4r,
\qquad |v+(N\ell,0)+t_0\nabla\Phi(\eta_\theta,\xi_\theta)|\leq r^{1/2+\varepsilon_0}
\end{equation*}
for some $|\ell|\leq C'r/N$, and denote their collection by $\mathcal W$.
The omitted packets have arbitrarily rapid decay on the ball by
\eqref{eq:K-euclidean-decay} and \eqref{eq:packet-decay}, applied with no periodic variables.
With $\bar\beta$ and $M$ defined as in the preceding case for this Euclidean wave packet family,
\begin{equation}\label{eq:K-euclidean-beta}
\bar\beta^2\lesssim\frac{N^{2m}}{r^{m+n}}.
\end{equation}
If a retained tube intersects $C_5B$, where $B$ is an $r^{1/2}$-ball centered
at $(z_B,t_B)$, the tube definition and the retained condition imply
\begin{equation*}
|z_B+(N\ell,0)+t_B\nabla\Phi(\eta_\theta,\xi_\theta)|
\lesssim r^{1/2+\varepsilon_0}
\end{equation*}
for some $|\ell|\leq C'r/N$.
The argument for \eqref{eq:K-M1}, now in dimension $m+n$, gives
$O(r^{2(m+n)\varepsilon_0})$ choices of $(\theta,v)$ for each fixed $\ell$.
As there are $O((r/N)^m)$ possible $\ell$, we have
\begin{equation}\label{eq:K-euclidean-M}
M\lesssim\left(\frac rN\right)^m r^{2(m+n)\varepsilon_0},\qquad
\bar\beta^2M\lesssim\frac{N^m}{r^n}r^{2(m+n)\varepsilon_0}.
\end{equation}
Moreover, by \eqref{eq:packet-L2} with no periodic variables and Plancherel's identity,
for $J>m$,
\begin{equation}\label{eq:K-euclidean-L2}
\begin{aligned}
\sum_{\iota\in\mathcal W}\|h_\iota^{(r,t_0)}\|_{L^2}^2
&\lesssim\|h\|_{L^2}^2\\
&\lesssim N^{2m} \sum_{\substack{|\ell|\leq C'r/N\\|\ell'|\leq C'r/N}}\frac{1}{(1+N|\ell-\ell'|)^J}\\
&\lesssim N^{2m}\left(\frac rN\right)^m.
\end{aligned}
\end{equation}
Applying Theorem \ref{thm:refined-decoupling} with $m=0$, $n$ replaced by $m+n$, and loss $\varepsilon/2$, 
and summing over the bounded covering of $\Omega$, we obtain
\begin{align*}
\int_\Omega|E_{m+n}\widehat h(x,y,t)|^{p_\ast}\rmd x\rmd y\rmd t
&\lesssim r^{1+\varepsilon/2}(\bar\beta^2M)^{2/(m+n)} \sum_{\iota\in\mathcal W}\|h_\iota^{(r,t_0)}\|_{L^2}^2+\frac1{N^A}\\
&\lesssim r^{1+\varepsilon/2+4\varepsilon_0} N^{2m}\left(\frac rN\right)^m \left(\frac{N^m}{r^n}\right)^{2/(m+n)}+\frac1{N^A}.
\end{align*}
The remainder has arbitrarily rapid decay by \eqref{eq:K-euclidean-L2} and
$M_{T,r}(h)\lesssim\|h\|_{L^\infty}\lesssim N^m$.
Combining this with \eqref{eq:K-periodic-average} and \eqref{eq:K-poisson}, we obtain
\begin{align*}
&\int_{\mathbb T_N^m\times B_r(y_0,t_0)}
|\mathcal E_{m,n}(t)f(x,y)|^{p_\ast}\rmd x\rmd y\rmd t\\
&\qquad\lesssim r^{1+\varepsilon/2+4\varepsilon_0}N^{2m}\left(\frac{N^m}{r^n}\right)^{2/(m+n)}+\frac1{N^A}.
\end{align*}
Taking $\varepsilon_0$ sufficiently small and applying Lemma \ref{lem:Ftau-limit} proves \eqref{eq:K-estimate} for $1<a_0<2$, completing the proof.

\bibliographystyle{abbrv}
\bibliography{references}

\end{document}